\documentclass[12pt]{amsart}
\usepackage{amssymb}
\usepackage{amsbsy}
\usepackage{amscd}
\usepackage[mathscr]{eucal}
\usepackage{verbatim}
\usepackage{version}

\usepackage{nicefrac}

\usepackage{pstricks}
\usepackage{pst-all}
\usepackage{color}

\usepackage[all]{xy}
\usepackage{mathdots}
\makeatletter
\def\cal{\mathcal}
\def\Bbb{\mathbb}

\newenvironment{pf*}[1]{\proof[#1]}{\endproof}
\makeatother

\makeatletter
\makeatletter
\@ifclasslater{amsart}{1999/11/24}{}{
\renewcommand*\subjclass[2][1991]{%
  \def\@subjclass{#2}%
  \@ifundefined{subjclassname@#1}{%
    \ClassWarning{\@classname}{Unknown edition (#1) of Mathematics
      Subject Classification; using '1991'.}%
  }{%
    \@xp\let\@xp\subjclassname\csname subjclassname@#1\endcsname
  }%
}
\renewcommand{\subjclassname}{%
  \textup{1991} Mathematics Subject Classification}
\@xp\let\csname subjclassname@1991\endcsname \subjclassname
\@namedef{subjclassname@2000}{%
  \textup{2000} Mathematics Subject Classification}
}

\renewcommand{\MR}[1]{}

\makeatother
\newenvironment{NB}{
\color{red}{\bf NB}. \footnotesize 
}{}
\excludeversion{NB}
\newenvironment{NB2}{
\color{blue}{\bf NB}. \footnotesize
}{}

\newtheorem{Theorem}[equation]{Theorem}

\newtheorem{Proposition}[equation]{Proposition}

\theoremstyle{definition}
\newtheorem{Definition}[equation]{Definition}
\newtheorem{Example}[equation]{Example}

\newtheorem{Conjecture}[equation]{Conjecture}

\theoremstyle{remark}

\newtheorem*{Claim}{Claim}
\newtheorem{Remarks}[equation]{Remarks}

\numberwithin{equation}{section}
\newcommand{\thmref}[1]{Theorem~\ref{#1}}
\newcommand{\secref}[1]{\S\ref{#1}}

\newcommand{\propref}[1]{Proposition~\ref{#1}}

\newcommand{\subsecref}[1]{\S\ref{#1}}

\newcommand{\defeq}{:=}
\newcommand{\C}{{\Bbb C}}
\newcommand{\Z}{{\Bbb Z}}
\newcommand{\Q}{{\Bbb Q}}
\newcommand{\R}{{\Bbb R}}
\newcommand{\proj}{{\Bbb P}}
\newcommand{\CP}{\proj}

\newcommand{\SL}{\operatorname{\rm SL}}

\newcommand{\GL}{\operatorname{GL}}

\newcommand{\U}{\operatorname{\rm U}}
\newcommand{\SO}{\operatorname{\rm SO}}

\newcommand{\Ext}{\operatorname{Ext}}

\newcommand{\Ima}{\operatorname{Im}}

\newcommand{\rank}{\operatorname{rank}}

\newcommand{\pd}[2]{\frac{\partial#1}{\partial#2}}

\newcommand{\ve}{\varepsilon}

\newcommand{\linf}{{\ell_\infty}}
\newcommand{\shfO}{\mathcal O}
\newcommand{\bp}{{{\widehat\proj}^2}}

\newcommand{\bM}{{\widehat M}}

\newcommand{\ch}{\operatorname{ch}}

\newcommand{\Todd}{\operatorname{Todd}}

\newcommand{\Zin}{Z^{\text{\rm inst}}}

\newcommand{\bZin}{\widehat{Z}^{\text{\rm inst}}}
\newcommand{\bZ}{\widehat{Z}}

\newcommand{\Fin}{F^{\text{\rm inst}}}
\newcommand{\Ain}{A^{\text{\rm inst}}}
\newcommand{\Bin}{B^{\text{\rm inst}}}

\newcommand{\hT}{\widetilde T}

\newcommand{\res}{\mathop{\text{\rm res}}}

\newcommand{\Res}{\operatornamewithlimits{Res}}

\makeatletter
\newcommand{\vechatom}{
    {\Vec{\omega}}
    \,\smash[b]{\hbox{\lower2\ex@\hbox{$\m@th\hat{\null}$}}}
}
\makeatother

\newcommand{\pt}{\operatorname{pt}}

\newcommand{\cV}{\mathcal V}
\newcommand{\Vcal}{\cV}
\newcommand{\cE}{\mathcal E}

\newcommand{\Eu}{\operatorname{Eu}}

\newcommand{\bgamma}{{\boldsymbol\gamma}}
\newcommand{\Nc}{r}

\newcommand{\T}{T} % the contact term
\newcommand{\fT}{\mathfrak T} % the contact term
\newcommand{\SW}{\operatorname{SW}} % the Seiberg-Witten invariant
\newcommand{\fs}{\mathfrak s}

\usepackage[dvips,bookmarks=false]{hyperref}

\begin{document}
\title[Donaldson = Seiberg-Witten]
{Donaldson = Seiberg-Witten
\\
from Mochizuki's formula and instanton counting
\\
% {\rm Preliminary Version (\today)  
% }
}
\author{Lothar G\"ottsche}
\address{International Centre for Theoretical Physics, Strada Costiera 11, 
34151 Trieste, Italy}
\email{gottsche@ictp.it}

\author{Hiraku Nakajima}
\address{Research Institute for Mathematical Sciences,
Kyoto University, Kyoto 606-8502,
Japan}
\email{nakajima@kurims.kyoto-u.ac.jp}
\thanks{The second named author is partly supported by the
  Grant-in-Aid for Scientific Research (B) (No.~19340006), Japan
  Society for the Promotion of Science.}

\author{K\={o}ta Yoshioka}
\address{Department of Mathematics, Faculty of Science, Kobe University,
Kobe 657-8501, Japan}
\address{Max-Planck-Institut f\"{u}r Mathematik,
Vivatsgasse 7
53111 Bonn, Germany}
\email{yoshioka@math.kobe-u.ac.jp}
\thanks{The third named author is supported by the Grant-in-Aid for
  Scientific Research (B) (No.~18340010), Japan Society for the
  Promotion of Science, and Max Planck Institute for Mathematics.}

\begin{abstract}
  We propose an explicit formula connecting Donaldson invariants and
  Seiberg-Witten invariants of a $4$-manifold of simple type via
  Nekrasov's deformed partition function for the $N=2$ SUSY gauge
  theory with a single fundamental matter. This formula is derived
  from Mochizuki's formula, which makes sense and was proved when the
  $4$-manifold is complex projective. Assuming our formula is true for
  a $4$-manifold of simple type, we prove Witten's conjecture and sum
  rules for Seiberg-Witten invariants (superconformal simple type
  condition), conjectured by Mari\~no, Moore and Peradze.
\end{abstract}

\maketitle
\tableofcontents

\section{Introduction}

Let $X$ be a smooth, compact, connected, and oriented $4$-manifold
with $b_1 = 0$ and $b_+ \ge 3$ odd. We set
\begin{equation*}
  (K_X^2) \defeq 2\chi(X) + 3\sigma(X), \qquad
  \chi_h(X) \defeq \frac{\chi(X)+\sigma(X)}4.
\end{equation*}
When $X$ is a complex projective surface, these are the
self-intersection of the canonical bundle and the holomorphic Euler
characteristic respectively, and our notation is consistent.

Let $\xi\in H^2(X,\Z)$, $\alpha\in H_2(X)$ and $p\in H_0(X)$ be the
point class. In \cite{Witten} Witten explained that the generating
function ${\mathscr D}^{\xi}(\alpha)$ of Donaldson invariants (see
\eqref{eq:Donseries} for the definition) is related to Seiberg-Witten
invariants by
\begin{equation}\label{eq:Witten}
  \begin{split}
    {\mathscr D}^{\xi}(\alpha)
  &\defeq \sum_{n,k} \frac1{k!} \left(D^{\xi,n}(\alpha^k)
  + \frac12 D^{\xi,n}(\alpha^kp)\right)
\\
  &= 2^{(K_X^2) - \chi_h(X) + 2} (-1)^{\chi_h(X)}
  e^{(\alpha^2)/2}
  \sum_{\fs} \SW(\fs) (-1)^{(\xi,\xi+c_1(\fs))/2}
  e^{(c_1(\fs),\alpha)},
  \end{split}
\end{equation}
where $(\ ,\ )$ is the intersection form, $(\alpha^2) =
(\alpha,\alpha)$, $\SW(\fs)$ is the Seiberg-Witten invariant of a
spin${}^c$ structure $\fs$, and $c_1(\fs) = c_1(S^+)\in H^2(X,\Z)$ is
the first Chern class of the spinor bundle of $\fs$. And $X$ is
assumed to be of {\it SW-simple type}, i.e., $c_1(\fs)^2 = (K_X^2)$ if
$\SW(\fs)\neq 0$.

Witten's argument was based on Seiberg-Witten's ansatz \cite{SW1} of
$N=2$ SUSY gauge theory, which is a physical theory underlying
Donaldson invariants \cite{Wit}. It was not given in a way which
mathematicians can justify, so \eqref{eq:Witten} becomes Witten's
conjecture among mathematicians.

Let us explain the main point of Witten's argument. (See
\cite[Introduction]{NY2} for a more detailed exposition for
mathematicians.)
Seiberg-Witten's ansatz roughly says that the $N=2$ SUSY gauge theory
is controlled by a family of elliptic curves (called Seiberg-Witten
curves)
\begin{equation*}
  y^2 = 4x(x^2 + ux + \Lambda^4)
\end{equation*}
parametrized by $u\in\C$. Here $\Lambda$ is a formal variable used to
count the dimension of instanton moduli spaces in the prepotential of
the theory. (In the Donaldson series, one usually sets $\Lambda =
1$.)  Witten explained that ${\mathscr D}^{\xi}(\alpha)$ is given by
an integration over $u\in \C$, and the integrand is supported only at
points $u = \pm 2\Lambda^2$, where the corresponding elliptic curve is
singular, when $b_+ \ge 3$. Those points contribute as given on the
right hand side of \eqref{eq:Witten}. 

In mathematics, the Seiberg-Witten curves appear as elliptic curves
for the $\sigma$-function in Fintushel-Stern's blow-up formula
\cite{FS} for Donaldson invariants, and the parameter $u$ corresponds
to the point class $p$. However, no mathematician succeeded to make
Witten's argument rigorous.

An alternative mathematically rigorous approach was proposed by
Pidstrigach and Tyurin \cite{PT}, and further pursued by Feehan-Leness
\cite{FL}. It is based on moduli spaces of $\SO(3)$-monopoles, which
are a higher rank analog of $\U(1)$-monopoles used to define
Seiberg-Witten invariants. In particular, under a certain technical
assumption on a property of $\SO(3)$-monopole moduli spaces,
Feehan-Leness \cite{FL} (see also \cite{FL1,FL2}, in particular
\cite[Th.~3.1]{FLA}) showed that Donaldson invariants have the form
\begin{equation}\label{eq:FL}
  D^{\xi,n}(\alpha^k p^l) =
  \sum_{\fs} f_{k,l}(\chi_h(X),(K_X^2),\fs,\xi,\alpha,\fs_0) \SW(\fs),
\end{equation}
where the coefficients $f_{k,l}$ are not explicit, but depend only
on the $\chi_h(X)$, $(K_X^2)$ and various intersection products among
$\fs$, $\xi$, $\alpha$, $\fs_0$. Here $\fs_0$ is an auxiliary
spin${}^c$ structure needed for $\SO(3)$-monopole moduli spaces. As an
application, they proved Witten's conjecture for $X$ which satisfies
$(K_X^2) \ge \chi_h(X) - 3$ or is abundant, i.e., the orthogonal
complement of Seiberg-Witten classes contains a hyperbolic sublattice
\cite{FLA}.

For a complex projective surface $X$, Mochizuki, motivated by
\cite{PT,FL} (and also by \cite{Don,GP}), proved a formula expressing
Donaldson invariants in the form \eqref{eq:FL}, but the coefficients
$f_{k,l}$ are given as residue of an explicit $\C^*$-equivariant
integral over the product of Hilbert schemes of points on $X$ (see
\thmref{thm:Mochizuki}). He obtained the formula by applying the
Atiyah-Bott-Lefschetz fixed point formula to the algebro-geometric
counterpart of $\SO(3)$-monopole moduli spaces.

Our first main result (\thmref{thm:partition}) says that Mochizuki's
coefficients are given by leading terms, denoted by $F_0$, $H$, $A$,
$B$ of Nekrasov's deformed partition function for the $N=2$ SUSY gauge
theory with a single fundamental matter, which is the physics
counterpart of the $\SO(3)$-monopole theory. Thus the coefficients are
`equivariant $\SO(3)$-monopole invariants for $\R^4$' in some sense.

The proof is almost the same as that of the authors' wall-crossing
formula of Donaldson invariants with $b_+=1$, expressed in terms of
Nekrasov's partition function for the pure gauge theory \cite{GNY}: By
a cobordism argument (due to Ellingsrud-G\"ottsche-Lehn \cite{EGL}),
it is enough to show it for toric surfaces. Then the integral is given
as the product of local contributions from torus fixed points of $X$,
and the local contribution can be considered as the case
$X=\R^4$. Thus it is, by its definition, Nekrasov's partition
function.

From this result, we see that Mochizuki's coefficients depend on the
various data in the same way as those of Feehan-Leness. In particular,
they make sense also for a smooth $4$-manifold $X$. (Here $\fs_0$ is
given by the complex structure.) Hoping that Mochizuki's coefficients
are the same as those of Feehan-Leness\footnote{See
  \subsecref{subsec:partition} for our heuristic proof of this
  hope. This paper is motivated by Feehan-Leness' papers, but the
  proof is independent.}, we propose a conjecture: our formula remains
true for a smooth $4$-manifold of SW-simple type
(Conjecture~\ref{con:main}).

Nekrasov's partition functions are defined in a mathematically
rigorous way and have explicit combinatorial expressions in terms of
Young diagrams \cite{Nek}. Furthermore, the leading part $F_0$ is
given by certain period integrals over Seiberg-Witten curves
\cite{NY1,NO,BE}, $H$ is explicit, and $A$, $B$ are also given in
terms of Seiberg-Witten curves \cite{NY2}. The proofs in
\cite{NY1,NY2} were given only for the pure theory, but we extend them
for the theory with one matter in this paper using the theory of
perverse coherent sheaves \cite{perv3}. Thus Mochizuki's coefficients
are now given by residue of a differential form expressed by
Seiberg-Witten curves.

The pole, at which we take the residue, is at $u=\infty$. It is very deep
pole, and a direct computation of the residue looks difficult.
Fortunately there is a hint: a similar problem, for certain limits of
Donaldson invariants with $b_+ = 1$, was analyzed by G\"ottsche-Zagier
\cite{GZ}. Observing that their differential is defined on $\proj^1$,
holomorphic outside $\infty$, $\pm 2\Lambda^2$, they showed that it is
enough to compute the residues at poles $\pm 2\Lambda^2$ which are
simple, and proved an analog of Witten's conjecture. Also this picture
is close to Witten's original intuition\footnote{G.~Moore pointed out
  us that the $u$-plane integrand is a total derivative, at least if
  we take a derivative with respect to the metric. See
  \cite[(11.16)]{MW}}.

Let us emphasize that the extension of the differential to $\proj^1$
is already a nontrivial assertion. In the original formulation the
parameter $u$ was a {\it formal\/} variable used to introduce a
generating function of invariants. Therefore it is, a priori, defined
only in the formal neighborhood of $u=\infty$.
The extension is done, so far, by an explicit formula of the
differential. Thus the geometric picture of moduli spaces becomes
obscure at the points $\pm 2\Lambda^2$.

Our situation is similar to one in \cite{GZ}, but slightly
different. The Seiberg-Witten curve for the theory with a fundamental
matter is
\begin{equation*}
  y^2 = 4x^2(x+u) + 4m\Lambda^3 x + \Lambda^6,
\end{equation*}
and has one more parameter $m$, called the {\it mass\/} of the matter
field. And in our formula, this $m$ is chosen so that the above curve
is singular. Therefore the family of curves is different from what
Witten used. We have two features of the new family. First since the
curves are singular, the differential is written by elementary
functions, not by modular functions as in \cite{GZ}. This makes our
computation much easier. Second, more importantly, we get another pole
besides $\infty$, $\pm 2\Lambda^2$, which is called the {\it
  superconformal point\/} in physics literature. (In the main text, we
change the variable from $u$ to another variable $\phi$ called the
{\it contact term}.)

The contribution of this point to the gauge theory with one matter was
studied by Mari\~no, Moore and Peradze \cite{MMP} at a physical level
of rigor. They argued that the partition function must be regular at
the superconformal point and then this condition leads to sum rules on
Seiberg-Witten invariants, i.e., $X$ must satisfy the following
condition.

\begin{Definition}[\cite{MMP}]\label{def:scs}
  Suppose that a $4$-manifold $X$ is of SW-simple type. We say $X$ is
  of {\it superconformal simple type\/} if $(K_X^2) \ge \chi_h(X) - 3$ or
  \begin{equation}\label{eq:scs}
    \sum_{\fs} (-1)^{(\tilde w_2(X),\tilde w_2(X)+c_1(\fs))/2} 
    \SW(\fs) (c_1(\fs),\alpha)^n = 0
  \end{equation}
  for any integral lift $\tilde w_2(X)$ of $w_2(X)$ and $0\le n\le
  \chi_h(X) - (K_X^2) - 4$.
\end{Definition}

Remark that $(K_X^2) \ge \chi_h(X) - 3$ is the condition which
Feehan-Leness \cite{FLA} assumed to prove Witten's conjecture. It
should be remarked that they also proved that $X$ is of superconformal
simple type if $X$ is abundant under the same technical assumption as
before \cite{FKLM,FLA}.

We analyze the residue of our differential at the superconformal point
and show that 1) the fact that ${\mathscr D}^\xi(\alpha)$, up to sign,
depends only on $(\xi\bmod 2)$ implies that $X$ is of superconformal
simple type, and 2) the differential is regular at the superconformal
point if $X$ is of superconformal simple type. Thus the residue
vanishes at the superconformal point, and hence we prove Witten's
conjecture for a $4$-manifold $X$ of simple type under
Conjecture~\ref{con:main}, and under no assumption for a complex
projective surface $X$.

\subsection*{Acknowledgments}
The authors thank Takuro Mochizuki for explanations of his results and
discussion over years.
The second-named author is grateful to Yuji Tachikawa for discussions
on Seiberg-Witten curves, and to Thomas Leness and Gregory Moore for
useful comments.

The project began when the second and third-named authors were
visiting the International Centre for Theoretical Physics in Aug.~2006.
Part of the calculation was done while the second-named author was
visiting Mathematical Institute of the University of Bonn, and the
third-named author was at Max Planck Institute for Mathematics. The
authors thank all the institutes for the hospitality.

\section{Preliminaries~(I) -- Donaldson and Seiberg-Witten invariants}

\subsection{Donaldson invariants}\label{subsec:Donaldson}

Let $y = (2,\xi,n)\in H^{\operatorname{even}}(X,\Z)$.
We take a Riemannian metric $g$ on $X$ and consider the moduli space
$M(y)$ of irreducible anti-self-dual connections on the adjoint bundle
$\operatorname{ad}(P)$ of a principal $U(2)$-bundle $P$ with $c_1(P) =
\xi$, $c_2(P) = n$.
For a generic metric $g$, this is a manifold of dimension $8n - 2(\xi^2)
- 6\chi_h(X)$. A choice of an orientation of $H^+$, a maximal positive
definite subspace of $H^2(X)$ with respect to the intersection
pairing, gives an orientation on $M(y)$.

Let $\mathcal P\to X\times M(y)$ be a universal $PU(2)$-bundle and
let $\mu\colon H_i(X)\to H^{4-i}(M(y))$ be the $\mu$-map defined
by $\mu(\beta)\defeq -\frac14 p_1(\mathcal P)/\beta$. Then the
Donaldson invariant of $X$ is a polynomial on $H_0(X)\oplus H_2(X)$
defined by
\begin{equation}\label{eq:Don}
  D^{\xi,n}(\alpha^k p^l) = \int_{M(y)} \mu(\alpha)^k \mu(p)^l,
\end{equation}
where $p\in H_0(X)$ is the point class. This is nonzero only when
$k + 2l = 4n - (\xi^2) - 3\chi_h(X)$.
As $M(y)$ is not compact, this integral must be justified by using
the Uhlenbeck compactification of $M(y)$.
When $b_+ \ge 3$ as we assumed, the integral is independent of the
choice of the Riemannian metric $g$.
The moduli space does not change by a twisting of $P$ by a line
bundle, since the adjoint bundle remains the same. Only the orientation
is different.
Thus the integral depends only on $\xi\bmod 2\in H^2(X,\Z/2)$ up to
sign.

We consider the generating function
\begin{equation*}
  D^{\xi}(\exp (\alpha z + px))
  = \sum_{n,k,l} D^{\xi,n}(\alpha^k p^l)\frac{z^k x^l}{k!\,l!}
  \Lambda^{4n-(\xi^2)-3\chi_h(X)}.
\end{equation*}
Since $n$ can be read off from $k$, $l$ as above, the variable
$\Lambda$ is redundant, and we often put $\Lambda = 1$, but it is also
useful when we will consider the partition function.

\begin{Definition}\label{def:KMsimple}
  A $4$-manifold $X$ is of {\it KM-simple type\/} if for any $\xi$ and
  $\alpha$,
  \begin{equation*}
    \frac{\partial^2}{\partial x^2}  D^{\xi}
    = 4\Lambda^4 D^{\xi}.
  \end{equation*}
\end{Definition}

For a $4$-manifold of KM-simple type, we define
\begin{equation}\label{eq:Donseries}
  {\mathscr D}^{\xi}(\alpha)
  \defeq D^{\xi}(\exp(\alpha)(1 + \frac12 p))
  = \sum_{n,k}
   D^{\xi,n}(\alpha^k)\frac1{k!} %\Lambda^{4n-(\xi^2)-3\chi_h(X)}
   + 
   \frac12 \sum_{n,k}
   D^{\xi,n}(\alpha^k p)\frac1{k!} %\Lambda^{4n-(\xi^2)-3\chi_h(X)}
   .
\end{equation}

Kronheimer-Mrowka's structure theorem \cite{KM} says that there is a
finite distinguished collection of $2$-dimensional cohomology classes
$K_i\in H^2(X,\Z)$ and nonzero rational numbers $\beta_i$ such that
\begin{equation*}
  {\mathscr D}^{\xi}(\alpha)
  = \exp((\alpha^2)/2) 
  \sum_i (-1)^{(\xi,\xi+K_i)/2} \beta_i \exp(K_i,\alpha).
\end{equation*}
Each $K_i$ is an integral lift of the second Stiefel-Whitney class
$w_2(X)$.

\subsection{Complex projective surfaces}\label{subsec:complex}

Now suppose $X$ is a complex projective surface. Take an ample line
bundle $H$ and consider the moduli space $M_H(y)$ of torsion free
$H$-semistable sheaves $E$ with $c_1(E) = \xi$, $c_2(E) = n$. Here we
assume $\xi$ is of type $(1,1)$. We take the orientation on $H^+$
given by $c_1(H)$ and the complex orientation on $H^{0,2}(X)$.

It is known that Donaldson invariants can be defined using $M_H(y)$
instead of $M(y)$ in \eqref{eq:Don} if $M_H(y)$ is of expected
dimension \cite{Li,Morgan}. We define the $\mu$-map by using a
universal sheaf $\cE$ instead of $\mathcal P$, as $\mu(\beta) =
(c_2(\cE) - c_1(\cE)^2/4)/\beta$. The orientation we used above is
differed from the complex orientation by $(-1)^{(\xi,\xi+K_X)/2}$,
where $K_X$ is the canonical class.

If $M_H(y)$ is not of expected dimension, we consider the blow-up at
sufficiently many points $p_1$, \dots, $p_N$ disjoint from cycles
representing $\alpha$, $p$. Then the moduli becomes of expected
dimension on the blow-up if $N$ is sufficiently large. We then use the
blow-up formula as the definition of the integral over $M_H(y)$. See
\cite[\S1.1]{GNY} for detail.

Mochizuki defines the invariants by using the obstruction theory on
the moduli spaces of pairs of sheaves and their sections with a
suitable stability condition. When the vector $y$ is primitive, the
stability is equivalent to the semistability for $M_H(y)$, and
Mochizuki's moduli is a projective bundle over $M_H(y)$. If, furthermore,
the moduli space $M_H(y)$ is of expected dimension, the virtual
fundamental class coincides with the ordinary one, and hence
Mochizuki's invariants are equal to the usual Donaldson invariants
(\cite[Lem.~7.3.5]{Moc}). In order to prove that his invariant
coincides with the above invariant for any $y$, one needs to prove the
blow-up formula for Mochizuki's. It follows a posteriori from our main
result that this is true. It should be possible to give  a more direct
proof by combining the theory of perverse coherent sheaves
\cite{perv1,perv2,perv3} with Mochizuki's method.

\subsection{Seiberg-Witten invariants}

Let $\fs$ be a spin${}^c$ structure and let $c_1(\fs) = c_1(S^+)\in
H^2(X)$ be the first Chern class of its spinor bundle.
\begin{NB}
  If $\fs$ is induced from the complex structure, we have $S^+
  = \Lambda^{0,0}\oplus \Lambda^{0,2}$. Hence $c_1(S^+) = -
  c_1(K_X)$. Witten \cite{Witten} defined the class $x$ as
  $-c_1(S^+)$.
\end{NB}

Let $N(\fs)$ be the moduli space of the solutions of monopole
equations. This is a compact manifold (more precisely, after a
perturbation) of dimension $d(\fs) \defeq (c_1(\fs)^2
- (K_X)^2)/4$.
It has the orientation induced from that of $H^+$ as in the case of
Donaldson invariants.
Let $\mathcal Q$ be the $S^1$-bundle associated with the evaluation
homomorphism from the gauge group at a point in $X$, and $c_1(\mathcal
Q)$ be its first Chern class. The Seiberg-Witten invariant of
$\fs$ is defined as
\begin{equation*}
  \SW(\fs) \defeq \int_{N(\fs)}
  c_1(\fs)^{d(\fs)/2}
\end{equation*}
This is independent of the choice of $g$ and the perturbation.

We call $\fs$ (or $c_1(\fs)$) a Seiberg-Witten class if
$\SW(\fs)\neq 0$. It is known that there are only finitely
many Seiberg-Witten classes.

\begin{Definition}
  A $4$-manifold $X$ is of {\it SW-simple type\/} if $\SW(\mathfrak
  s)$ is zero for all $\fs$ with $d(\fs) > 0$.
\end{Definition}

For $c\in H^2(X;\Z)$ which is a lift of $w_2(X)$, we define $\SW(c)$
as the sum
\begin{equation*}
  \SW(c) = \sum_{c_1(\fs) = c} \SW(\fs).
\end{equation*}

When $X$ is a complex projective surface, it is known that all
Seiberg-Witten classes are of type (1,1). The moduli space $N(\fs)$ is
identified with the moduli space of pairs of a holomorphic line bundle
and its section. It is an unperturbed moduli space, and does not have
the expected dimension $d(\fs)$ in general, but can be equipped with an
obstruction theory to define the invariants \cite{F-M}. It is also
known that $X$ is of SW-simple type.

We will not use so much on results on Seiberg-Witten invariants,
except the most basic one:
\begin{equation*}
  \SW(-\fs) = (-1)^{\chi_h(X)} \SW(\fs),
\end{equation*}
where $-\fs$ is the complex conjugate of the spin${}^c$ structure
$\fs$. (See e.g., \cite[Cor.~6.8.4]{M-book}.)

\subsection{Witten's conjecture}

Witten's conjecture states that if $X$ is of SW-simple type, it is
also of KM-simple type and $\beta_i$, $K_i$ are determined by
Seiberg-Witten invariants. See \eqref{eq:Witten} in Introduction.

%\subsection{Examples}
\begin{NB}
This is Witten's:
\begin{equation*}
    {\mathscr D}^{\xi}(\alpha)
  = 2^{(K_X^2) - \chi_h(X) + 2} (-1)^{\chi_h(X)}
  \exp((\alpha^2)/2) 
  \sum_{\fs} \SW(\fs) (-1)^{(\xi,\xi+c_1(\fs))/2}
  \exp(c_1(\fs),\alpha).
\end{equation*}

This is a copy of (5.94).
  \begin{equation*}
    \sum_{\xi_1}
(-1)^{\frac{(\xi,\xi+K_X)}{2}+\chi({\cal O}_X)}
\Bigg(
\widetilde{SW}(c_1(\fs))
(-1)^{\frac{(\xi,\xi+c_1(\fs))}{2}}
2^{(K_X^2)-\chi({\cal O}_X)+1}
\exp\left(\frac{1}{2}
(\alpha^2) z^2 +(c_1(\fs),\alpha) z \right)
\Bigg).
  \end{equation*}

  The factor $(-1)^{\frac{(\xi,\xi+K_X)}{2}}$ comes from the sign
    convention. The factor $2$ comes from Mochizuki's definition. (In
    general, his invariant is $1/\rank$ of the usual one.)
\end{NB}

\begin{Example}
  Let $X$ be a $K3$ surface. The Donaldson series is known \cite{Og}:
  \begin{equation*}
    {\mathscr{D}}^{\xi}(\alpha) = (-1)^{\frac{(\xi,\xi)}2}
    \exp ((\alpha^2)/2).
  \end{equation*}
  The only Seiberg-Witten class is $c_1(\fs) = 0$ and
  $\SW(\fs) = 1$.
\begin{NB}
  This is compatible with (5.94) up to multiplication by $2$, and the
  factor $(-1)^{\nicefrac{(\xi,\xi)}2}$, which comes from the
  orientation convention. We have $\chi_h = 2$, $(K_X^2) = 0$.
\end{NB}
\end{Example}

\begin{Example}
  Let $X$ be a quintic surface in $\CP^3$. The Donaldson series was given in
  \cite[Example~2]{KM}:
  \begin{equation*}
    \begin{split}
    {\mathscr{D}}^{0}(\alpha) &= 8 \exp ((\alpha^2)/2) \sinh(K_X,\alpha),
\\
    {\mathscr{D}}^{K_X}(\alpha) &= -8 \exp ((\alpha^2)/2) \cosh(K_X,\alpha).
    \end{split}
  \end{equation*}
  We have $\chi_h = 5$, $(K_X^2) = 5$. The Seiberg-Witten classes are
  $\pm K_X$, and $\SW(-K_X) = 1$, $\SW(K_X) = (-1)^{\chi_h} = -1$ by
  \cite[Prop.~7.3.1]{M-book}.
\begin{NB}
  This is compatible with (5.94) up to multiplication by $2$: In the
  first case $\xi=0$ we have
  \begin{multline*}
    (-1)^{\chi({\cal O}_X)}
    2^{(K_X^2)-\chi({\cal O}_X)+1}
\Biggl(
{\SW}(K_X)
\exp\left(\frac{1}{2}
(\alpha^2) z^2 +(K_X,\alpha) z \right)
%\Bigg)
\\
 + {\SW}(-K_X)
%2^{(K_X^2)-\chi({\cal O}_X)+1}
\exp\left(\frac{1}{2}
(\alpha^2) z^2 -(K_X,\alpha) z \right)
\Biggr).
  \end{multline*}
In the second case $\xi=K_X$, we have
\begin{multline*}
%      \sum_{\xi_1}
(-1)^{(K_X^2)+\chi({\cal O}_X)}
2^{(K_X^2)-\chi({\cal O}_X)+1}
\Biggl(
{\SW}(K_X)
(-1)^{(K_X^2)}
\exp\left(\frac{1}{2}
(\alpha^2) z^2 +(K_X,\alpha) z \right)
\\
+ {\SW}(-K_X)
\exp\left(\frac{1}{2}
(\alpha^2) z^2 - (K_X,\alpha) z \right)
\Biggr).
\end{multline*}
\end{NB}
\end{Example}

\begin{Example}\label{ex:elliptic}
  Let $X$ be an elliptic surface $X$ without multiple fibers such that
  $H^1(X,\shfO_X) = 0$. Let $f$ be the class of a fiber. We have $K_X
  = \shfO_X(df)$ with $\chi_h(X) = d+2$ and $(K_X^2) = 0$. The
  Donaldson series is given by Fintushel-Stern \cite{FS2}:
  \begin{equation*}
    {\mathscr{D}}^{0}(\alpha) = \exp ((\alpha^2)/2) \sinh^{\chi_h(X)-2}(f,\alpha).
  \end{equation*}
  The Seiberg-Witten invariants were computed by Friedman-Morgan
  \cite{F-M}:
\begin{equation*}
    \SW((2p-d)f) = (-1)^{p}\binom{d}p\quad\text{for $p=0,\dots,d$},
    \qquad
    \SW(c) = 0 \quad \text{for other $c$}.
\end{equation*}
\begin{NB}
  $\SW(-K_X) = \SW(-df) = 1$, $\SW(K_X) = \SW(df) = (-1)^d =
  (-1)^{\chi_h(X)}$. It seems that the convention is compatible with
  other parts.
\end{NB}%
\begin{NB}
  On the other hand, (5.94) gives
  \begin{equation*}
    (-1)^{d} \exp((\alpha^2)/2) 2^{-d-1}
    \sum_{p}(-1)^p \binom{d}{p} \exp((2p-d)(f,\alpha))
    \begin{NB2}
      = (-1)^{d} \exp((\alpha^2)/2) 2^{-d-1} e^{-d(f,\alpha)}
      (1 - e^{2(f,\alpha)})^d
    \end{NB2}
    = \frac12 \exp((\alpha^2)/2) \sinh^{\chi_h(X)-2}(f,\alpha).
  \end{equation*}
  The factor $(-1)^{(K_X^2)} = -1$ is the orientation convention. If
  we remove it, we get $-4 \exp((\alpha^2)/2)\cosh(K_X,\alpha)$.
\end{NB}
\end{Example}

\subsection{Superconformal simple type}

Let us briefly study the superconformal simple type condition
(Definition~\ref{def:scs}) in this subsection. More examples can be
found in \cite{MMP}.

If we take another integral lift $\tilde w_2'(X)$ of $w_2(X)$ in
\eqref{eq:scs}, we have $(-1)^{(\tilde w_2'(X),\tilde
  w_2'(X)+c_1(\fs))/2} = (-1)^{(\tilde w_2(X),\tilde
  w_2(X)+c_1(\fs))}(-1)^{((\tilde w_2'(X)-\tilde w_2(X))/2)^2}$.
\begin{NB}
\begin{equation*}
  \begin{split}
  & (-1)^{(\tilde w_2'(X),\tilde w_2'(X)+c_1(\fs))/2} 
= (-1)^{(\tilde w_2'(X)-\tilde w_2(X)+\tilde w_2(X),\tilde w_2'(X)-\tilde
  w_2(X)+\tilde w_2(X) + c_1(\fs))/2} 
\\
=\; & (-1)^{(\tilde w_2(X),\tilde w_2(X)+c_1(\fs))/2} 
  (-1)^{(\tilde w_2'(X)-\tilde w_2(X),\tilde w_2'(X)-\tilde w_2(X))/2}
  (-1)^{(\tilde w_2'(X)-\tilde w_2(X),\tilde w_2(X))} 
  (-1)^{(\tilde w_2'(X)-\tilde w_2(X),c_1(\fs))/2}.
  \end{split}
\end{equation*}
The last expression can be further modified as
\begin{equation*}
  (-1)^{(\tilde w_2'(X)-\tilde w_2(X),c_1(\fs))/2} 
  = (-1)^{((\tilde w_2'(X)-\tilde w_2(X))/2, c_1(\fs))}
  = (-1)^{((\tilde w_2'(X)-\tilde w_2(X))/2, w_2(X))}
  = (-1)^{((\tilde w_2'(X)-\tilde w_2(X))/2)^2}.
\end{equation*}
\end{NB}%
Therefore it is enough to assume \eqref{eq:scs} for {\it some\/}
integral lift $\tilde w_2(X)$ of $w_2(X)$.
We will consider the case when $X$ is a complex projective surface,
and take $K_X$ as a lift.

If $X$ is a minimal surface of general type, we have the Noether's
inequality $(K_X^2)/2 + 2 \ge
\begin{NB}
  p_g(X) =   
\end{NB}%
\chi_h(X) - 1$. Together with $(K_X^2) \ge 1$, it implies $(K_X^2) \ge
\chi_h(X) -3$. Thus $X$ is of superconformal simple type by definition
(\cite[\S7.1]{MMP}). In fact, it is known that the Seiberg-Witten
classes are $\pm K_X$, and $\SW(-K_X) = 1$, $\SW(K_X) =
(-1)^{\chi_h(X)}$ (see e.g., \cite{M-book}). Therefore we cannot have
a nontrivial identity like \eqref{eq:scs}.

We consider
\begin{equation*}
  {\mathcal{SW}}(\alpha)
  \defeq \sum_\fs (-1)^{\frac{(K_X,K_X+c_1(\fs))}2} \SW(\fs) \exp{(c_1(\fs),\alpha)}.
\end{equation*}
The condition \eqref{eq:scs} is equivalent to ${\mathcal{SW}}(\alpha)$
having zero of order $\ge \chi_h(X) - (K_X^2) - 3$ at $\alpha = 0$. We
have
\begin{equation*}
  {\mathcal{SW}}(-\alpha) = (-1)^{\chi_h(X)-(K_X^2)}{\mathcal{SW}}(\alpha)
\end{equation*}
by $\SW(-c) = (-1)^{\chi_h(X)}\SW(c)$. Therefore ${\mathcal{SW}}$ is
an even (resp.\ odd) function if $\chi_h(X)-(K_X^2)$ is even (resp.\
odd). Therefore the order of zero is automatically $\ge \chi_h(X) -
(K_X^2) - 2$ under the above condition.

\begin{Example}[\protect{\cite[\S7.2]{MMP}}]
  Let $X$ be an elliptic surface $X$ without multiple fibers such that
  $H^1(X,\shfO_X) = 0$. Let $f$ be the class of a fiber. Then $K_X =
  \shfO_X(df)$ with $\chi_h(X) = d+2$. We have $(K_X^2) = 0$. The
  Seiberg-Witten invariants were computed by Friedman-Morgan
  \cite{F-M} as in Example~\ref{ex:elliptic}.
  Therefore
\begin{equation*}
  \mathcal{SW}(\alpha)
  \begin{NB}
  = \sum_{p} (-1)^p \binom{d}p e^{(2p-d)(f,\alpha)}
  = e^{-d(f,\alpha)} \left( 1 - e^{2(f,\alpha)} \right)^d
  \end{NB}
  = (-2)^{\chi_h(X)-2} \sinh^{\chi_h(X)-2}(f,\alpha).
\end{equation*}
This has zero of order $\chi_h(X) - 2$ at $\alpha = 0$. Hence $X$ is
of superconformal simple type. This example can be generalized to the
case of elliptic surfaces with multiple fibers.
\end{Example}

\begin{Example}[\protect{\cite[\S7.3]{MMP}}]
  Consider a one point blow-up $\widehat X\to X$. Let $C$ be the
  exceptional divisor. We have $(K_{\widehat X}^2) = (K_X^2) - 1$ and
  $\chi_h(\widehat X) = \chi_h(X)$.
  Let us add the subscript $X$ and $\widehat X$ to the Seiberg-Witten
  invariants $\SW$ (and $\mathcal{SW}$) in order to clarify which
  surface we consider.
  Then we have $\SW_{\widehat X}(c\pm C) = \SW_X(c)$ for $c\in H^2(X)$
  and the other $\SW_{\widehat X}(c + nC)$ vanish. Therefore
  \begin{equation*}
    \mathcal{SW}_{\widehat X}(\alpha + zC)
    = -2\mathcal{SW}_{X}(\alpha)\sinh(z).
  \end{equation*}
  \begin{NB}
    $K_{\widehat X} = K_X + C$. Therefore
    \begin{equation*}
      (-1)^{(K_{\widehat X},K_{\widehat X}+c\pm C)/2}
      = (-1)^{(K_X,K_X+c)} (-1)^{(C,C\pm C)/2}.
    \end{equation*}
  \end{NB}%
  Thus $\mathcal{SW}_X(\alpha)$ has a zero of order $\ge \chi_h(X) -
  (K_X^2) - 3$ at $\alpha = 0$ if and only if $\mathcal{SW}_{\widehat
    X}$ has a zero of order $\ge \chi_h(X) - (K_X^2) - 2 =
  \chi_h(\widehat X) - (K_{\widehat X}^2) - 3$ at $(\alpha,z) =
  0$. Thus $X$ is of superconformal simple type if and only if so is
  $\widehat X$.
\end{Example}

From these two examples and the classification of complex surfaces, we
conclude that all complex projective surfaces with $p_g > 0$, $b_1 =
0$ are of superconformal simple type. (See \cite[\S7.3]{MMP}.)

\section{Preliminaries~(II) -- Instanton counting}

\subsection{Framed moduli spaces of torsion free sheaves}

We briefly recall the framed moduli spaces of torsion free sheaves on
$\proj^2$. See \cite[Chap.~2]{Lecture} and \cite[\S3]{NY2} for more
detail.

Let $\linf$ be the line at infinity of $\proj^2$. A {\it framed
  sheaf\/} $(E,\varphi)$ on $\proj^2$ is a pair
\begin{itemize}
\item a coherent sheaf $E$, which is locally free in a neighborhood
  of $\linf$, and
\item an isomorphism $\Phi\colon E|_{\linf}\to \shfO_{\linf}^{\oplus
    r}$, where $r$ is the rank of $E$.
\end{itemize}
Let $M(r,n)$ be the moduli space of framed sheaves $(E,\varphi)$ of
rank $r$ and $c_2(E) = n$. This is a nonsingular quasi-projective
variety of dimension $2rn$. It has an ADHM type description.

Let $M_0(r,n)$ be the corresponding Uhlenbeck partial
compactification. There is a projective morphism $\pi\colon M(r,n)\to
M_0(r,n)$. This is a crepant resolution of $M_0(r,n)$.

Let $\C^*\times\C^*$ act on $\proj^2$ by $[z_0:z_1:z_2]\mapsto
[z_0:tz_1:tz_2]$, where the line $\linf$ at infinity is $z_0 = 0$. Let
$T$ be the maximal torus of $\SL_r(\C)$ consisting of diagonal
matrices and let $\hT = \C^*\times\C^*\times T$. It acts on $M(r,n)$
as follows: the first factor $\C^*\times\C^*$ acts by pull-backs of
sheaves $E$, and $T$ acts by the change of the framing $\varphi$.
It also acts on $M_0(r,n)$ and $\pi$ is equivariant.
We consider the equivariant homology group $H_*^{\hT}(M(r,n))$,
$H_*^{\hT}(M_0(r,n))$. Let $[M(r,n)]$, $[M_0(r,n)]$ be the fundamental
classes.

Fixed points $M(r,n)^{\hT}$ are parametrized by $r$-tuples of Young
diagrams $\vec{Y} = (Y_1,\dots,Y_r)$. Each $Y_\alpha$ corresponds to a
monomial ideal $I_\alpha$ of the polynomial ring $\C[x,y]$, and gives
a framed rank $1$ sheaf. The direct sum $I_1\oplus\cdots\oplus I_r$ is
a torus fixed point.
The equivariant Euler class $\Eu(T_{\vec{Y}}M(r,n))$ of the tangent
space of $M(r,n)$ at $\vec{Y}$ is given by a certain combinatorial
formula (see \cite[\S\S3,4]{NY2}), but its explicit form will not be
used in this paper.
On the other hand, $M_0(r,n)$ has a unique fixed point: the rank $r$
trivial sheaf together with a singularity concentrated at the origin.

Let $\ve_1$, $\ve_2$, $a_1$,\dots, $a_r$ (with $a_1+\cdots +a_r = 0$)
be the coordinates of the Lie algebra of $\hT$. We also use the
notation $\vec{a} = (a_1,\dots,a_r)$. The equivariant cohomology
$H^*_{\hT}(\pt)$ of a single point is naturally identified with the
polynomial ring $S(\hT) \defeq \C[\ve_1,\ve_2,a_2,\dots,a_r]$.  Let
$\mathfrak S(\hT)$ be its quotient field. The localization theorem for
the equivariant homology group says that the push-forward homomorphism
$\iota_{0*}$ of the inclusion $M_0(r,n)^{\hT}\to M_0(r,n)$ induces an
isomorphism of equivariant homology groups after tensoring by $\mathfrak
S(T)$. Since $M_0(r,n)^{\hT}$ is a single point, as we remarked above,
we have
\begin{equation*}
  \iota_{0*}\colon
  H_*^{\hT}(M_0(r,n)^{\hT})\otimes_{S(T)}\mathfrak S(T) = \mathfrak S(T) 
  \xrightarrow{\cong}H_*^{\hT}(M_0(r,n))\otimes_{S(T)}\mathfrak S(T).
\end{equation*}
Let $\iota_{0*}^{-1}$ be the inverse of $\iota_{0*}$.

We also have an isomorphism
\begin{equation*}
  \iota_{*}\colon 
  H_*^{\hT}(M(r,n)^{\hT})\otimes_{S(T)}\mathfrak S(T)
  = \mathfrak S(T)^{\oplus \#\{\vec{Y}\}}
  \xrightarrow{\cong}H_*^{\hT}(M(r,n))\otimes_{S(T)}\mathfrak S(T),
\end{equation*}
where $\iota\colon M(r,n)^{\hT}\to M(r,n)$. By the functoriality of
pushforward homomorphisms, we have
\begin{equation}\label{eq:commute}
   \sum_{\vec{Y}}\circ \iota_*^{-1} = \iota_{0*}^{-1}\circ\pi_*,
\end{equation}
where $\sum_{\vec{Y}}$ is the map
\(
   \mathfrak S(T)^{\oplus \#\{\vec{Y}\}}\to \mathfrak S(T)
\)
defined by taking sum of components.

Since $M(\Nc,n)$ is smooth, $\iota_*^{-1}$ is given by
\begin{equation*}
  \frac{\iota^*(\bullet)}{\Eu(T_{\vec{Y}}M(\Nc,n))},
\end{equation*}
where $\iota^*$ is the pull-back homomorphism of equivariant
cohomology groups, considered as a map between equivariant homology
groups via Poincar\'e duality.

Nekrasov's deformed partition function for the {\it pure\/} gauge
theory is defined as the generating function of
$\iota_{0*}^{-1}\pi_*[M(r,n)]$, where we let $n$ run. By the discussion
above, it is the generating function of $1/\Eu(T_{\vec{Y}}M(\Nc,n))$ for
all $\vec{Y}$. It was introduced in \cite{Nek} and studied in
\cite{NO,NY1,NY2}.

\subsection{The partition function for the theory with fundamental
  matters}

We need a variant of the partition function. It is called the
partition function for the theory with fundamental matters in the
physics literature.

Over the moduli space $M(\Nc,n)$, we have a natural vector bundle $\Vcal$,
whose fiber at $(E,\varphi)$ is
\(
   H^1(E(-\ell_\infty)).
\)
It has rank $n$.
\begin{NB}
In the ADHM description, this is the bundle associated with the
principal $\GL(V)$-bundle $\mu^{-1}(0)^{\mathrm{s}} \to
\mu^{-1}(0)^{\mathrm{s}}/ \GL(V) = M(\Nc,n)$.
\end{NB}%
If $\cE$ denotes the universal sheaf on $\proj^2\times M(\Nc,n)$, we
have $\Vcal = R^1 q_{2*}(\cE\otimes q_1^*(\shfO(-\linf)))$, where
$q_1$, $q_2$ are the projection from $\proj^2\to M(\Nc,n)$ to the
first and second factors respectively.

In fact, a computation shows that it is more natural to replace
$H^1(E(-\linf))$ by the $L^2$-kernel of the Dirac operator $D_{\!A}
\colon E\otimes S^- \to E\otimes S^+$, where $A$ is the instanton
corresponding to $E$ (assuming it is locally free). This is simply
given by tensoring the half canonical bundle $K_{\C^2}^{1/2}$ of
$\C^2$, i.e.\ a trivial line bundle with weight
$e^{-(\ve_1+\ve_2)/2}$. This makes sense even if $E$ is not locally
free, so we can consider it as a definition of the kernel of the Dirac
operator.

For a positive integer $N_f$, we consider a vector space $M =
\C^{N_f}$, called the {\it flavor space}. The group $\GL(M)$ naturally
acts on $M$. Let $T_M$ be the diagonal subgroup. Let $\vec{m} =
(m_1,\dots, m_{N_f})$ denote an element in $\operatorname{Lie}T_M$. 
We consider the equivariant class
\(
   \Eu(\Vcal\otimes K_{\C^2}^{1/2}\otimes M) \cap [M(\Nc,n)].
\)
This has the degree (or virtual dimension) $(2r-N_f)n$.
The theory is called {\it conformal\/} when $N_f = 2r$ and {\it
  asymptotically free\/} when $N_f < 2r$. We assume $N_f < 2r$
hereafter. We set $\bgamma \defeq 2\Nc - N_f$.

We define the instanton part of the partition function
\begin{equation}\label{eq:partition}
  \Zin(\ve_1,\ve_2,\vec{a},\vec{m};\Lambda)
%   = \sum_{n=0}^\infty \q^n 
%   Z_n(\ve_1,\ve_2,\vec{a},\vec{m})
  = \sum_{n=0}^\infty \Lambda^{\bgamma n}
  \iota_{0*}^{-1} \pi_* \left(\Eu(\Vcal\otimes K_{\C^2}^{1/2}\otimes M)
    \cap [M(\Nc,n)]\right),
\end{equation}
where $\Lambda$ is a formal variable.

Using \eqref{eq:commute} we can replace $\iota_{0*}^{-1} \pi_*$ by
$\sum_{\vec{Y}}\iota_*^{-1}$. Then we get
\begin{equation}\label{eq:sum}
  \Zin(\ve_1,\ve_2,\vec{a},\vec{m};\Lambda)
  = \sum_{n=0}^\infty \Lambda^{\bgamma |\vec{Y}|}
  \frac{\Eu(\left.\Vcal\right|_{\vec{Y}}\otimes K_{\C^2}^{1/2}\otimes M)}
    {\Eu(T_{\vec{Y}}M(\Nc,n))},
\end{equation}
where $\left.\Vcal\right|_{\vec{Y}}$ is the fiber of $\Vcal$ at the
fixed point $\vec{Y}$, and $|\vec{Y}|$ is the sum of numbers of boxes
in Young diagrams $Y_\alpha$. The right hand side has a combinatorial
expression, which will be not used in this paper.

It is known that
\begin{enumerate}
\item $\ve_1\ve_2\log\Zin(\ve_1,\ve_2,\vec{a},\vec{m};\Lambda)$ is regular
at $\ve_1$, $\ve_2 = 0$, and hence has the expansion
\begin{multline}\label{eq:expand}
%  \ve_1\ve_2\log\Zin(\ve_1,\ve_2,\vec{a},\vec{m};\Lambda) 
% \\
%  = 
 \Fin_0(\vec{a},\vec{m};\Lambda)
 + (\ve_1+\ve_2)H^{\text{\rm inst}}(\vec{a},\vec{m};\Lambda)
 + \ve_1\ve_2 \Ain(\vec{a},\vec{m};\Lambda) 
 + \frac{\ve_1^2+\ve_2^2}3 \Bin(\vec{a},\vec{m};\Lambda)
 + \cdots.
\end{multline}
\item The leading term $\Fin_0(\vec{a},\vec{m};\Lambda)$ is the
  instanton part of the Seiberg-Witten prepotential.
\end{enumerate}

For the pure theory (i.e., $N_f = 0$) these were proved by the second
and third-named authors \cite{NY1}, Nekrasov-Okounkov \cite{NO}, and
Braverman-Etingof \cite{BE} independently. The proof in \cite{NO}
works also for theories with matters.
We also need to know the next three terms $H$, $A$, $B$. These were
computed in \cite{NY2} for the pure theory.
The corresponding results for our case $\Nc=2$, $N_f=1$ along the
argument in \cite{NY1,NY2} will be explained below
(\secref{sec:curve}). We need to use the theory which we have
developed in \cite{perv1,perv2,perv3}.
In particular, we have $H^{\text{\rm inst}} \equiv 0$, which means
that the partition function is `topological': $H^{\text{\rm inst}}$ is
coupled with $\ve_1+\ve_2 = -c_1(K_{\C^2})$, which depends on the
complex structure, but it vanishes.

Since we will only consider the case $\Nc=2$, $N_f = 1$, we denote
$a_2$, $m_1$ simply by $a$, $m$ respectively. In application to
Mochizuki's formula below, we need to specialize $a=m$. This is
well-defined: Setting $a=m$ means that we restrict the acting group
from $\hT\times T_M$ to a smaller subgroup. But the smaller subgroup
still has the same fixed points (as $T_M$ acts trivially), and the
fixed point formula can be specialized.

In view of Conjecture~\ref{con:main}, it is desirable to have a direct
definition of the partition function in terms of the Uhlenbeck
compactification $M_0(r,n)$, not appealing to the algebro-geometric
object $M(r,n)$. Since $\Vcal$ is not a pull-back from $M_0(r,n)$,
this is a nontrivial problem. If we consider $M_0(r,n)$ as an affine
algebraic variety, then
\(
  \pi_* \left(\Eu(\Vcal\otimes K_{\C^2}^{1/2}\otimes M)
    \cap [M(\Nc,n)]\right)
\)
is a limit of the formal $\hT\times T_M$-character of the space of
sections of certain virtual sheaves on $M_0(\Nc,n)$ as in
\cite[\S4]{NY1}. It should be possible to replace this virtual sheaf
by a complex of vector bundles.

\section{Mochizuki's formula and the partition function}

As we mentioned in Introduction, we will use Mochizuki's formula
relating Donaldson invariants and Seiberg-Witten invariants. Before
stating his formula, let us briefly explain the idea behind its proof.
A reader can safely jump to \subsecref{subsec:Mochizuki} if he/she
accepts Mochizuki's formula. But the authors encourage the reader to
learn Mochizuki's beautiful ideas. Of course he/she should read the
book \cite{Moc} for more detail.

When $X$ is a complex projective surface, Mochizuki first developed
the obstruction theory for moduli spaces of pairs of sheaves and
sections and related spaces.
Then he obtained a general machinery to write down the difference of
invariants for two moduli spaces defined with different stability
condition. A point is to introduce a $\C^*$-equivariant obstruction
theory on the `master space' containing two moduli spaces as
$\C^*$-fixed point loci. He integrated the class $\exp(\mu(\alpha z +
px))\cup a$ over the master space, where $a$ is the generator of the
equivariant cohomology group $H^*_{\C^*}(\pt)$ of a single point. 
Since the integral vanishes at the nonequivariant limit $a=0$, the sum
of residues of fixed point loci contributions is zero by the
Atiyah-Bott-Lefschetz fixed point formula. This gives the difference
of the invariants as the sum of residues of `exceptional' fixed points
loci contributions. The exceptional fixed points are products of lower
rank sheaves and pairs. Up to this point, the framework is essentially
the same as the $\SO(3)$-monopole cobordism program, except for a
systematic usage of the obstruction theory. But a crucial difference
is that Mochizuki's obstruction theory enables him to treat moduli
spaces as if they are {\it smooth}. In particular, his `residues' are
given explicitly in terms of equivariant Euler classes of virtual
normal bundles.

He applied this theory to the case of moduli spaces of rank $2$ pairs.
When a stability condition is suitably chosen, moduli spaces of pairs
are projective bundles over moduli of genuine sheaves, thus the
invariants are reduced to Donaldson invariants. On the other hand, for
another stability condition, moduli spaces become the empty set. The
difference of the invariants, which is just Donaldson invariants, is
given by the sum of residues of equivariant integrals over other
`exceptional' fixed point loci, which are moduli spaces of pairs of
rank $1$ sheaves with sections of one factor.
These exceptional contribution can be identified with a product of the
Seiberg-Witten invariant and an equivariant integral over the product
$X^{[n_1]}\times X^{[n_2]}$ of Hilbert schemes of points in $X$.
This is because rank $1$ sheaves are just ideal sheaves twisted by line
bundles.
The class $Q$ appearing the formula below is the equivariant Euler
class of the normal bundle mentioned above.

\subsection{Mochizuki's formula}\label{subsec:Mochizuki}

Let $y = (2,\xi,n)$, $\alpha$, $p$, $z$, $x$ as in the definition of
Donaldson invariants (\subsecref{subsec:Donaldson}).
Suppose that we have decompositions $\xi = \xi_1 + \xi_2$, $n -
(\xi_1,\xi_2) = n_1 + n_2$. We denote by $e^{\xi_i}$ the holomorphic
line bundle whose first Chern class is $\xi_i$. Let $\mathcal I_i$
(resp.\ $\shfO_{Z_i}$) denote the universal ideal sheaf (resp.\
subscheme) over $X\times X^{[n_i]}$. Their pull-backs to $X\times
X^{[n_1]}\times X^{[n_2]}$ are denoted by the same notation.
Let $q_2\colon X\times X^{[n_1]}\times X^{[n_2]}\to X^{[n_1]}\times
X^{[n_2]}$ be the projection.

Let $\C^*$ act trivially on $X^{[n_1]}\times X^{[n_2]}$ and consider
the equivariant cohomology group $H^*_{\C^*}(X^{[n_1]}\times
X^{[n_2]}) \cong H^*(X^{[n_1]}\times X^{[n_2]})[a]$, where $a$ is the
variable for $H^*_{\C^*}(\pt)$, i.e., $H^*_{\C^*}(\pt) = \C[a]$.
We consider the following equivariant cohomology classes on
$X^{[n_1]}\times X^{[n_2]}$:
\begin{equation*}
  \begin{split}
    P(\cal I_1 e^{\xi_1-a} \oplus \cal I_2 e^{\xi_2+a}) &\defeq
    \exp(-\ch_2(\cal I_1 e^{\xi_1-a-\frac{\xi}{2}} \oplus 
    \cal I_2 e^{\xi_2+a-\frac{\xi}{2}})
    /(\alpha z+px)),
\\
   Q(\cal I_1 e^{\xi_1-a} \oplus \cal I_2 e^{\xi_2+a})
&\defeq \Eu(-\Ext^*_{q_2}(\cal I_1 e^{\xi_1-a}, 
\cal I_2 e^{\xi_2+a}))
\Eu(-\Ext^*_{q_2}(\cal I_2 e^{\xi_2+a}, 
\cal I_1 e^{\xi_1-a})),
  \end{split}
\end{equation*}
where $\Ext^*_{q_2}$ is the alternating sum $\Ext^0_{q_2} - \Ext^1_{q_2} +
\Ext^2_{q_2}$, and $\Ext^\bullet_{q_2}$ is the derived functor of the composite
$q_{2*}\circ{\mathcal H}om$.

Roughly speaking, $Q$ is the equivariant Euler class of the virtual
normal bundle of $X^{[n_1]}\times X^{[n_2]}$ in $M_H(y)$. Here one
should consider that the embedding is given by $(I_1,I_2)\mapsto
e^{\xi_1}I_1\oplus e^{\xi_2}I_2$. And $P$ is the restriction of the
integrand appearing in Donaldson invariants. But the precise
formulation requires the master space, and is omitted in this paper.

Note that $Q$ is invertible in $H^*(X^{[n_1]}\times
X^{[n_2]})[a,a^{-1}]$ as it has a form
\(
  Q(\cal I_1 e^{\xi_1-a} \oplus \cal I_2 e^{\xi_2+a})
  = a^{N} + (\text{lower degree in $a$})
\)
for some $N$. We consider the following class
in $H^*(X^{[n_1]}\times X^{[n_2]})[a,a^{-1}]$:
\begin{multline*}
  \widetilde{\varPsi}(\xi_1,\xi_2,n_1,n_2;a)
\\
  \defeq 
\frac{P(\cal I_1 e^{\xi_1-a} \oplus \cal I_2 e^{\xi_2+a})}
{Q(\cal I_1 e^{\xi_1-a} \oplus \cal I_2 e^{\xi_2+a})}
\frac{\Eu(H^*(({\cal O}/{\cal I}_1) e^{\xi_1}))
\Eu(H^*(({\cal O}/{\cal I}_2) e^{\xi_2+2a}))}
{(2a)^{n_1 +n_2-p_g}},
\end{multline*}
where $H^*({\cal O}/{\cal I}_i)$ is the alternating sum of the higher
direct image sheaves $R^\bullet q_{2*}({\cal O}/{\cal I}_i)$.
This is the same as Mochizuki's $\varPsi$ (\cite[\S1.4.2]{Moc}), except
that we do not take the residue with respect to $a$. Therefore we put
`$\tilde{\ }$' in the notation.

We set
\begin{equation*}
\widetilde{\cal A}(\xi_1,y;a)=
2^{1-\chi(y)}\sum_{n_1 +n_2=n-(\xi_1, \xi_2)}
\int_{X^{[n_1]} \times X^{[n_2]}}
  \widetilde\varPsi(\xi_1,\xi_2,n_1,n_2;a),
\end{equation*}
where $\chi(y)$ is the Euler characteristic of the class $y$.
By Riemann-Roch, we have
\begin{equation*}
  \chi(y) = \frac{(\xi,\xi-K_X)}{2}+2\chi_h(X)-n.
\end{equation*}

\begin{Theorem}[\protect{\cite[Th.~1.4.6]{Moc}}]\label{thm:Mochizuki}
  Assume that $\chi(y) > 0$, $(\xi,H)/2 > (K_X,H)$ and $(\xi,H) >
  (c_1(\fs) + K_X,H)$ for any Seiberg-Witten class $\fs$. Then we have
  \begin{equation*}
    \frac12
    \int_{M_H(y)} \exp(\mu(\alpha z + px))
    = \sum_{\xi_1} \SW(\tilde\xi_1)
    \Res_{a=\infty} \widetilde{\cal A}(\xi_1,y;a)da,
  \end{equation*}
  where $\tilde\xi_1 \defeq 2\xi_1 - K_X$.
\end{Theorem}

Let us give several remarks.

\begin{Remarks}
  (1) The left hand side is Mochizuki's definition of the invariant
  using the obstruction theory. It is equal to the usual Donaldson
  invariant if $y$ is primitive and $M_H(y)$ is of expected
  dimension. This is not an essential assumption, as we explained in
  \subsecref{subsec:complex}.

(2) Mochizuki took the residue at $a=0$,
instead of $a=\infty$. But ours is just the negative of Mochizuki's,
as $\widetilde{\cal A}(\xi_1,y;a)$ is in $\C[a,a^{-1}]$.

(3) The factor $1/2$ in the left hand side comes from Mochizuki's
convention. He considered the integration over the moduli space of
{\it oriented\/} sheaves. There is a natural \'etale proper morphism
from the oriented moduli space to the usual one of degree
$(\rank)^{-1} = 1/2$.

(4) The assumption is satisfied if we replace $y$ by $y e^{kH}$ for
sufficiently large $k$. But it is not clear, a priori, that the right
hand side is independent of $k$. This will become important for our
later analysis of the residue of $\widetilde{\mathcal A}$.

(5) Mochizuki denoted the usual Seiberg-Witten invariant by
$\widetilde{\SW}$ and set $\SW(\xi_1) =
\widetilde{\SW}(2\xi_1-K_X)$. We keep $\SW$ for the notation of the
usual Seiberg-Witten invariant. On the other hand, the Seiberg-Witten
class $2\xi_1 - K_X$ will naturally appears in the Witten's formula
\eqref{eq:Witten}. Therefore we have denoted it by $\tilde\xi_1$. Thus
our $\SW(\tilde\xi_1)$ is Mochizuki's $\SW(\xi_1)$.

(6) Since the expected dimension $\dim M_H(y)$ is
$4n-(\xi^2)-3\chi_h(X)$, we have
\begin{equation}\label{eq:chiy}
  4\chi(y) = ((\xi-K_X)^2) - (K_X^2) - \dim M_H(y) + 5\chi_h(X).
\end{equation}
\begin{NB}
  By Riemann-Roch, we have
  \begin{equation*}
    \chi(y) = \int_X (2 + \xi + \frac12\xi^2 - n)(1 - \frac12 K_X + \Todd_2(X))
    = 2\chi_h(X) + (\xi,\xi-K_X)/2 - n.
  \end{equation*}
\end{NB}
\end{Remarks}

\subsection{Formula in terms of the partition function}
\label{subsec:partition}

Now we prove our first main result. Recall $y = (2,\xi,n)$. Let us
introduce the generating function of the  $\widetilde{\mathcal
  A}(\xi_1,y;a)$:
\begin{equation*}
  \mathcal B(\xi_1,\xi;a) \defeq \sum_n \Lambda^{4n-(\xi^2)-3\chi_h(X)}
  \widetilde{\mathcal A}(\xi_1,(2,\xi,n);a).
\end{equation*}

\begin{Theorem}\label{thm:partition}
  We have
  \begin{multline*}
    \mathcal B(\xi_1,\xi;a) da
    = - \frac{da}a
    (-1)^{(\xi,\xi+K_X)/2 + (K_X,K_X + \tilde\xi_1)/2 +\chi_h(X)} %\\
    2^{-2\chi_h(X)-
%      \frac{((\xi_2-\xi_1)^2)+(\xi_2-\xi_1,\tilde\xi_1)}{2}}\\
      (\xi-K_X-\tilde\xi_1,\xi-K_X)/2}\\
    \times \left(\frac{2a}\Lambda\right)
    ^{((\xi-K_X)^2)+(K_X^2)+3\chi_h(X)-2(\xi-K_X,\tilde\xi_1)}%\\
    \exp\left({-(\xi-K_X-\tilde\xi_1,\alpha)az-a^2 x}\right)\\
    \times \exp \left[
      \frac{1}{3}\frac{\partial \Fin_0}{\partial \log \Lambda}x
      +
      \left(\frac{1}{8}\frac{\partial^2 \Fin_0}{\partial a^2}
        +\frac{1}{4}\frac{\partial^2 \Fin_0}{\partial a \partial m}
        +\frac{1}{8}\frac{\partial^2 \Fin_0}{\partial m^2} \right)
      ((\xi-K_X)^2) \right.\\
    -\frac{1}{4}\left(\frac{\partial^2 \Fin_0}{\partial a \partial m}
      +\frac{\partial^2 \Fin_0}{\partial a^2} \right)
    (\xi-K_X,\tilde\xi_1)\\
    +\frac{1}{6}\left(
      \frac{\partial^2 \Fin_0}{\partial a \partial \log \Lambda}+
      \frac{\partial^2 \Fin_0}{\partial m \partial \log \Lambda} \right)
    (\xi-K_X,\alpha) z
    -\frac{1}{6}
    \frac{\partial^2 \Fin_0}{\partial a \partial \log \Lambda}
    (\tilde\xi_1,\alpha) z\\
    \left.+\frac{1}{18}
      \frac{\partial^2 \Fin_0}{(\partial \log \Lambda)^2}
      (\alpha^2) z^2
      + \chi_h(X)(12\Ain-8\Bin)
      +(K_X^2)
      \left(\Bin-\Ain+\frac{1}{8}\frac{\partial^2 \Fin_0}{\partial a^2}
        \right)\right],
  \end{multline*}
  where $\tilde\xi_1 = 2\xi_1 - K_X$ as above, and the derivatives of
  $\Fin_0$, $\Ain$ and $\Bin$ are evaluated at $(a,m,\Lambda) =
  (a,a,\Lambda^{4/3}a^{-1/3})$.
\end{Theorem}

Observe that our formula does not depend on the complex structure of
$X$ when we consider the canonical class $K_X$ as a choice of a
spin${}^c$ structure. Therefore, the above expression makes sense for
a smooth $4$-manifold $X$.
Further observe that $K_X$ appears in the above expression only as
either $(K_X^2)$ or the combination $\xi-K_X$, except in the sign
factor.
\begin{NB}
We can do slightly more precise.

The first sign factor $(-1)^{(\xi,\xi+K_X)/2}$ comes from the
difference between the complex orientation and the usual orientation
as mentioned in \subsecref{subsec:complex}.
The second factor $(-1)^{(K_X,K_X+\tilde\xi_1)/2}$ can be also replaced by
$(-1)^{(\xi,\xi-\tilde\xi_1)/2}$ up to an expression involving only
$\xi-K_X$, $(K_X^2)$. (See the discussion below.****) This
$(-1)^{(\xi,\xi-\tilde\xi_1)/2}$ already appears in Witten's formula.
\end{NB}%
If we ignore the sign, the Donaldson invariants depend only on
$(\xi\bmod 2)$, so we can consider $\xi - K_X$ as auxiliary cohomology
class. The only requirement is that it is equal to $(\xi\bmod 2) +
w_2(X)$ in $H^2(X,\Z/2)$.

Therefore we pose the following:
\begin{Conjecture}\label{con:main}
  Mochizuki's result (\thmref{thm:Mochizuki}) holds for a smooth
  $4$-manifold $X$ with $b_1 = 0$, $b_+\ge 3$ odd, up to sign, if we
  replace $\widetilde{\mathcal A}(\xi_1,y;a)$ by coefficients of
  $\mathcal B(\xi_1,\xi;a)$ in \thmref{thm:partition}.
\end{Conjecture}

We have the conditions $(\xi,H)/2 > (K_X,H)$, $(\xi,H) > (c_1(\fs) +
K_X,H)$, which we do not know how to interpret for a smooth
$4$-manifold $X$. Therefore we just ignore this condition and
conjecture that Mochizuki's result holds without it.

This conjecture is compatible with Feehan-Leness' result \eqref{eq:FL}.
Our formula in \thmref{thm:partition} involves only the intersection pairings
among $\tilde\xi_1$, $\xi-K_X$ and $\alpha$.
Their formula involves an auxiliary cohomology class, denoted by
$\Lambda$ in \cite[Th.~3.1]{FLA}, is equal to $\Lambda = c_1(\fs_0) +
\xi$ for a chosen spin${}^c$ structure $\fs_0$. We take the canonical
spin${}^c$ structure of the complex surface $X$ as $\fs_0$, so their
$\Lambda$ should be identified with our $\xi - K_X$.
In fact, $\Lambda$ satisfies the same condition which we have assumed
for $\xi-K_X$. It is required to satisfy the same condition as
$\chi(y) > 0$ thanks to \eqref{eq:chiy} (written as `$I(\Lambda) >
\delta$' [loc.\ cit.]).
We also remark that the exponent
\(
  ((\xi-K_X)^2) + (K_X^2) + 3\chi_h(X) - 2(\xi-K_X,\tilde\xi_1)
\)
of $2a/\Lambda$ is equal to $-r(\Lambda,c_1(\mathfrak s))$ in
\cite[(1.12)]{FL2} if we take $\Lambda = \xi-K_X$, $\mathfrak s =
\tilde\xi_1$ and replace $(K_X^2)$ by $(c_1(\mathfrak s)^2)$.

Thus our conjecture follows immediately if the coefficients $f_{k,l}$
appearing in Feehan-Leness' formula \eqref{eq:FL} are the same as
ours.
This does not directly follow from Feehan-Leness' statement itself,
as Seiberg-Witten invariants satisfy nontrivial relations, namely
superconformal simple type condition, therefore the coefficients are
not uniquely determined.

The authors' heuristic proof is the following: there is a morphism
from Mochizuki's master space to the moduli space of $\SO(3)$-monopole
when $X$ is complex projective. Then by the functoriality of
pushforward homomorphisms as used in \eqref{eq:commute}, the
contributions of Seiberg-Witten invariants are the same for
Mochizuki's and Feehan-Leness' formulas. Since Feehan-Leness' formula
is universal, it is enough to calculate them for complex projective
$X$, and our calculation gives the answer.

This proof works only for $X$ of simple type and for which there is a
complex projective surface $X_0$ with $\chi(X) = \chi(X_0)$,
$\sigma(X) = \sigma(X_0)$.
To generalize it for hypothetical $X$ of non-simple type, we need to
connect Feehan-Leness' coefficients with Nekrasov partition function
more directly.

\begin{proof}[Proof of \thmref{thm:partition}]
%  To be copied from Kota's note.
  The proof is similar to that of the wall-crossing formula for
  $b_+=1$ in \cite{GNY}. When the argument is really the same, we just
  point to the corresponding argument in [loc.\ cit.].

  We denote $a$ in Mochizuki's $\widetilde{\mathcal A}$ by $s$ for a
  moment.

  We first write $\mathcal B$ as a product of the `perturbative term',
  i.e., an expression independent of $n_1$, $n_2$ and the `instanton
  part', which is $1$ if $n_1 = n_2 = 0$. For the term $P$, we have
\begin{equation*}
\begin{split}
& P(\cal I_1 e^{\xi_1-s} \oplus \cal I_2 e^{\xi_2+s}) \\
=\; & \exp(-(\xi_2 -\xi_1,\alpha)sz-s^2 x)
\exp([c_2({\cal I}_1)+c_2({\cal I}_2)]/(\alpha z+px)).
\end{split}
\end{equation*}
Thus the perturbative term is $\exp(-(\xi_2 -\xi_1,\alpha)sz-s^2 x)$.
\begin{NB}
More detail:
\begin{equation*}
\begin{split}
& P(\cal I_1 e^{\xi_1-s} \oplus \cal I_2 e^{\xi_2+s}) \\
=& \exp(-\ch_2(\cal I_1 e^{\xi_1-s-\frac{\xi}{2}} \oplus 
\cal I_2 e^{\xi_2+s-\frac{\xi}{2}})
/(\alpha z+px))\\
=& \exp(-\ch_2(\cal I_1 e^{\frac{\xi_1-\xi_2}{2}-s} 
\oplus \cal I_2 e^{\frac{\xi_2 -\xi_1}{2}+s})
/(\alpha z+px))\\
=& \exp(-[\frac{1}{2}(\frac{\xi_1-\xi_2}{2}-s)^2+c_2({\cal I}_1)+
\frac{1}{2}(\frac{\xi_2-\xi_1}{2}+s)^2+c_2({\cal I}_2)]/(\alpha z+px)))\\
=& \exp(-(\xi_2 -\xi_1,\alpha)sz-s^2 x)
\exp([c_2({\cal I}_1)+c_2({\cal I}_2)]/(\alpha z+px)).
\end{split}
\end{equation*}
\end{NB}%
For $Q$, the perturbative term is
\begin{equation*}
\begin{split}
&\Eu(H^*({\cal O}_X(\xi_1-\xi_2))e^{-2s})
\Eu(H^*({\cal O}_X(\xi_2-\xi_1))e^{2s})\\
=\; &
(-2s)^{\chi({\cal O}_X(\xi_1-\xi_2))}
(2s)^{\chi({\cal O}_X(\xi_2-\xi_1))}\\
=\; & (-1)^{\frac{(\xi_1-\xi_2,\xi_1-\xi_2-K_X)}{2}+\chi_h(X)}
(2s)^{((\xi_1-\xi_2)^2)+2\chi_h(X)}\\
=\; &(-1)^{\frac{(\xi,\xi-K_X)}{2}+(K_X,\xi_2)+\chi_h(X)}
(2s)^{((\xi_1-\xi_2)^2)+2\chi_h(X)}.
\end{split}
\end{equation*}
We also have $2^{1-\chi(y)}(2s)^{p_g-n_1-n_2}$ whose perturbative part is
\begin{equation*}
  2^{1 - \frac{(\xi,\xi-K_X)}2 - 2\chi_h(X) + (\xi_1,\xi_2)}
  (2s)^{\chi_h(X) - 1}.
\end{equation*}
For the power of $\Lambda$, the perturbative part is
\begin{equation*}
  \Lambda^{4(\xi_1,\xi_2) - (\xi^2) - 3\chi_h(X)}
  = \Lambda^{-((\xi_1-\xi_2)^2)-3\chi_h(X)}.
\end{equation*}
Combining all these terms, we find that the perturbative part of
$\mathcal B$ is
\begin{multline}
  \label{eq:Bpert}
  (-1)^{\frac{(\xi,\xi-K_X)}{2}+(K_X,\xi_2)+\chi_h(X)}
\\
  \times\frac{1}{s}
  \left(\frac{2s}{\Lambda}\right)^{((\xi_1-\xi_2)^2)+3\chi_h(X)}
  e^{-(\xi_2-\xi_1,\alpha)sz-s^2 x}
  2^{-2\chi_h(X)-\frac{(\xi,\xi-K_X)}{2}+(\xi_1,\xi_2)}.
\end{multline}

By the argument in [loc.\ cit., \S5], it is enough to compute the
instanton part for a toric surface $X$.
\begin{NB}
The image of 
$$
(X,\xi,\eta,\beta) \mapsto
((\xi^2),(\xi,\eta),(\xi,K_X),(\xi,\beta),(\eta^2), (\eta,K_X),
(\eta,\beta),(K_X^2),c_2(X),
(K_X,\beta),(\beta^2),\beta)
$$
spans 12 dimensional $\Q$-vector space.
Indeed 
$$
(X,\xi,\eta,\beta)=({\Bbb P}^2,0,0,0), 
({\Bbb P}^1 \times {\Bbb P}^1,0,0,0),({\Bbb P}^2,0,0,H),
({\Bbb P}^2,0,0,2H),({\Bbb P}^2,0,0,p)$$
 spans the subspace
$x_1=x_2=\cdots=x_7=0$.
By adding
$$
(X,\xi,\eta,\beta)=({\Bbb P}^2,H,0,0),({\Bbb P}^2,2H,0,0),
({\Bbb P}^2,H,0,H),$$
 we have the subspace $x_1=\cdots=x_4=0$.  
For 
$$
(X,\xi,\eta,\beta)=({\Bbb P}^2,H,0,0),({\Bbb P}^2,2H,0,0),({\Bbb P}^2,H,0,H),
({\Bbb P}^2,H,H,0), 
$$
we can spans the first 4 components $x_1,...,x_4$.
Therefore we get the claim.
\end{NB}%

We write $\chi\defeq\chi(X)$ for brevity.
Let $p_1,p_2,...,p_{\chi}$ be the torus fixed points,
$x_i,y_i$ the torus equivariant coordinates at $p_i$ and 
$w(x_i),w(y_i)$ the weights of the torus action.
As in [loc.\ cit., \S3.2] we apply the Atiyah-Bott-Lefschetz fixed point
formula to $\widetilde{\mathcal A}(\xi_1,y;s)$.
At torus fixed points, the ideal sheaves $\cal I_1$, $\cal I_2$ are
the intersection of ideal sheaves supported at points
$p_i$. Accordingly the cohomology groups in $Q$ and the matter factor
\(
\Eu(H^*(({\cal O}/{\cal I}_1) e^{\xi_1}))
  \Eu(H^*(({\cal O}/{\cal I}_2) e^{\xi_2+2s}))
\)
decompose as products of local contributions at $p_i$.
As in [loc.\ cit., \S3.2] we will identify these local contributions
with factors in the partition function $\Zin$.

Let us first study how variables appearing in $\Zin$ will be
identified with expressions in the local contribution of
$\widetilde{\mathcal A}$ at $p_i$.
The variables $\ve_1$, $\ve_2$ in $\Zin$ are identified with $w(x_i)$,
$w(y_i)$.
In order to identify $a_1$, $a_2$, consider the factor $Q$.
In the definition of the partition function the first Chern class of
the universal sheaf is normalized to be $0$ as $a_1+a_2 = 0$. In view
of [loc. cit., Lemma~3.4], this normalization must be performed for
$Q$ as
\begin{equation*}
  \Ext^*_{q_2}(\cal I_1 e^{\xi_1-s}, \cal I_2 e^{\xi_2+s})
    = \Ext^*_{q_2} (\cal I_1 e^{\xi_1 - s - \xi/2},
      \cal I_2 e^{\xi_2+s-\xi/2}).
\end{equation*}
Thus we get the same expression appearing in $P$, and we will identify
variables as
\begin{equation*}
  \begin{split}
  a_1 &= - s + \iota_{p_i}^*(\xi_1  - \xi/2)
  = - s - \iota_{p_i}^*(\xi_2 - \xi_1)/2, \\
  a_2 &= s + \iota_{p_i}^*(\xi_2 - \xi/2) = s + \iota_{p_i}^*(\xi_2 - \xi_1)/2
  \end{split}
\end{equation*}
in $\Zin$ and $\widetilde{\mathcal A}$. Here $\iota_{p_i}^*$ is the
pull-back homomorphism associated with the inclusion $\iota_{p_i}$ of
the fixed point $p_i$ into $X$.

Accordingly we normalize the matter factor as
\begin{multline*}
  \Eu(H^*(({\cal O}/{\cal I}_1) e^{\xi_1}))
  \Eu(H^*(({\cal O}/{\cal I}_2) e^{\xi_2+2s}))
\\
  =
  \Eu(H^*(({\cal O}/{\cal I}_1) e^{\xi_1 - s - \xi/2 + s + \xi/2}))
  \Eu(H^*(({\cal O}/{\cal I}_2) e^{\xi_2 + s - \xi/2 + s + \xi/2})).
\end{multline*}
Recalling that we put $K_{\C^2}^{1/2}$ in the partition function, we
identify the variable $m$ for the matter with
$s+\iota_{p_i}^*(\xi-K_X)/2$, as $\xi_\alpha - s - \xi/2$ is
$a_\alpha$ for $\alpha=1,2$.

Next we consider the variable $\Lambda$.
After removing the perturbative part as above, we consider
$\Lambda^{4(n_1+n_2)}$ in $\mathcal B$. On the other hand, we use
$\Lambda^{3(n_1+n_2)}$ in the definition of the partition function.
We combine this with $s^{n_1+n_2}$ in $\widetilde{\varPsi}$, which
we then absorb into the variable $\Lambda$ in the partition function
\eqref{eq:partition}. Therefore $\Lambda$ in \eqref{eq:partition} will
be replaced by $\Lambda^{4/3}s^{-1/3}$.

Now we use the argument in [loc.\ cit., \S3.2] to write the instanton
part of $\mathcal B$ in terms of the partition function:
\begin{multline*}
  \lim_{\ve_1,\ve_2\to 0} \Res_{s=0} \prod_{i=1}^{\chi}
\Zin(w(x_i),w(y_i),\iota_{p_i}^*(\frac{\xi_2-\xi_1}{2})+s,
\iota_{p_i}^*(\frac{\xi-K_X}{2})+s; \frac{\Lambda^{4/3}}{s^{{1}/{3}}}
e^{\iota_{p_i}^*(\frac{\alpha z+px}{3})})
\end{multline*}

We need to explain the last expression $e^{\iota_{p_i}^*(\frac{\alpha
    z+px}{3})}$. This comes from
$\exp([c_2({\cal I}_1)+c_2({\cal I}_2)]/(\alpha z+px))$, which
is the instanton part of $P$.
We use the same argument as in [loc.\ cit., Cor.~3.18], which was
based on \cite[\S4.5]{NY2}. Let us briefly recall the point of the
argument:
We can put more variables $\vec{\tau} = (\tau_\rho)_{\rho\ge 1}$ into
the partition function $\Zin$ as in [loc.\ cit., (1.4)],
\cite[\S4.2]{NY2}. But we only need $\tau_1$ since we only use $\ch_2$
and not higher Chern classes in the Donaldson invariants. Then
$\tau_1$ can be absorbed into the variable $\Lambda$ as $\ch_2$ is
determined by $n$ of $M(r,n)$.
We identify $\tau_1 = -\iota_{p_i}^*(\alpha z + px)$ as in [loc.\ cit.].
In fact, the absorption of $\tau_1$ into $\Lambda$ is simpler than 
in \cite[\S4.2]{NY2}, as we do not put the perturbative term in the
partition function. We just need to note that it is a multiplication
of $e^{\iota_{p_i}^*(\frac{\alpha z+px}{3})}$ instead of
$e^{\iota_{p_i}^*(\frac{\alpha z+px}{4})}$, because we use
$\Lambda^{\bgamma n} = \Lambda^{3n}$ instead of $\Lambda^{4n}$ in the
definition of the partition function.

We now use the expansion \eqref{eq:expand} together with $H^{\text{\rm
    inst}} = 0$. As in [loc.~cit., proof of Th.~4.2], we have
{\allowdisplaybreaks
\begin{multline*}
\prod_{i=1}^{\chi}
\Zin(w(x_i),w(y_i),i_{p_i}^*(\frac{\xi_2-\xi_1}{2})+s,
i_{p_i}^*(\frac{\xi-K_X}{2})+s;
\Lambda(\frac{\Lambda}{s})^{\frac{1}{3}}
e^{i_{p_i}^*(\frac{\alpha z+px}{3})})\\
=
\exp \left[
\sum_i \frac{1}{w(x_i)w(y_i)}
\left(\Fin_0+ \frac{\partial \Fin_0}{\partial a}
i_{p_i}^*(\frac{\xi_2-\xi_1}{2})+
\frac{\partial \Fin_0}{\partial m}
i_{p_i}^*(\frac{\xi-K_X}{2})+
\frac{\partial \Fin_0}{\partial \log \Lambda}
i_{p_i}^*(\frac{\alpha z+px}{3}) \right. \right. \\
+\frac{1}{2}\frac{\partial^2 \Fin_0}{\partial a^2}
i_{p_i}^*(\frac{\xi_2-\xi_1}{2})^2
+\frac{\partial^2 \Fin_0}{\partial a \partial m}
i_{p_i}^*(\frac{\xi_2-\xi_1}{2})i_{p_i}^*(\frac{\xi-K_X}{2})
+\frac{1}{2}\frac{\partial^2 \Fin_0}{\partial m^2}
i_{p_i}^*(\frac{\xi-K_X}{2})^2\\
+
\frac{\partial^2 \Fin_0}{\partial a \partial \log \Lambda}
i_{p_i}^*(\frac{\xi_2-\xi_1}{2})i_{p_i}^*(\frac{\alpha z+px}{3})
+
\frac{\partial^2 \Fin_0}{\partial m \partial \log \Lambda}
i_{p_i}^*(\frac{\xi-K_X}{2})i_{p_i}^*(\frac{\alpha z+px}{3})+
\frac{1}{2}
\frac{\partial^2 \Fin_0}{\partial \log \Lambda^2}
i_{p_i}^*(\frac{\alpha z+px}{3})^2\\
\left. \left. +w(x_i)w(y_i)\Ain+\frac{w(x_i)^2+w(y_i)^2}{3}\Bin
\right)\right]\\
=\exp \left[
\frac{1}{3}\frac{\partial \Fin_0}{\partial \log \Lambda}x
+\frac{1}{8}\frac{\partial^2 \Fin_0}{\partial a^2}
((\xi_2-\xi_1)^2)
+\frac{1}{4}\frac{\partial^2 \Fin_0}{\partial a \partial m}
(\xi_2-\xi_1,\xi-K_X)
+\frac{1}{8}\frac{\partial^2 \Fin_0}{\partial m^2}
((\xi-K_X)^2) \right.\\
\left. +\frac{1}{6}
\frac{\partial^2 \Fin_0}{\partial a \partial \log \Lambda}
(\xi_2-\xi_1,\alpha) z
+\frac{1}{6}
\frac{\partial^2 \Fin_0}{\partial m \partial \log \Lambda}
(\xi-K_X,\alpha) z+
\frac{1}{18}
\frac{\partial^2 \Fin_0}{\partial \log \Lambda^2}
(\alpha^2) z^2
+ \chi(X)\Ain+\sigma(X)\Bin\right]\\
+O(\ve_1,\ve_2),
\end{multline*}
where} we evaluate the derivatives of $\Fin_0$, $\Ain$, $\Bin$ at
$a_2=a=s$, $m=s$, $\Lambda=\frac{\Lambda^{{4}/{3}}}{s^{{1}/{3}}}$.
We now safely change $s$ back to $a$.
\begin{NB}
\begin{multline*}
(-1)^{\frac{(\xi,\xi-K_X)}{2}+(K_X,\xi_2)+\chi({\cal O}_X)}
\frac{1}{2s}
\left(\frac{2s}{\Lambda}\right)^{((\xi_1-\xi_2)^2)+3\chi({\cal O}_X)}
e^{-(\xi_2-\xi_1,\alpha)sz-s^2 x}
2^{-2\chi({\cal O}_X)+1-\frac{(\xi,\xi-K_X)}{2}+(\xi_1,\xi_2)}\\
\times \exp \left[
\frac{1}{3}\frac{\partial \Fin_0}{\partial \log \Lambda}x
+\frac{1}{8}\frac{\partial^2 \Fin_0}{\partial a_2^2}
((\xi_2-\xi_1)^2)
+\frac{1}{4}\frac{\partial^2 \Fin_0}{\partial a_2 \partial m}
(\xi_2-\xi_1,\xi-K_X)
+\frac{1}{8}\frac{\partial^2 \Fin_0}{\partial m^2}
((\xi-K_X)^2) \right.\\
\left. +\frac{1}{6}
\frac{\partial^2 \Fin_0}{\partial a_2 \partial \log \Lambda}
(\xi_2-\xi_1,\alpha) z
+\frac{1}{6}
\frac{\partial^2 \Fin_0}{\partial m \partial \log \Lambda}
(\xi-K_X,\alpha) z+
\frac{1}{18}
\frac{\partial^2 \Fin_0}{\partial \log \Lambda^2}
(\alpha^2) z^2
+ \chi(X)\Ain+\sigma(X)\Bin\right]\\
=
(-1)^{\frac{((\xi_2-\xi_1)^2)+(K_X,\xi_2-\xi_1))}{2}+\chi({\cal O}_X)}
\frac{1}{s}
\left(\frac{2s}{\Lambda}\right)^{((\xi_1-\xi_2)^2)+3\chi({\cal O}_X)}
e^{-(\xi_2-\xi_1,\alpha)sz-s^2 x}
2^{-2\chi({\cal O}_X)-
\frac{((\xi_2-\xi_1)^2)+(\xi_2-\xi_1,\widetilde{\xi}_1)}{2}}\\
\times \exp \left[
\frac{1}{3}\frac{\partial \Fin_0}{\partial \log \Lambda}x
+
\left(\frac{1}{8}\frac{\partial^2 \Fin_0}{\partial a_2^2}
+\frac{1}{4}\frac{\partial^2 \Fin_0}{\partial a_2 \partial m}
+\frac{1}{8}\frac{\partial^2 \Fin_0}{\partial m^2} \right)
((\xi_2-\xi_1)^2) \right.\\
+\frac{1}{4}\left(\frac{\partial^2 \Fin_0}{\partial a_2 \partial m}
+\frac{\partial^2 \Fin_0}{\partial m^2} \right)
(\xi_2-\xi_1,\widetilde{\xi}_1)\\
 +\frac{1}{6}\left(
\frac{\partial^2 \Fin_0}{\partial a_2 \partial \log \Lambda}+
\frac{\partial^2 \Fin_0}{\partial m \partial \log \Lambda} \right)
(\xi_2-\xi_1,\alpha) z
+\frac{1}{6}
\frac{\partial^2 \Fin_0}{\partial m \partial \log \Lambda}
(\widetilde{\xi}_1,\alpha) z\\
\left.+\frac{1}{18}
\frac{\partial^2 \Fin_0}{\partial \log \Lambda^2}
(\alpha^2) z^2
+ (\chi(X)+\sigma(X))(\Ain+\frac{1}{4}\frac{\partial^2 \Fin_0}{\partial m^2})
+\sigma(X)(\Bin-\Ain+\frac{1}{8}\frac{\partial^2 \Fin_0}{\partial m^2})\right]
\\
=
(-1)^{\frac{((\xi_2-\xi_1)^2)+(K_X,\xi_2-\xi_1))}{2}+\chi({\cal O}_X)}
\frac{1}{s}
\left(\frac{2s}{\Lambda}\right)
^{((\xi-K_X)^2)+(K_X^2)+3\chi({\cal O}_X)-2(\xi-K_X,\widetilde{\xi}_1)}\\
e^{-(\xi-K_X-\widetilde{\xi}_1,\alpha)sz-s^2 x}
2^{-2\chi({\cal O}_X)-
\frac{((\xi_2-\xi_1)^2)+(\xi_2-\xi_1,\widetilde{\xi}_1)}{2}}\\
\times \exp \left[
\frac{1}{3}\frac{\partial \Fin_0}{\partial \log \Lambda}x
+
\left(\frac{1}{8}\frac{\partial^2 \Fin_0}{\partial a_2^2}
+\frac{1}{4}\frac{\partial^2 \Fin_0}{\partial a_2 \partial m}
+\frac{1}{8}\frac{\partial^2 \Fin_0}{\partial m^2} \right)
((\xi-K_X)^2) \right.\\
-\frac{1}{4}\left(\frac{\partial^2 \Fin_0}{\partial a_2 \partial m}
+\frac{\partial^2 \Fin_0}{\partial a^2} \right)
(\xi-K_X,\widetilde{\xi}_1)\\
 +\frac{1}{6}\left(
\frac{\partial^2 \Fin_0}{\partial a_2 \partial \log \Lambda}+
\frac{\partial^2 \Fin_0}{\partial m \partial \log \Lambda} \right)
(\xi-K_X,\alpha) z
-\frac{1}{6}
\frac{\partial^2 \Fin_0}{\partial a_2 \partial \log \Lambda}
(\widetilde{\xi}_1,\alpha) z\\
\left.+\frac{1}{18}
\frac{\partial^2 \Fin_0}{\partial \log \Lambda^2}
(\alpha^2) z^2
+ (\chi(X)+\sigma(X))(3\Ain-2\Bin)
+(2\chi(X)+3\sigma(X))
(\Bin-\Ain+\frac{1}{8}\frac{\partial^2 \Fin_0}{\partial a^2})\right],
\end{multline*}
\end{NB}

We use 
\(
    \chi(X) = 12\chi_h(X) - (K_X)^2,
\)
\(
    \sigma(X) = (K_X^2) - 8\chi_h(X),
\)
\(
  \xi = \xi_1+\xi_2,
\)
\(
   \widetilde{\xi}_1 = 2\xi_1 - K_X
\)
and
\(
    (\widetilde{\xi}_1^2)=(K_X^2)
\)
to get the assertion,
where the last equality is nothing but the SW-simple type condition.
\begin{NB}
More detail:
\begin{equation*}
\begin{split}
(\xi,\xi-K_X)/{2}+(K_X,\xi_2)
= (\xi,\xi-K_X)/2+(K_X,\xi-\xi_1)
= (\xi,\xi+K_X)/2-(K_X,\widetilde\xi_1+K_X)/2.
\end{split}
\end{equation*}
\begin{equation*}
\begin{split}
& -\frac{(\xi,\xi-K_X)}{2}+(\xi_1,\xi_2)
=
  -\frac{(\xi,\xi-K_X)}{2}+(\frac{\widetilde\xi_1+K_X}2,
  \xi-\frac{\widetilde\xi_1+K_X}2)\\
\\
=\; &
  -\frac{(\xi,\xi-K_X)}{2}+
  (\frac{\widetilde\xi_1+K_X}2, \xi-K_X)
  - (\frac{\widetilde\xi_1+K_X}2, \frac{\widetilde\xi_1-K_X}2)
=
- \frac{(\xi-K_X-\tilde\xi_1,\xi-K_X)}2.
\end{split}
\end{equation*}
\end{NB}%
\end{proof}

\section{Blow-up formula for the partition function}

We start to analyze the partition function in this section. Our
technique is the same as one in \cite{NY1,NY2,NY3}: we study the
blow-up formula of the partition function.

\subsection{Partition function on the blow-up}

Let $p\colon \bp\to \proj^2$ be the blow-up of $\proj^2$ at the origin
$[1:0:0]$. Let $C = p^{-1}([1:0:0])$ be the exceptional divisor.
Let $\bM(r,k,n)$ be the moduli space of framed sheaves $(E,\Phi)$ on
$\bp$ with rank $r$, $c_1(E) = kC$, $(c_2(E)-(r-1)c_1(E)^2/(2r),[\bp])
= n$, where the framing is defined on $p^{-1}(\linf)$. (See
\cite[\S3]{NY1} or \cite[\S3.2]{NY2}.) This is nonsingular
quasi-projective of dimension $2rn$. We normalize as $0\le k <
r$. This is always possible by twisting by a power of $\shfO(C)$.
There is a projective morphism $\widehat\pi\colon \bM(r,k,n)\to
M_0(r,n-k(r-k)/2r)$.

We pull-back the $\C^*\times\C^*$-action on $\proj^2$ to $\bp$. Then
we have an action of $\hT$ on $\bM(r,k,n)$ as in the case of
$M(r,n)$. The action is lifted to the universal sheaf $\cE$ on
$\bp\times\bM(r,k,n)$. The morphism $\widehat\pi$ is
$\hT$-equivariant.

We define $\mu(C)$ as appeared in the definition of Donaldson's
invariants:
\begin{equation*}
  \mu(C) = \left(c_2(\cE)-\frac{r-1}{2r}c_1(\cE)^2\right)/[C]
  \in H^2_{\hT}(\bM(r,k,n)).
\end{equation*}

Over $\bM(r,k,n)$ we have two natural vector bundles, which correspond
to $\Vcal$:
\begin{equation*}
  \Vcal_0 \defeq R^1q_{2*}(\cE\otimes q_1^*\shfO(-\linf)),\quad
  \Vcal_1 \defeq R^1q_{2*}(\cE\otimes q_1^*\shfO(C-\linf)).
\end{equation*}
These are vector bundles of rank $n + k^2/(2r) - k/2$ and $n +
k^2/(2r) + k/2$ respectively thanks to the vanishing of other higher
direct image sheaves, and play a fundamental role in the ADHM type
description of $\bM(r,k,n)$ (see e.g., \cite{perv1}).

Therefore we have two possible choices of {\it matters\/} on
blow-up. Here we take $\Vcal_0$ since the $\Vcal_1$ version can be
reduced to the $\Vcal_0$ one after twisting by the line bundle
$\shfO(C)$. We define the partition function (or better to call the
correlation function since we put the operator $\mu(C)$) as in
\eqref{eq:partition} by
\begin{multline*}
   \bZin_{c_1=kC}(\ve_1,\ve_2,\vec{a},\vec{m};t;\Lambda)
\\
   \defeq 
   \Lambda^{N_f k(\Nc-k)/(2\Nc)}
   \sum_{n=0}^\infty \Lambda^{\bgamma n} \iota_{0*}^{-1}
   \widehat\pi_*\left(e^{t\mu(C)}\cap
     \Eu\left(\Vcal_0\otimes p^*(K_{\C^2}^{1/2})\otimes M\right)
   \cap [\bM(\Nc,k,n)]\right).
\end{multline*}
Here $p^*(K_{\C^2}^{1/2})$ looks a little bit artificial, but is
necessary as in the case of $\C^2$. The square root
$K^{1/2}_{\widehat\C^2}$ does not make sense since $\widehat\C^2$ is
not spin.

As in the case of the original partition function
$\Zin(\ve_1,\ve_2,\vec{a},\vec{m};\Lambda)$, this one also has a
combinatorial expression like \eqref{eq:sum}. We do not write it down
here, we only explain the parameter set for the fixed points. Similar
to the case of $M(r,n)^{\hT}$, it is the set of triples
$(\vec{k},\vec{Y}^1,\vec{Y}^2)$ of an $r$-tuple of integers $\vec{k} =
(k_1,\dots,k_{\Nc})$ and the pair of $r$-tuples of Young diagrams
$\vec{Y}^1 = (Y^1_1,\dots,Y^1_r)$, $\vec{Y}^2 =
(Y^2_1,\dots,Y^2_r)$.
The corresponding framed sheaf is $I_1(k_1C)\oplus\cdots \oplus
I_{r}(k_rC)$, where $I_\alpha$ is an ideal sheaf fixed by the
$\C^*\times\C^*$-action. The blow-up $\widehat{\C^2}$ has two fixed
points $p_1$, $p_2$, and $I_\alpha$ is given by two monomial ideals
with respect to toric coordinates at $p_1$ and $p_2$. In this way,
$I_\alpha$ is parametrized by a pair of Young diagrams
$(Y^1_\alpha,Y^2_\alpha)$.

From this combinatorial description of the fixed point set, we can
write down the correlation function $\bZin_{c_1=kC}$ as sum over the
lattice for $\{\vec{k}\}$ of products of two $\Zin$'s for $p_1$,
$p_2$, and contribution from line bundles $\shfO(k_\alpha C)$. We
postpone to write down the explicit formula until we introduce the
perturbative term in the next subsection.

\subsection{Perturbative term}

The partition function defined above does not behave well in many
aspects. It is more natural to add what is called the {\it
  perturbative term}, which is an explicit function. We recall its
definition in this subsection. We return back to arbitrary $\Nc$,
$N_f$.

Let $\gamma_{\ve_1,\ve_2}(x;\Lambda)$ be the function used to define
the perturbative part of the partition function in \cite[\S E]{NY2}:
\begin{equation}\label{eq:pert_expand}
\begin{split}
   &\gamma_{\ve_1,\ve_2}(x;\Lambda)
%    = \left.\frac{d}{ds}\right|_{s=0}
%    \frac{\Lambda^s}{\Gamma(s)} \sum_{n=0}^\infty c_n
%    \int_0^\infty \frac{dt}t t^{n+s-2} e^{-tx}
% \\
%    =\; & \left.\frac{d}{ds}\right|_{s=0}
%    \left(\frac{\Lambda}{x}\right)^s \sum_{n=0}^\infty c_n
%    x^{2-n} \frac{\Gamma(n+s-2)}{\Gamma(s)}
% \\
%   =\; & 
   =
   \begin{aligned}[t]
     & \frac1{\ve_1\ve_2}\left\{
       - \frac12 x^2 \log\left(\frac{x}{\Lambda}\right) + \frac34 x^2\right\}
     + \frac{\ve_1+\ve_2}{2\ve_1\ve_2} \left\{
       - x \log\left(\frac{x}{\Lambda}\right) + x\right\}
     \\
     & \qquad
     - \frac{\ve_1^2 + \ve_2^2 + 3\ve_1\ve_2}{12\ve_1\ve_2}
     \log\left(\frac{x}{\Lambda}\right)
  + \sum_{n=3}^\infty (n-3)!{c_n (-x)^{2-n}},
   \end{aligned}
\end{split}
\end{equation}
where $c_n$ is defined by
\begin{equation*}
   \frac1{(1 - e^{-\ve_1 t})(1 - e^{-\ve_2 t})}
   = \sum_{n=0}^\infty {c_n} t^{n-2}.
\end{equation*}

\begin{NB}
  The following is copied from below. The original should be erased
  later.
\end{NB}

If we consider the equivariant cohomology group $H^*_{T^2}(\C^2)$ of
$\C^2$ with respect to the two dimensional torus action, we have
\begin{equation*}
  c_n = \int_{\C^2} \Todd_n(\C^2),
\end{equation*}
where $\Todd_n$ is the degree $n$ part of the Todd genus, and
$\int_{\C^2}$ is defined by the localization formula applied to
$\C^2$: $\iota_0^*(\bullet)/{\Eu(T_0\C^2)}$. Here $0$ is the unique
fixed point and $\iota_0$ is the inclusion $\{0\}\to \C^2$.

If $\gamma_0(x;\Lambda) = - \frac12 x^2\log(x/\Lambda) + \frac34 x^2$
denotes the leading part of $\gamma_{\ve_1,\ve_2}(x;\Lambda)$ (`genus
$0$ part'), we have
\begin{equation*}
  \gamma_{\ve_1,\ve_2}(x;\Lambda) = 
  \sum_{n=0}^\infty \int_{\C^2}\Todd_n(\C^2) \gamma_0^{(n)}(x).
\end{equation*}
\begin{NB}
\begin{equation*}
\begin{split}
\gamma_0'(x) &= -x \log(\frac{x}{\Lambda})+x\\
\gamma_0''(x) &= -\log(\frac{x}{\Lambda})\\
\gamma^{(n)}(x) &= (-1)^n (n-3)! x^{2-n}, n \geq 3.
\end{split}
\end{equation*}
\end{NB}

We introduce the function for the matter contribution as
\begin{equation*}
  \begin{split}
    \delta_{\ve_1,\ve_2}(x;\Lambda) & \defeq
  {\gamma}_{\ve_1,\ve_2}(x-\frac{\ve_1+\ve_2}{2};\Lambda)
\\
  & =\frac1{\ve_1\ve_2}\left\{
       - \frac12 x^2 \log\left(\frac{x}{\Lambda}\right) + \frac34 x^2\right\}
     + \frac{\ve_1^2+\ve_2^2}{24}
     \log\left(\frac{x}{\Lambda}\right) + \cdots.
  \end{split}
\end{equation*}
The shift $-(\ve_1+\ve_2)/2$ is identified with $K_{\C^2}/2$, and is
compatible with our shift for the instanton partition function.

We define the full partition function as
\begin{equation*}
  Z(\ve_1,\ve_2,\vec{a},\vec{m};\Lambda)
  \defeq
  \exp\left[
    - \sum_{\vec{\alpha}\in\Delta}\gamma_{\ve_1,\ve_2}(
    \langle\vec{a},\vec{\alpha}\rangle;\Lambda)
    + \sum_{f, \alpha} \delta_{\ve_1,\ve_2}(a_\alpha+m_f;\Lambda)
  \right]
  \Zin(\ve_1,\ve_2,\vec{a},\vec{m};\Lambda).
\end{equation*}
\begin{NB}
Therefore
\begin{equation*}
  \begin{split}
    & F_0^{\mathrm{pert}} = 
    \begin{aligned}[t]
    & \sum_{\vec{\alpha}\in\Delta} \left\{
      \frac12 \langle\vec{a},\vec{\alpha}\rangle^2
      \log\left(\frac{\langle\vec{a},\vec{\alpha}\rangle}{\Lambda}\right)
      - \frac34 \langle\vec{a},\vec{\alpha}\rangle^2 \right\}
    \\
    & \quad +
    \sum_{f,\alpha}\left\{ - \frac12 (a_\alpha+m_f)^2
      \log\left(\frac{a_\alpha+m_f}{\Lambda}\right) + \frac34
      (a_\alpha+m_f)^2\right\},
    \end{aligned}
    \\
    & H^{\mathrm{pert}} = \sum_{\vec{\alpha}\in\Delta} \left\{ 
      - \frac12 \langle\vec{a},\vec{\alpha}\rangle
      \log\left(\frac{\langle\vec{a},\vec{\alpha}\rangle}{\Lambda}\right)
      + \frac12 \langle\vec{a},\vec{\alpha}\rangle\right\}
    = -\pi\sqrt{-1}\langle\vec{a},\rho\rangle,
\\
    & A^{\mathrm{pert}} = \frac12 \sum_{\vec{\alpha}\in\Delta^+}
    \log\left(\frac{\sqrt{-1}\langle\vec{a},\vec{\alpha}\rangle}
      {\Lambda}\right),
\\
    & B^{\mathrm{pert}} =
    \frac12 \sum_{\vec{\alpha}\in\Delta^+}
    \log\left(\frac{\sqrt{-1}\langle\vec{a},\vec{\alpha}\rangle}
      {\Lambda}\right)
    + \frac18 \sum_{f,\alpha}
    \log\left(\frac{a_\alpha+m_f}{\Lambda}\right).
  \end{split}
\end{equation*}
\end{NB}%

Let us expand the perturbative part as
\begin{multline*}
    - \sum_{\vec{\alpha}\in\Delta}\gamma_{\ve_1,\ve_2}(
    \langle\vec{a},\vec{\alpha}\rangle;\Lambda)
    + \sum_{f, \alpha} \delta_{\ve_1,\ve_2}(a_\alpha+m_f;\Lambda)
\\
    = \frac1{\ve_1\ve_2}
    \left(
      F_0^{\mathrm{pert}}
      + (\ve_1+\ve_2) H^{\mathrm{pert}}
      + \ve_1\ve_2 A^{\mathrm{pert}} + \frac{\ve_1^2+\ve_2^2}3 B^{\mathrm{pert}}
      + \cdots
      \right)
\end{multline*}
as in the instanton part.

For a future reference, we give explicit formulas for some terms when
$\Nc=2$, $N_f=1$:
{\allowdisplaybreaks
\begin{equation}\label{eq:pert}
\begin{split}
    & H^{\mathrm{pert}} = \pi\sqrt{-1}a,
\\
    & \pd{F_0^{\mathrm{pert}}}{\log\Lambda} 
    = -3 a^2 + m^2,
\\
   & \frac{\partial^2 F_0^{\mathrm{pert}}}{\partial (\log\Lambda)^2} = 0,
\\
   &
   -\frac1{\bgamma}\frac{\partial^2 F_0^{\mathrm{pert}}}
   {\partial\log\Lambda\partial a} = 2a,
\\
   &  \frac{\partial^2 F_0^{\mathrm{pert}}}{\partial a^2}
   = 8 \log\frac{-2\sqrt{-1}a}{\Lambda}
  - \log\frac{(a+m)(-a + m)}{\Lambda^2},
\\
   &  \frac{\partial^2 F_0^{\mathrm{pert}}}{\partial a\partial m}
   = \log\left(\frac{-a+m}{\Lambda}\right)
   - \log\left(\frac{a+m}{\Lambda}\right),
\\
   & \frac{\partial^2 F_0^{\mathrm{pert}}}{\partial m^2}
   = -\log\left(\frac{-a+m}{\Lambda}\right)
   - \log\left(\frac{a+m}{\Lambda}\right),
\\
   & \frac{\partial^2 F_0^{\mathrm{pert}}}{\partial m\partial\log\Lambda}
   = 2m,
\\
   & A^{\mathrm{pert}}
   = \frac12 \log\left( \frac{-2\sqrt{-1}a}\Lambda\right),
\\
   & B^{\mathrm{pert}}
   = \frac12 \log\left( \frac{-2\sqrt{-1}a}\Lambda\right)
   + \frac18 \log\left(\frac{(m - a)(m+a)}{\Lambda^2}
     \right).
  \end{split}
\end{equation}
}
\begin{NB}
  A little more detail:
  \begin{equation*}
    \begin{split}
      &\pd{F_0^{\mathrm{pert}}}{\log\Lambda} 
    = - (a_2 - a_1)^2 + \frac12 (a_1 + m)^2 + \frac12 (a_2 + m)^2,
\\    
   & \pd{F_0^{\mathrm{pert}}}{m}
   \begin{aligned}[t]
   &= 
   -(a_1+m)\log\left(\frac{a_1+m}{\Lambda}\right)
   + (a_1+m)
   -(a_2+m)\log\left(\frac{a_2+m}{\Lambda}\right)
   + (a_2+m)
\\
   &=
   -(-a+m)\log\left(\frac{-a+m}{\Lambda}\right)
   -(a+m)\log\left(\frac{a+m}{\Lambda}\right)
   + 2m,
   \end{aligned}
    \end{split}
  \end{equation*}
It is also useful to note
\begin{equation*}
   \left(\pd{u}{a}\right)^{\mathrm{pert}}  =      
   -\frac1{\bgamma}\frac{\partial^2 F_0^{\mathrm{pert}}}
   {\partial\log\Lambda\partial a} = 2a,
   \qquad
    \left(\frac{\pi}{\omega}\right)^{\mathrm{pert}} 
   = - \frac1{2\sqrt{-1}} \left(\pd{u}{a}\right)^{\mathrm{pert}}
   = \sqrt{-1}a,
\end{equation*}
\end{NB}

\subsection{Blow-up formula}

Similarly we put the perturbative part to the correlation function on
the blow-up as
\begin{multline*}
  \bZ_{c_1=kC}(\ve_1,\ve_2,\vec{a},\vec{m};t;\Lambda)
\\
  \defeq
  \exp\left[
    - \sum_{\vec{\alpha}\in\Delta}\gamma_{\ve_1,\ve_2}(
    \langle\vec{a},\vec{\alpha}\rangle;\Lambda)
    + \sum_{f, \alpha} \delta_{\ve_1,\ve_2}(a_\alpha+m_f;\Lambda)
  \right]
  \bZin(\ve_1,\ve_2,\vec{a},\vec{m};t;\Lambda).
\end{multline*}
As in \cite[\S4.4]{NY2}, we get the following
\begin{multline}\label{eq:blow-up1}
    \bZ_{c_1=kC}(\ve_1,\ve_2,\vec{a},\vec{m};t;\Lambda)
\\
    = 
    \exp \left[
      \frac{t}{\bgamma} \left(\left(
        \frac{\Nc}{12}(2\Nc + N_f-2) % (\ve_1+\ve_2)
        + \frac{N_f}{2}
        \frac{k^2}{\Nc}\right)(\ve_1+\ve_2)
      + %\frac1{\bgamma}
       (\frac{\Nc}2-k)\sum_f m_f
      \right)
      \right]
\\
    \times 
    \sum_{\vec{k}}
    \begin{aligned}[t]
    & Z\left(\ve_1,\ve_2-\ve_1,\vec{a}+\ve_1\vec{k},
      \vec{m}+\left(\frac{k}\Nc - \frac12\right)\ve_1\vec{e};
    \Lambda e^{t\ve_1/\bgamma}\right)
    \\
    &\quad\times 
    Z\left(\ve_1-\ve_2,\ve_2,\vec{a}+\ve_2\vec{k},
      \vec{m}+\left(\frac{k}\Nc - \frac12\right)\ve_2\vec{e};
      \Lambda e^{t\ve_2/\bgamma}\right)
    \end{aligned}
\end{multline}
by analyzing the fixed points in $\bM(\Nc,k,n)$ and then using a
difference equation satisfied by the perturbative term.
Here $\vec{k}$ runs over
\begin{equation*}
  \left\{ \vec{k} = (k_1,\dots, k_{\Nc})\in\Q^{\Nc}\left|\,
    \sum k_\alpha = 0, k_\alpha \equiv - \frac{k}{\Nc} \bmod\Z\right.\right\}.
\end{equation*}
This is slightly different from the $\vec{k}$ which appeared in the
parametrization of the fixed point set $\bM(r,k,n)^{\hT}$: We
subtract $k/\Nc$ from each factor so that the sum of entries becomes
$0$.

The complete proof will be given in \cite{NY4}, but is a straightforward modification of the original one.

In \cite[Th.~2.1]{perv3} we proved the following vanishing theorem:
\begin{equation}\label{eq:vanish}
 %\lim_{\ve_1,\ve_2\to 0}
  \frac{\bZ_{c_1=0}(\ve_1,\ve_2,\vec{a},\vec{m};t;\Lambda)}
  {Z(\ve_1,\ve_2,\vec{a},\vec{m};\Lambda)}
  = 1 + O(t^{\max(\Nc+1,2\Nc-N_f)}).
\end{equation}
This is a generalization of the vanishing theorem for the pure theory
($N_f = 0$), which was proved by the dimension counting argument in
\cite{NY1}. The proof of this generalization requires the theory of
perverse coherent sheaves in \cite{perv1,perv2,perv3}, but there is a
similar flavor with the original one. In particular, the exponent
$2\Nc-N_f$, which is written $\bgamma$ here, comes from the formula
for $\deg\left(\Eu(\Vcal\otimes K_{\C^2}^{1/2}\otimes M) \cap
  [M(\Nc,n)]\right) = (2\Nc - N_f) n = \gamma n$.

From \eqref{eq:blow-up1} together with \eqref{eq:vanish}, we can prove
\begin{enumerate}
\item $\ve_1\ve_2\log Z(\ve_1,\ve_2,\vec{a},\vec{m};\Lambda)$ is
  regular at $\ve_1, \ve_2 = 0$.
\item The instanton part satisfies
\(
%  \begin{equation*}
   \Zin(\ve_1,-2\ve_1,\vec{a},\vec{m};\Lambda)
   =  \Zin(2\ve_1,-\ve_1,\vec{a},\vec{m};\Lambda).
\)    
%  \end{equation*}
\end{enumerate}
The proofs of these assertions are exactly as in \cite[\S5.2]{NY2} and
\cite[Lem.~7.1]{NY1} respectively. They will be reproduced in
\cite{NY4} for this version, and are not repeated here.

We expand the partition function as in \eqref{eq:expand}:
\begin{multline*}
  \ve_1\ve_2\log Z(\ve_1,\ve_2,\vec{a},\vec{m};\Lambda)
\\
  = F_0(\vec{a},\vec{m};\Lambda) + (\ve_1+\ve_2) H(\vec{a},\vec{m};\Lambda)
  + \ve_1\ve_2 A(\vec{a},\vec{m};\Lambda) 
  + \frac{\ve_1^2+\ve_2^2}3 B(\vec{a},\vec{m};\Lambda) + \cdots.
\end{multline*}
From the symmetry property (2) of $Z$, we see that $H$ comes only from
the perturbative part. This is already explained above. As in
\cite[\S5.3]{NY2} (which has the sign mistake) we have
\begin{equation*}
  H(\vec{a},\vec{m};\Lambda)
  = - \pi\sqrt{-1} \langle \vec{a},\rho\rangle,
\end{equation*}
where $\rho$ is one half of the sum of the positive roots.

As in \cite[\S6]{NY2} we can take the limit of \eqref{eq:blow-up1} to
get
\begin{multline}\label{eq:blow-up2}
  \lim_{\ve_1,\ve_2\to 0}
  \frac{\bZ_{c_1=kC}(\ve_1,\ve_2,\vec{a},\vec{m};t;\Lambda)}
  {Z(\ve_1,\ve_2,\vec{a},\vec{m};\Lambda)}
\\
  =
  \exp\Biggl[
  \begin{aligned}[t]
    & -\frac12
    \sum_{f,f'} \frac{\partial^2 F_0}{\partial m_f\partial m_{f'}}
    \left(\frac{k}{\Nc}-\frac12\right)^2
    + A - B
\\
    & 
    - \frac{t}{\bgamma}\left\{
      \sum_f 
      \left(\frac{k}{\Nc} - \frac12\right)
      \left(
      \frac{\partial^2 F_0}{\partial\log\Lambda\partial m_f}
      - {m_f}{\Nc}\right)
      \right\}
% \\
%     & 
    - \frac1{\bgamma^2}
    \frac{\partial^2 F_0}{\partial(\log\Lambda)^2}
    \frac{t^2}2
    \Biggr]
  \end{aligned}
\\
  \times\Theta_{E_k}\left(
    -\frac1{2\pi\sqrt{-1}}
    \frac{\partial^2 F_0}{\partial \vec{a}\partial m_f}
      \left(\frac{k}{\Nc} - \frac12\right)
      -
      \frac{t}{\bgamma}
      \frac1{2\pi\sqrt{-1}}
      \frac{\partial^2 F_0}{\partial\vec{a}\partial\log\Lambda}
      \Biggm|\tau \right),
\end{multline}
where $\Theta_{E_k}$ is the Riemann theta function with the
characteristic $E_k$ as in \cite[\S B]{NY2}. The period matrix $\tau$
is given by
\[
  \tau_{kl} = - \frac1{2\pi\sqrt{-1}}\frac{\partial^2 F_0}{\partial
    a^k\partial a^l}.
\]
Here we change the coordinate from $(a_2,\dots,a_r)$ to the root
system coordinate defined as $\vec{a} = \sum a^i \alpha_i^\vee$ by
simple coroots
\(
   \alpha_i^\vee = (0,\dots, 0, \overset{i}{1},
   \overset{i+1}{-1},0,\dots,0)
\), 
$i=1,\dots,\Nc$. 

In the $\Nc=2$ case, we have $a^1 = a_1 = -a_2 = -a$.
Therefore we need to note
\(
  \partial/\partial \vec{a} = - \partial/\partial a
\)
when we use \eqref{eq:blow-up2}.

For a later purpose, we need another vanishing for $c_1\neq 0$:
\begin{equation}\label{eq:vanish2}
  \bZ_{c_1 = kC}(\ve_1,\ve_2,\vec{a},\vec{m};t;\Lambda)
  = O(t^{k(\Nc-k)})
\end{equation}
for $0 < k < \Nc$. This is \cite[Th.~2.5]{perv3}. This is again proved
by a version of the dimension counting argument, and $k(r-k)$ appears as
the dimension of the Grassmannian of $k$-planes in $\C^r$.

\subsection{Lower terms}

We assume $\Nc = 2$, $N_f = 1$ hereafter. Therefore $\bgamma = 3$.

Let us define a function $u$ by
\begin{equation}\label{eq:u}
  u \defeq -\frac1{\bgamma}
  \left(\pd{F_0}{\log\Lambda} - m^2
  \right)
  = a^2 -\frac1{\bgamma} \pd{\Fin_0}{\log\Lambda}
  .
\end{equation}
In the formula in \thmref{thm:partition}, this appears as the
coefficient of $x$. Note that $x$ is a variable for the $\mu$-class of
the point. Its gauge theoretic interpretation is already implicitly
used in the proof of \thmref{thm:partition}, but becomes clear if we
look again the partition function as follows: Consider
\begin{equation}\label{eq:corr}
  \frac{
  \sum_{n=0}^\infty \Lambda^{\bgamma n}
  \iota_{0*}^{-1} \pi_* 
  \left(
    \ch_2(\cE)/[0]\cap
    \Eu(\Vcal\otimes K_{\C^2}^{1/2}\otimes M)
    \cap [M(2,n)]\right)
  }{
  \sum_{n=0}^\infty \Lambda^{\bgamma n}
  \iota_{0*}^{-1} \pi_* 
  \left(
    \Eu(\Vcal\otimes K_{\C^2}^{1/2}\otimes M)
    \cap [M(2,n)]\right)},
\end{equation}
where $[0]$ is the equivariant homology class of the origin. The
denominator is nothing but $\Zin(\ve_1,\ve_2,a,m;\Lambda)$, and we
have $\ch_2(\cE)/[0] = a^2 - n\ve_1\ve_2$. Therefore this is equal to
\begin{equation*}
  a^2 - \frac{\ve_1\ve_2}{\bgamma}\pd{}{\log\Lambda}
  \log\Zin(\ve_1,\ve_2,a,m;\Lambda).
\end{equation*}
From the expansion \eqref{eq:expand}, this converges to \eqref{eq:u}
at $\ve_1$, $\ve_2 = 0$. In other words, the function $u$ is the limit
of \eqref{eq:corr} at $\ve_1$, $\ve_2 = 0$.

This can be generalized as follows. A power $u^p$ ($p > 0$) is the
limit of \eqref{eq:corr} where $\ch_2(\cE)/[0]$ is replace by its
$p^{\mathrm{th}}$ power, since terms with higher derivatives 
of $\Fin_0$ disappear at $\ve_1$, $\ve_2 = 0$.
\begin{NB}
  If we replace $\ch_2(\cE)/[0]$ by its $p^{\mathrm{th}}$ power, we get
  \begin{equation*}
    \Zin(\ve_1,\ve_2,a,m;\Lambda)^{-1}
    \sum_{k=0}^p\binom{p}{k} a^{2k}
    \left(-\frac{\ve_1\ve_2}\gamma\pd{}{\log\Lambda}\right)^{p-k}
    \Zin(\ve_1,\ve_2,a,m;\Lambda).
  \end{equation*}
  If we set $\mathcal F \defeq \ve_1\ve_2 \log\Zin$, we also get
  higher derivatives of $\mathcal F$. But they always come with
  $\ve_1\ve_2$ and disappear in the limit $\ve_1$, $\ve_2\to
  0$. Therefore we differentiate only $\Zin$, and get
  \begin{equation*}
    \sum_{k=0}^p\binom{p}{k} a^{2k}
    \left(-\frac1\gamma\pd{\Fin_0}{\log\Lambda}\right)^{p-k}
    = \left(a^2 - \frac1\gamma\pd{\Fin_0}{\log\Lambda}\right)^p
  \end{equation*}
at $\ve_1$, $\ve_2 = 0$.
\end{NB}

In \cite[Th.~2.6]{perv3} a general structural result of the blow-up
formula was proved. An integral
\begin{equation*}
\iota_{0*}^{-1}
   \widehat\pi_*\left(e^{t\mu(C)}\cap
     \Eu\left(\Vcal_0\otimes p^*(K_{\C^2}^{1/2})\otimes M\right)
   \cap [\bM(2,k,n)]\right)
\end{equation*}
appearing in the correlation function on the blow-up, can be written
as a linear combination of
\begin{equation*}
\iota_{0*}^{-1}
   \widehat\pi_*\left(
     \left(\ch_2(\cE)/[0]\right)^p
     \cap
     \Eu\left(\Vcal\otimes K_{\C^2}^{1/2}\otimes M\right)
   \cap [M(2,n-k(2-k)/4 - j)]\right)
\end{equation*}
for various $p$, $j\ge 0$, where coefficients are in
$\C[m,\ve_1,\ve_2][[t]]$. (In higher rank cases, we also need higher
Chern classes.) Moreover, the coefficients depend on $p$, $j$ (and
$k$), but not on $n$. Therefore the ratio
\[
  \frac{\bZ_{c_1=C}(\ve_1,\ve_2,a,m;t;\Lambda)}
  {Z(\ve_1,\ve_2,a,m;\Lambda)}
\]
is a formal power series in $t$ with coefficients in
$\C[m,\ve_1,\ve_2,u,\Lambda]$. Here the finiteness as power series in
$u$, $\Lambda$ comes from the cohomological degree reason.

In particular, when we expand the ratio in $t$, we only get  finitely
many powers of $\Lambda$, and the coefficients can be computed from
the integrals over finitely many moduli spaces.
By using the combinatorial expressions of the partition and
correlation functions, these are really possible to compute. We use a
Maple program to get
\begin{align}
  &
  \begin{aligned}[t]
  & \lim_{\ve_1,\ve_2\to 0}
  \frac{\bZ_{c_1=C}(\ve_1,\ve_2,a,m;t;\Lambda)}
  {Z(\ve_1,\ve_2,a,m;\Lambda)}
\\
  &\qquad%\qquad
  = -\Lambda t - \frac{t^3}{3!}{\Lambda u}
  - \frac{t^5}{5!}\Lambda\left(u^2 + 2m \Lambda^3\right)
  - \frac{t^7}{7!}\Lambda\left(u^3 + 6u m \Lambda^3 + 6\Lambda^6 \right)
  + O(t^9).
  \end{aligned}
  \label{eq:coeff}
\end{align}
In fact, we have computed the ratio, before taking
$\lim_{\ve_1,\ve_2\to 0}$, but imposing $\ve_1 + \ve_2 = 0$
instead. Otherwise, the program runs very slow.

Let us check the cohomological degree, which we briefly mentioned
above. We have $\deg \Lambda = \deg m = 1$, $\deg u = 2$. Then the
coefficient of $t^n$ has  degree $n$.

\section{Seiberg-Witten curves}\label{sec:curve}

In this section we determine coefficients $F_0$, $A$, $B$ of $Z$ in
terms of certain `periods' of a family of elliptic curves, called the
{\it Seiberg-Witten curves}.
Our derivation of the Seiberg-Witten curves is analogous to 
Fintushel-Stern's method \cite{FS}: They described (in fact, before
Seiberg-Witten's work) that the blow-up formula of Donaldson
invariants is given by elliptic integrals, associated with cubic
curves of Weierstrass form. And the moduli parameter $u$ for the
cubics is coupled to the $\mu$-class of the point.
We define $u$, and derive cubic curves in the same way by using the
partition function $Z$ instead of Donaldson invariants. The cubic
curves are the Seiberg-Witten curves for the theory with one fundamental
matter. In fact, our derivation is much simpler, as we already see the
theta function in the blow-up formula\footnote{In higher rank cases,
  the story becomes much more complicated, as we need to show that the
  theta function is associated with a hyper-elliptic curve. See
  \cite{NY4}.}

% See \cite[Introduction]{NY2} for a brief explanation of the original
% physical motivation for mathematicians.

\subsection{Elliptic curve}

\begin{NB}
Since $a_1 + a_2 = 0$, we set $a \defeq a_2$ and use it as a
fundamental variable. Since the partition function is invariant under
the exchange $a_1\leftrightarrow a_2$, we may consider $a^2$ as a
variable.

Our $\vec{a}$ is defined as $\vec{a} = \sum a^i \alpha_i^\vee$ by
simple coroots
\(
   \alpha_i^\vee = (0,\dots, 0, \overset{i}{1},
   \overset{i+1}{-1},0,\dots,0)
\)
$i=1,\dots,\Nc$.
\end{NB}

As before, we set
\begin{equation}\label{eq:tau}
  \tau \defeq -\frac1{2\pi\sqrt{-1}}\frac{\partial^2 F_0}{\partial a^2}
\end{equation}
and the corresponding elliptic curve $E_\tau$ with the period $\tau$.
\begin{NB}
  This is a tentative definition. My $F_0$ is a homogeneous version.
\end{NB}%
\begin{NB}
Since
\begin{equation*}
  \tau = \frac1{2\pi\sqrt{-1}}
  \left( -8 \log\frac{\sqrt{-1}(a_1-a_2)}{\Lambda}
  + \log\frac{(a_1+m)(a_2 + m)}{\Lambda} 
  \right)
  + O(\Lambda),
\end{equation*}
we have $\Ima\tau \gg 0$ if $\sqrt{-1}a_1$, $\sqrt{-1}a_2$,
$\sqrt{-1}m$, $\Lambda\in\R$, and 
\(
  \sqrt{-1}(a_1 - a_2) \gg \sqrt{-1}m \gg \Lambda > 0.
\)
Therefore $E_\tau$ is nonsingular.
\end{NB}%
\begin{NB}
  The second term
  \begin{equation*}
    \log\frac{(a_1+m)(a_2 + m)}{\Lambda} 
    = \log \frac{-(\sqrt{-1}m)^2 + (\sqrt{-1}a)^2}{\Lambda}
  \end{equation*}
  can be absorbed into the first term.
\end{NB}%
We put
\begin{equation*}
  q = \exp(\pi\sqrt{-1}\tau) 
  = \exp\left(-\frac12\frac{\partial^2 F_0}{\partial a^2}\right).
\end{equation*}
\begin{NB}
  This $q$ is Kota's $\sqrt{q}$.
\end{NB}

We have defined $u$ in \eqref{eq:u}. Since
\begin{equation*}
  u = a^2 + O(\Lambda),
\end{equation*}
we can take $u$ as a variable instead of $a^2$ if $\Lambda$ is
sufficiently small. This viewpoint will be taken later since the
curve $E_\tau$ will be explicitly given as a cubic curve so that
its coefficients are polynomials in $u$. This $u$ is the coordinate of
what Seiberg-Witten called the {\it $u$-plane}, a family of vacuum
states.

We realize the elliptic curve $E_\tau$ as $\C/(\Z\omega + \Z\omega')=
\C/(\Z\omega + \Z\omega\tau)$, where
\begin{equation}\label{eq:duda}
  \omega \defeq - 2\pi\sqrt{-1}\left(\pd{u}{a}\right)^{-1}
  = \left( \frac1{2\pi\sqrt{-1}}\frac1{\bgamma}
    \frac{\partial^2 F_0}{\partial a\partial\log\Lambda}
    \right)^{-1}.
    \begin{NB}
      \qquad
      \frac{\pi}\omega
      = \frac{\sqrt{-1}}2\pd{u}a
    \end{NB}
\end{equation}
\begin{NB}
  Careful on the sign. This is 
  \begin{equation*}
    -\left( \frac1{2\pi\sqrt{-1}}\frac1{\bgamma}
    \frac{\partial^2 F_0}{\partial \vec{a}\partial\log\Lambda}
    \right)^{-1}
  \end{equation*}
\end{NB}%
Using the Weierstrass $\wp$-function associated with
$\Z\omega+\Z\omega'$, we can realize $E_\tau$ in the Weierstrass form:
\begin{equation*}
  y^2 = 4x^3 - g_2 x - g_3.
\end{equation*}
\begin{NB}
Therefore
\begin{equation*}
  \int_{A} \frac{dx}{\sqrt{4x^3 - g_2 x - g_3}}
   = 2\int_{\infty}^{e_1} \frac{dx}{\sqrt{4x^3 - g_2 x - g_3}}
   = 2\pi\sqrt{-1}\left(\pd{u}{a}\right)^{-1}
\end{equation*}
\end{NB}

Then the blow-up formula for the $c_1 = C$ case \eqref{eq:blow-up2}
can be re-written in terms of the $\sigma$-function:
\begin{equation*}
  \lim_{\ve_1,\ve_2\to 0}
  \frac{\bZ_{c_1=C}(\ve_1,\ve_2,a,m;t;\Lambda)}
  {Z(\ve_1,\ve_2,a,m;\Lambda)}
  = - \exp\left[A-B
    -t^2\left\{
    \frac1{2\bgamma^2}\frac{\partial^2 F_0}{\partial\log\Lambda^2}
    +\frac{\pi^2}{6\omega^2} E_2(\tau)\right\}
    \right]
%  e^{ut^2/6}
  \sigma(t) %\Lambda
%  \times
  \frac{\theta_{11}'(0)}{\omega}.
\end{equation*}
\begin{NB}
  I still need to check the sign. It should come from $H$. Since we
  made a mistake on $H$, the theta characteristic must be $-\theta_{11}$.
\end{NB}%
\begin{NB}
  \begin{equation*}
    \sigma(t) = \omega \exp({\nicefrac{\eta}\omega t^2})
    \frac{\theta_{11}(\nicefrac{t}\omega)}{\theta'_{11}(0)}
    = \omega \exp({\nicefrac{\pi^2}{6\omega^2}E_2(\tau) t^2})
    \frac{\theta_{11}(\nicefrac{t}\omega)}{\theta'_{11}(0)}.
  \end{equation*}
\end{NB}

We compare the expansion
\begin{multline*}
  e^{-Tt^2}\sigma(t)
  = t - {T} t^3 + \left(\frac{T^2}2 - \frac{g_2}{2\cdot 5!}\right)
  t^5
  + \left(
    -\frac{T^3}{3!}
    + \frac{T g_2}{2\cdot 5!}
    -\frac{6g_3}{7!}
  \right)
  t^7 + \cdots
\end{multline*}
\begin{NB}
  \begin{equation*}
    \sigma(t) = 
    t - \frac{g_2}{2}\frac{t^5}{5!} - 6g_3\frac{t^7}{7!} + \cdots.
  \end{equation*}
  \end{NB}%
with our computation of lower terms of the blow-up formula
\eqref{eq:coeff}. We get
\begin{align}
   & \exp(A-B) \frac{\theta_{11}'(0)}{\omega} = \Lambda, \label{eq:AB}
\\
   & \frac1{2\bgamma^2}\frac{\partial^2 F_0}{\partial\log\Lambda^2}
    +\frac{\pi^2}{6\omega^2} E_2(\tau) = - \frac{u}{6}, \label{eq:d^2F}
\\
   & g_2 = \frac43 u^2 - 4 m \Lambda^3, \label{eq:g2}
\\
   & g_3 = -\frac8{27}u^3 + \frac43 u m \Lambda^3 - \Lambda^6. \label{eq:g3}
\end{align}
\begin{NB}
This is compatible with results in the pure theory, we have
  \(
    \exp(A - B) = \theta_{01}(0)^{-1}.
  \)
  On the other hand, we have
  \(
    \theta_{11}'(0) = -\pi\theta_{01}(0)\theta_{10}(0)\theta_{00}(0),
  \)
  and
  \(
    \omega = - 2\pi\sqrt{-1}\left(\pd{u}a\right)^{-1}
    = - \frac{\pi}\Lambda \theta_{00}(0)\theta_{10}(0),
  \)
  therefore
  \(
    \frac{\theta_{11}'(0)}{\omega}
    = \Lambda \theta_{01}(0).
  \)

  \eqref{eq:d^2F} is also compatible with what we know for the pure
  theory.
\end{NB}
In particular, the curve $E_\tau$ has the Weierstrass form
\begin{equation*}\label{eq:SW}
  y^2 = 4x^3 - (\frac43 u^2 - 4 m \Lambda^3) x 
   +\frac8{27}u^3 - \frac43 u m \Lambda^3 + \Lambda^6.
\end{equation*}
Replacing $x$ by $x+u/3$, we get
\begin{equation}\label{eq:SW2}
  y^2 = 4x^2(x + u) + 4m\Lambda^3 x + \Lambda^6.
\end{equation}
This is nothing but the Seiberg-Witten curve for the theory with one
fundamental matter, determined at first in \cite{SW2}. There is a  vast
literature on this curve. For example, \cite{AGMZ} was useful for the
authors.
\begin{NB}
  This is isomorphic to
  \begin{equation*}
    y^2 = (z^2 - u)^2 - 4\Lambda^3 (z+m).
  \end{equation*}
We take a solution $z = \alpha$ of the $\mathrm{RHS} = 0$, and set
$z = 1/X + \alpha$, $y = Y/X^2$. Then we have a cubic curve
in $X$, $Y$. We then make the quadratic term to vanish.
\end{NB}%

The discriminant $\Delta = g_2^3 - 27g_3^2$ is given by
\begin{NB}
  $\Delta = 16\left(\nicefrac{\pi}{\omega}\right)^{12}
  \left(\nicefrac{\theta_{11}'(0)}\pi\right)^8$.
\end{NB}%
\begin{equation}\label{eq:disc}
  \Delta = -\Lambda^6(16 u^3 - 16 u^2 m^2 - 72 um\Lambda^3 +64m^3\Lambda^3
  + 27\Lambda^6).
\end{equation}
\begin{NB}
  If $\sqrt{-1}a \gg \sqrt{-1}m \gg \Lambda > 0$, we have $\Delta >
  0$, as $u\approx a^2$ is negative.

  We have
  \begin{equation*}
    \left.\Delta\right|_{\Lambda=0}
    = - 16 \Lambda^6 u^2 (u - m^2).
  \end{equation*}
  Therefore $\Delta$ vanishes from some point around $u\approx m^2$
  when we move $u = -\infty$ to $u = 0$.
\end{NB}

Let $e_1 - u/3$, $e_2 - u/3$, $e_3 - u/3$ be the solutions of the right
hand side of \eqref{eq:SW2} $= 0$.
\begin{NB}
The right hand side is
$4(x - e_1 + u/3)(x - e_2 + u/3)(x - e_3 + u/3)$.
\end{NB}%
We number them as in \cite[p.361]{Bateman}:
\begin{equation}
  \label{eq:e_}
\begin{gathered}
  e_1 = \frac13 \left(\frac{\pi}{\omega}\right)^2 
  (\theta_{00}^4 + \theta_{01}^4),
\quad
  e_2 = \frac13 \left(\frac{\pi}{\omega}\right)^2 
  (\theta_{10}^4 - \theta_{01}^4),
\\
  e_3 = -\frac13  \left(\frac{\pi}{\omega}\right)^2 
  (\theta_{10}^4 + \theta_{00}^4).
\end{gathered}
\end{equation}
\begin{NB}
  When $q = e^{\pi\sqrt{-1}\tau} \to 0$, $\theta_{00}\to 1$,
  $\theta_{01}\to 1$, $\theta_{10}\to 0$. Therefore $e_1\to 2/3\,
  (\pi/\omega)^2$, $e_2$, $e_3\to -1/3\, (\pi/\omega)^2$.
  If we further expand in $\Lambda$, we have
  $\pi/\omega = -\sqrt{-1}a + O(\Lambda)$, and hence
\(
   e_1 \to -2/3\, a^2, e_2, e_3 \to 1/3\, a^2.
\)
If $\sqrt{-1}a\in\R$ as before, we have $e_1 > 0 > e_2, e_3$.
Hence when $u\to \infty$, $e_1 - u/3 \approx -u$ goes to $\infty$ while
$e_2 - u/3$, $e_3 - u/3$ goes to $0$.
\end{NB}%

We can revert the role of $u$ and $a$. We consider $u$ as a variable
and introduce the cubic curve \eqref{eq:SW2}. We define the function
$a$ by the formula \eqref{eq:duda}. Since $da/du\neq 0$, we can
consider $a$ (or $a^2$) as a variable.
Then we define $F_0$ by \eqref{eq:d^2F}.
\begin{NB}
  How we fix the constant ?
\end{NB}%

The blow-up formula is further simplified as
\begin{equation*}
    \lim_{\ve_1,\ve_2\to 0}
  \frac{\bZ_{c_1=C}(\ve_1,\ve_2,a,m;t;\Lambda)}
  {Z(\ve_1,\ve_2,a,m;\Lambda)}
  = - e^{ut^2/6}\sigma(t)\Lambda.
\end{equation*}
(cf.\ \cite[\S6.3]{NY2}.) This is the form of Fintushel-Stern's
blow-up formula for the Donaldson invariants if we replace the curve
appropriately, i.e., the Seiberg-Witten curve for the {\it pure\/}
theory.

\subsection{Seiberg-Witten differential}

In this subsection, we write $a$ as an integral of  a certain
differential form $dS$ on the Seiberg-Witten curve. It is the usual
framework to relate the Seiberg-Witten curve and the partition
function. This is not necessary for our computation of derivatives of
$F_0$, but we explain it for  completeness.

Let $Q(x)$ be the right hand side of \eqref{eq:SW2}.
We set
\begin{equation*}
  dS \defeq 
  \frac{Q'(x) dx}{4x y} 
  \begin{NB}
    = \frac{dx}y
    \left(3 x + 2u + \frac{m\Lambda^3}x\right)
  \end{NB}%
.
\end{equation*}
We differentiate \eqref{eq:SW2} to get
\begin{equation*}
  2y dy = Q'(x) dx.
  %\left(12 x^2 + 8x u + 4m\Lambda^3\right) dx.
\end{equation*}
Therefore
\begin{NB}
\begin{equation*}
  \frac{2dy}x
  = \frac{Q'(x) dx}{xy}
%  \left(12 x + 8 u + \frac{4m\Lambda^3}x\right)
%  \left(12 x - \frac1x \left( \frac43u^2 - 4m\Lambda^3\right)\right),
\end{equation*}
and hence
\end{NB}%
\begin{equation*}
  dS = \frac{dy}{2x}.
\end{equation*}

We differentiate \eqref{eq:SW2} by $u$ after setting $y$ to be
constant:
\begin{equation*}
  0 = Q'(x) % \left(12 x^2 + 8x u + 4m\Lambda^3\right)
  \left.\pd{x}{u}\right|_{y=\mathrm{const}}
  + 4x^2.
\end{equation*}
Hence
\begin{equation*}
  \left.\pd{}{u} dS\right|_{y=\mathrm{const}}
  = - \frac{dy}{2x^2}  \left.\pd{x}{u}\right|_{y=\mathrm{const}}
  = \frac{2 dy}{Q'(x)}
  = \frac{dx}y.
\end{equation*}
Therefore
\begin{equation}\label{eq:a}
  a = \frac1{2\pi\sqrt{-1}} \int_{A} dS
\end{equation}
up to a constant independent of $u$.

Note that $dS$ has a pole at $x=0$. We have $y = \pm \Lambda^3$,
hence the residue is
\begin{equation*}
  \Res_{x=0,y=\pm \Lambda^3} dS = \pm m.
\end{equation*}
Therefore we need to specify the $A$-cycle in \eqref{eq:a}, otherwise
the residue is well-defined only up to $\Z m$.
This is possible by studying the perturbative part of the integral,
but we leave the details to \cite{NY4}.

\begin{NB}
We have
\begin{equation*}
  \omega = 2 \int_{\infty}^{e_1-u/3} \frac{dx}y.
\end{equation*}
Therefore
\begin{equation}\label{eq:a2}
  a = \frac1{\pi\sqrt{-1}} \int_{\infty}^{e_1-u/3} dS.
\end{equation}
\end{NB}%

\subsection{Genus $1$ part}\label{subsec:genus1}

We next determine the coefficients $A$ and $B$. This was done in
\cite[\S7.1]{NY2} for the pure theory. We use the same method.

Consider the blow-up formula \eqref{eq:blow-up1} for $c_1 = C$
\begin{NB}
we need a formula before $\ve_1,\ve_2=0$, not only \eqref{eq:coeff}.
\end{NB}%
and take the coefficient of $t^0\cdot (\ve_1+\ve_2)$. By
\eqref{eq:vanish2} it is zero.
\begin{NB}
We need a detailed formula before taking $\ve_1$, $\ve_2\to 0$.  
\end{NB}%
As in [loc.\ cit.] we get
\begin{equation*}
  \pd{}{a}(A - \frac13 B) = - \frac13 \pd{}{a} \log
  \theta_{11}'(0).
\end{equation*}
Therefore we have
\begin{equation*}
  \exp(A - \frac13 B) = C  \theta_{11}'(0)^{-1/3}
\end{equation*}
for some constant $C$ independent of $a$. Together with \eqref{eq:AB}
we get
\begin{equation*}
  \exp A = \left( C^3 \Lambda^{-1} \omega^{-1}\right)^{1/2},
\qquad
  \exp B = \left( C \Lambda^{-1} \omega^{-1} \right)^{3/2} \theta_{11}'(0)
  = C^{3/2} (2\pi)^{-1/2} \Lambda^{-3/2} \Delta^{1/8},
\end{equation*}
where $\Delta = 16\left(\nicefrac{\pi}{\omega}\right)^{12}
\left(\nicefrac{\theta_{11}'(0)}\pi\right)^8$ is the discriminant.

\begin{NB}
The constant $C$ will be determined later.  
\end{NB}

The perturbative part of $\exp A$ is 
\begin{equation*}
  \left(\frac{-2\sqrt{-1}a}\Lambda\right)^{1/2}.
\end{equation*}
On the other hand, $\omega^{-1/2} =
(-2\pi\sqrt{-1})^{-1/2}(\nicefrac{\partial{u}}{\partial a})^{1/2}$ has
\begin{equation*}
  (-2\pi\sqrt{-1})^{-1/2} \sqrt{2a}.
\end{equation*}
Therefore
\begin{equation*}
   \exp A \left(\frac{-\sqrt{-1}}\Lambda\pd{u}a\right)^{-1/2}
\end{equation*}
has the perturbative part $1$. On the other hand, from the discussion
above, this is a constant independent of $a$. From the degree
consideration as in [loc.\ cit.], it is a homogeneous element. However
the instanton part is a formal power series in $\Lambda/a$ and
$m/a$. Therefore it must be $1$.
\begin{NB}
  $C = -(2\pi)^{1/3}$.
\end{NB}
Hence
\begin{equation}\label{eq:genus1}
  \exp A = \left(\frac{-\sqrt{-1}}\Lambda\pd{u}a\right)^{1/2},
\qquad
  \exp B = \sqrt{-1} \Lambda^{-3/2}\Delta^{1/8}
  \begin{NB}
  = \frac{\theta_{11}'(0)}{\omega\Lambda} \exp A
  \end{NB}%
  .
\end{equation}

\begin{NB}
  It seems that the sign $-1$ was missing in \cite[(7.2)]{NY2}. And it
  seems that this mistake does not effect the main result of
  \cite{GNY}, since we only use
  $(\exp A)^{\chi(X)}(\exp B)^{\sigma(X)} = (\exp A)^{-\sigma(X) + 4}
  = (\exp A)^4 \exp(B - A)^{\sigma (X)}$.
\end{NB}

\begin{NB}
\end{NB}

\subsection{Derivatives of $F_0$}

We will redo the computation in this subsection  at the point
$a=m$ again later, so the reader can safely jump to the next section. But we
just want to point out that the derivatives of $F_0$ can be computed
before specializing $a=m$.

Let us re-write the blow-up formula \eqref{eq:blow-up2} for $c_1=0$ in
terms of the $\sigma$-function:
\begin{multline}\label{eq:sigma}
    \lim_{\ve_1,\ve_2\to 0}
  \frac{\bZ_{c_1=0}(\ve_1,\ve_2,a,m;t;\Lambda)}
  {Z(\ve_1,\ve_2,a,m;\Lambda)}
\\
  = 
  \begin{aligned}[t]
  & \theta_{01}(0)
  \exp\Biggl[
    -\frac18\frac{\partial^2 F_0}{\partial m^2}
    + A-B
    -\frac{\eta}{\omega}\left(\frac{\omega}{4\pi\sqrt{-1}}
    \frac{\partial^2 F_0}{\partial a\partial m}\right)^2
\\
   & \qquad +
   t\left\{\frac1{\bgamma}\left(
       m + \frac12
       \frac{\partial^2 F_0}{\partial\log\Lambda\partial m}\right)
     + \frac{\eta}{2\pi\sqrt{-1}}
     \frac{\partial^2 F_0}{\partial a\partial m}
   \right\} + \frac{ut^2}6 \Biggr]
   \sigma_3(t - \frac{\omega}{4\pi\sqrt{-1}} 
    \frac{\partial^2 F_0}{\partial a\partial m}),
  \end{aligned}
\end{multline}
where $\eta = \zeta(\omega/2) = \pi^2 E_2(\tau)/6\omega$.
\begin{NB}
Recall
  \begin{equation*}
    \sigma_3(t) \defeq
    \exp(\nicefrac{\eta}{\omega} t^2)
    \frac{\theta_{01}(\nicefrac{t}{\omega})}{\theta_{01}(0)}.
  \end{equation*}
We have $\sigma_3(t) = \sigma_3(-t)$.
\end{NB}%
Taking the coefficients of $t^0$, $t^1$, $t^2$ and comparing with
\eqref{eq:vanish}, we get
\begin{align}
  &
  \theta_{01}(0)
  \exp\Biggl[
    -\frac18\frac{\partial^2 F_0}{\partial m^2}
    + A-B
    - {\eta}{\omega}\left(\frac{1}{4\pi\sqrt{-1}}
    \frac{\partial^2 F_0}{\partial a\partial m}\right)^2
  \Biggr]
   \sigma_3(-\frac{\omega}{4\pi\sqrt{-1}} 
    \frac{\partial^2 F_0}{\partial a\partial m}) = 1,
    \label{eq:d^2F/dm^2}
\\
  &
  \frac1{\bgamma}\left(
       m + \frac12
       \frac{\partial^2 F_0}{\partial\log\Lambda\partial m}\right)
     + \frac{\eta}{2\pi\sqrt{-1}}\frac{\partial^2 F_0}{\partial a\partial m}
     + \frac{d}{dt}(\log\sigma_3)(-\frac{\omega}{4\pi\sqrt{-1}} 
    \frac{\partial^2 F_0}{\partial a\partial m}) = 0,
    \label{eq:d^2F/Lm}
\\
  & \frac{u}{3}
  + \frac{d^2}{dt^2} (\log\sigma_3)(-\frac{\omega}{4\pi\sqrt{-1}} 
    \frac{\partial^2 F_0}{\partial a\partial m}) = 0.
\end{align}

Since the second derivative of $\log\sigma$ is $(-1)$ times the
Weierstrass $\wp$-function, we have
\begin{equation}\label{eq:d^2Fdadm}
  \begin{split}
  \frac{u}3 &= - \frac{d^2}{dt^2} (\log\sigma_3)(-\frac{\omega}{4\pi\sqrt{-1}} 
    \frac{\partial^2 F_0}{\partial a\partial m})
    = \wp(\frac{\omega_3}2 - \frac{\omega}{4\pi\sqrt{-1}} 
    \frac{\partial^2 F_0}{\partial a\partial m})
% \\
%     &= \wp(-\frac{\omega_3}2 + \frac{\omega}{4\pi\sqrt{-1}} 
%     \frac{\partial^2 F_0}{\partial a\partial m})
  \end{split}
\end{equation}
from the last equation.
\begin{NB}
  We have $\sigma_3(t) = \sigma(t+\omega_3/2)/\sigma(\omega_3/2)
  \exp(-t\eta_3)$. Therefore
  \begin{equation*}
    \begin{split}
    &  \frac{d}{dt} (\log\sigma_3)(t)
    = \frac{d}{dt} (\log\sigma)(t + \omega_3/2) - \eta_3
    = \zeta(t + \omega_3/2) - \eta_3,,
\\
    & \frac{d^2}{dt^2} (\log\sigma_3)(t)
    = \frac{d^2}{dt^2} (\log\sigma)(t + \omega_3/2)
    = - \wp(t + \omega_3/2).
    \end{split}
  \end{equation*}
Note that $\omega_\alpha$ in \cite{Bateman} is our
$\omega_\alpha/2$. But $\eta_\alpha$ is the same.
\end{NB}%
Therefore
\begin{equation*}
  - \frac{\omega}{4\pi\sqrt{-1}} 
    \frac{\partial^2 F_0}{\partial a\partial m}
    = \int_{\infty}^{0} \frac{dx}y - \frac{\omega_3}2
    = \int_{e_3 - u/3}^{0} \frac{dx}y,
\end{equation*}
where $y$ is as in \eqref{eq:SW2} and $u/3$ is replaced by $0$ since
the quadratic term of \eqref{eq:SW2} is $4u$. Note that this $0$ is
the point where $dS$ has a pole.

\begin{NB}
  We need to specify the choice of the path. Since $\wp$-function is
  doubly periodic, the equation determine
  $\frac{\omega}{4\pi\sqrt{-1}} 
    \frac{\partial^2 F_0}{\partial a\partial m}$ only up to
  the lattice $\Z\omega+\Z\omega'$.
\end{NB}

\begin{NB}
Here is the $\sigma$-version, instead of $\sigma_3$-version.

Let us re-write the blow-up formula \eqref{eq:blow-up2} for $c_1=0$ in
terms of the $\sigma$-function:
\begin{equation*}
  \begin{split}
  & \lim_{\ve_1,\ve_2\to 0}
  \frac{\bZ_{c_1=0}(\ve_1,\ve_2,\vec{a},\vec{m};t;\Lambda)}
  {Z(\ve_1,\ve_2,\vec{a},\vec{m};\Lambda)}
\\
  =\; &
  {\sqrt{-1}}\Lambda   \exp \Biggl[
  \begin{aligned}[t]
    &
%    - \frac{\pi\sqrt{-1}}2
    - \frac18 \left(\frac{\partial^2 F_0}{\partial m^2}
    + 2\frac{\partial^2 F_0}{\partial a\partial m}
    + \frac{\partial^2 F_0}{\partial a^2}\right)
    - \eta\omega\left\{
    \frac1{4\pi\sqrt{-1}}\left(
    \frac{\partial^2 F_0}{\partial {a}\partial m} + 
    \frac{\partial^2 F_0}{\partial {a}^2}
    \right)
      \right\}^2
\\
   &\quad + \frac{t}{\bgamma}\left\{
      \frac12
      \left(
      \frac{\partial^2 F_0}{\partial\log\Lambda\partial m}
      - 2{m}\right)
      \right\}
    + \pi\sqrt{-1} \frac{t}\omega
    \\
    & \qquad
    + \frac{t\eta}{2\pi\sqrt{-1}}\left(
    \frac{\partial^2 F_0}{\partial {a}\partial m} + 
    \frac{\partial^2 F_0}{\partial {a}^2}
    \right)
    + \frac{t^2 u}6
    \Biggr]
\\
   & \qquad\qquad\times \sigma\left(
     t - 
       \frac\omega{4\pi\sqrt{-1}}\left(
    \frac{\partial^2 F_0}{\partial {a}\partial m} + 
    \frac{\partial^2 F_0}{\partial {a}^2}
    \right)
     \right),
  \end{aligned}
  \end{split}
\end{equation*}
where $\eta = \zeta(\omega/2) = \pi^2 E_2(\tau)/6\omega$.

\begin{NB2}
Let us give the detailed calculation.

We start with
\begin{multline*}
  \lim_{\ve_1,\ve_2\to 0}
  \frac{\bZ_{c_1=0}(\ve_1,\ve_2,\vec{a},\vec{m};t;\Lambda)}
  {Z(\ve_1,\ve_2,\vec{a},\vec{m};\Lambda)}
\\
  =
  \exp\Biggl[
  \begin{aligned}[t]
    & -\frac18
    \frac{\partial^2 F_0}{\partial m^2}
    + A - B
% \\
%     & 
    + \frac{t}{\bgamma}\left\{
      \frac12
      \left(
      \frac{\partial^2 F_0}{\partial\log\Lambda\partial m}
      - 2{m}\right)
      \right\}
% \\
%     & 
    - \frac1{\bgamma^2}
    \frac{\partial^2 F_0}{\partial(\log\Lambda)^2}
    \frac{t^2}2
    \Biggr]
  \end{aligned}
\\
  \times\theta_{01}\left(
    -\frac1{4\pi\sqrt{-1}}
    \frac{\partial^2 F_0}{\partial {a}\partial m}
      +
      \frac{t}{\bgamma}
      \frac1{2\pi\sqrt{-1}}
      \frac{\partial^2 F_0}{\partial{a}\partial\log\Lambda}
      \Biggm|\tau \right).
\end{multline*}
This is nothing but \eqref{eq:blow-up2}.
Let
\begin{equation*}
   z\defeq  -\frac1{4\pi\sqrt{-1}}
    \frac{\partial^2 F_0}{\partial {a}\partial m}
      +
      \frac{t}{\bgamma}
      \frac1{2\pi\sqrt{-1}}
      \frac{\partial^2 F_0}{\partial{a}\partial\log\Lambda}
   = -\frac1{4\pi\sqrt{-1}}
    \frac{\partial^2 F_0}{\partial {a}\partial m} + \frac{t}\omega.
\end{equation*}
Set
\begin{equation*}
  \delta \defeq 
    \frac1{4\pi\sqrt{-1}}\left(
    \frac{\partial^2 F_0}{\partial {a}\partial m} + 
    \frac{\partial^2 F_0}{\partial {a}^2}
    \right).
\end{equation*}
Hence
\begin{equation*}
  z +\frac12\tau =\frac{t}\omega - \delta
  = \frac{t}\omega -
    \frac1{4\pi\sqrt{-1}}\left(
    \frac{\partial^2 F_0}{\partial {a}\partial m} + 
    \frac{\partial^2 F_0}{\partial {a}^2}
    \right).
\end{equation*}
Then we have
\begin{equation*}
  \begin{split}
  \theta_{01}(z) &= \exp\left(\frac{\pi\sqrt{-1}}4\tau
    + \pi\sqrt{-1}(z + \frac12)\right) \theta_{11}(z+\frac12 \tau)
\\
  & = \exp\left(\frac{\pi\sqrt{-1}}4\tau
    + \pi\sqrt{-1}(z + \frac12)\right)
  \sigma(\omega(z+\frac12\tau))
  \frac{\theta_{11}'(0)}{\omega}
  \exp(-{\eta}{\omega} (z+\frac12\tau)^2)
\\
  & = \exp\left(-\frac{\pi\sqrt{-1}}4\tau
    + \pi\sqrt{-1}(\frac{t}\omega - \delta)
    + \frac{\pi\sqrt{-1}}2 \right)
  \sigma(t - \delta\omega)
  \frac{\theta_{11}'(0)}{\omega}
  \exp(-{\eta}{\omega} (\frac{t}\omega - \delta)^2)
  \end{split}
\end{equation*}
Therefore
\begin{equation*}
  \begin{split}
  & \lim_{\ve_1,\ve_2\to 0}
  \frac{\bZ_{c_1=0}(\ve_1,\ve_2,\vec{a},\vec{m};t;\Lambda)}
  {Z(\ve_1,\ve_2,\vec{a},\vec{m};\Lambda)}
\\
  =\; &
  \Lambda\exp \Biggl[
  \begin{aligned}[t]
    &
      \frac{\pi\sqrt{-1}}2
    - \frac18 \left(\frac{\partial^2 F_0}{\partial m^2}
    - \frac{\partial^2 F_0}{\partial a^2}\right)
      + \frac{t}{\bgamma}\left\{
      \frac12
      \left(
      \frac{\partial^2 F_0}{\partial\log\Lambda\partial m}
      - 2{m}\right)
      \right\}
    \\
    & \quad
    +\pi\sqrt{-1}
    \left(
      \frac{t}\omega - 
    \frac1{4\pi\sqrt{-1}}\left(
    \frac{\partial^2 F_0}{\partial {a}\partial m} + 
    \frac{\partial^2 F_0}{\partial {a}^2}
      \right)\right)
    \\
    & \qquad
    - \eta\omega\left\{
      \frac{t}\omega -
    \frac1{4\pi\sqrt{-1}}\left(
    \frac{\partial^2 F_0}{\partial {a}\partial m} + 
    \frac{\partial^2 F_0}{\partial {a}^2}
    \right)
      \right\}^2
    - \frac1{\bgamma^2}
    \frac{\partial^2 F_0}{\partial(\log\Lambda)^2}
    \frac{t^2}2
    \Biggr]
\\
   & \qquad\qquad\times \sigma\left(
     t - 
       \frac\omega{4\pi\sqrt{-1}}\left(
    \frac{\partial^2 F_0}{\partial {a}\partial m} + 
    \frac{\partial^2 F_0}{\partial {a}^2}
    \right)
     \right)
  \end{aligned}
\\
  =\; &
  {\sqrt{-1}}\Lambda   \exp \Biggl[
  \begin{aligned}[t]
    &
%    - \frac{\pi\sqrt{-1}}2
    - \frac18 \left(\frac{\partial^2 F_0}{\partial m^2}
    + 2\frac{\partial^2 F_0}{\partial a\partial m}
    + \frac{\partial^2 F_0}{\partial a^2}\right)
    - \eta\omega\left\{
    \frac1{4\pi\sqrt{-1}}\left(
    \frac{\partial^2 F_0}{\partial {a}\partial m} + 
    \frac{\partial^2 F_0}{\partial {a}^2}
    \right)
      \right\}^2
\\
   &\quad + \frac{t}{\bgamma}\left\{
      \frac12
      \left(
      \frac{\partial^2 F_0}{\partial\log\Lambda\partial m}
      - 2{m}\right)
      \right\}
    + \pi\sqrt{-1} \frac{t}\omega
    \\
    & \qquad
    + \frac{t\eta}{2\pi\sqrt{-1}}\left(
    \frac{\partial^2 F_0}{\partial {a}\partial m} + 
    \frac{\partial^2 F_0}{\partial {a}^2}
    \right)
    + \frac{t^2 u}6
    \Biggr]
\\
   & \qquad\qquad\times \sigma\left(
     t - 
       \frac\omega{4\pi\sqrt{-1}}\left(
    \frac{\partial^2 F_0}{\partial {a}\partial m} + 
    \frac{\partial^2 F_0}{\partial {a}^2}
    \right)
     \right)
  \end{aligned}
  \end{split}
\end{equation*}
\end{NB2}

\begin{NB2}
This is an earlier attempt:

We use
\begin{equation*}
  \sigma_3(t + \frac{\omega}{4\pi\sqrt{-1}} 
    \frac{\partial^2 F_0}{\partial a\partial m})
  = 
  \sigma_3(- t - \frac{\omega}{4\pi\sqrt{-1}} 
    \frac{\partial^2 F_0}{\partial a\partial m})
  = \sigma(\frac{\omega_3}2 - t - \frac{\omega}{4\pi\sqrt{-1}} 
    \frac{\partial^2 F_0}{\partial a\partial m})
    \frac{\exp(-t\eta_3)}{\sigma(\omega_3/2)},
\end{equation*}
where $\eta_3 = \zeta(\omega_3)$.
Note that
\begin{equation*}
  \frac{\omega_3}2 - \frac{\omega}{4\pi\sqrt{-1}} 
    \frac{\partial^2 F_0}{\partial a\partial m}
  = - \frac{\omega}{4\pi\sqrt{-1}}\left(
    \frac{\partial^2 F_0}{\partial a^2}
     + \frac{\partial^2 F_0}{\partial a\partial m}
    \right)
\end{equation*}
does not contain the term $\log(m-a)/\Lambda$ in the perturbative
part. Therefore we can evaluate this term at $a=m$.

The Legendre relation says $\eta \omega_3 - \eta_3 \omega =
\pi\sqrt{-1}$. Hence
\begin{equation*}
  \begin{split}
    & \exp(-t\eta_3) = \exp\left(t \frac{\pi\sqrt{-1}}\omega - t
      \eta\frac{\omega_3}{\omega} \right) = \exp\left(t
      \frac{\pi\sqrt{-1}}\omega - t \eta \tau\right)
    \\
    =\; & \exp\left(t \frac{\pi\sqrt{-1}}\omega + 
      \frac{t\eta}{2\pi\sqrt{-1}}\frac{\partial^2 F_0}{\partial a^2}\right).
  \end{split}
\end{equation*}
\end{NB2}

Taking the coefficients of $t^0$, $t^1$, $t^2$ and comparing with
\eqref{eq:vanish}, we get
\begin{align*}
  &
  1 = 
  \begin{aligned}[t]
  \sqrt{-1}\Lambda &
  \exp\Biggl[
    - \frac18 \left(\frac{\partial^2 F_0}{\partial m^2}
    + 2\frac{\partial^2 F_0}{\partial a\partial m}
    + \frac{\partial^2 F_0}{\partial a^2}\right)
    - \eta\omega\left\{
    \frac1{4\pi\sqrt{-1}}\left(
    \frac{\partial^2 F_0}{\partial {a}\partial m} + 
    \frac{\partial^2 F_0}{\partial {a}^2}
    \right)
      \right\}^2
  \Biggr]
\\
    & \quad \times 
   \sigma
%    \left(-\frac{\omega}{4\pi\sqrt{-1}} 
%      \left(
%     \frac{\partial^2 F_0}{\partial a\partial m}
%     +
%     \frac{\partial^2 F_0}{\partial a^2}
%     \right)\right)
   ,
  \end{aligned}
\\
  &
  0 = \frac1{\bgamma} \left\{
      \frac12
      \left(
      \frac{\partial^2 F_0}{\partial\log\Lambda\partial m}
      - 2{m}\right)
      \right\}
    + \frac{\pi\sqrt{-1}}\omega
%     \\
%     & \qquad
    + \frac{\eta}{2\pi\sqrt{-1}}\left(
    \frac{\partial^2 F_0}{\partial {a}\partial m} + 
    \frac{\partial^2 F_0}{\partial {a}^2}
    \right)
    \Biggr]
     + \frac{d}{dt}(\log\sigma)
%      \left(
%      -\frac{\omega}{4\pi\sqrt{-1}} 
%      \left(
%     \frac{\partial^2 F_0}{\partial a\partial m}
%     +
%     \frac{\partial^2 F_0}{\partial a^2}
%     \right)\right)
     ,
\\
  & 0 = \frac{u}{3}
  + \frac{d^2}{dt^2} (\log\sigma)
%   \left(
%      -\frac{\omega}{4\pi\sqrt{-1}} 
%      \left(
%     \frac{\partial^2 F_0}{\partial a\partial m}
%     +
%     \frac{\partial^2 F_0}{\partial a^2}
%     \right)\right)
  ,
\end{align*}
where the $\sigma$-function is evaluated at
\begin{equation*}
%    \left(
     -\frac{\omega}{4\pi\sqrt{-1}} 
     \left(
    \frac{\partial^2 F_0}{\partial a\partial m}
    +
    \frac{\partial^2 F_0}{\partial a^2}
    \right)
%  \right)
  .
\end{equation*}
\end{NB}%

\section{Partition functions at the singular point}

Recall that we need to specialize $a=m$ in
\thmref{thm:partition}. At this point, the Seiberg-Witten curve is
singular, and many formulas are simplified.

\subsection{The special point $a=m$}

Recall that the period $\tau$ of the Seiberg-Witten curve was given by
the second derivative of $F_0$ with respect to $a$ \eqref{eq:tau}. Its
perturbative part is given by \eqref{eq:pert}. In particular, $q =
\exp(\pi\sqrt{-1}\tau)$ vanishes at $a=m$ since it contains a factor
$-a+m$. Therefore $\theta_{00}\to 1$, $\theta_{01}\to 1$,
$\theta_{10}\to 0$ at $a=m$, and hence we have $e_2 = e_3$ from
\eqref{eq:e_}. The cycle encircling $e_2$, $e_3$ vanishes and the
curve develops singularities.

\begin{NB}
In \eqref{eq:a2} we change the path to $e_2\to e_3$. Then we pick up
the residue of $dS$, as we explained above. Therefore
\begin{equation*}
  a = \frac1{\pi\sqrt{-1}}\int_{e_2-u/3}^{e_3-u/3} dS + m.
\end{equation*}
We consider the degeneration of the curve at $e_2 = e_3$.
We then have $\theta_{10} = 0$, i.e., $q = 0$. Therefore
\begin{gather*}
  e_1 = \frac23 \left(\frac{\pi}{\omega}\right)^2, \quad
  e_2 = e_3 = - \frac13\left(\frac{\pi}{\omega}\right)^2,
\quad
  g_2 = \frac43 \left(\frac{\pi}{\omega}\right)^4,\quad
  g_3 = \frac8{27} \left(\frac{\pi}{\omega}\right)^6.
\end{gather*}
In particular, we get
\begin{equation}\label{eq:g at a=m}
    \frac43 \left(\frac{\pi}{\omega}\right)^4
    = \frac43 u^2 - 4 m \Lambda^3,
\qquad
    \frac8{27} \left(\frac{\pi}{\omega}\right)^6
    = -\frac8{27}u^3 + \frac43 u m \Lambda^3 - \Lambda^6
\end{equation}
from (\ref{eq:g2}, \ref{eq:g3}). The discriminant $\Delta$ vanishes at
this point.

Since $dS$ is regular at the point $e_2$,
\begin{NB2}
  Recall $dS = \frac{Q'(x) dx}{4xy}$. Then $y = \sqrt{Q(x)}$ and
  $Q'(x)$ are both simple zero at $x = e_2 + u/3$.
\end{NB2}%
we have $a = m$. This is compatible with the formula for $B$
\eqref{eq:genus1}: the perturbative part of $\exp B$, and hence $\exp
B$ itself vanishes at $a=m$.

Note that $u$ is the solution of the cubic equation $\Delta = 0$ (see
\eqref{eq:disc}) satisfying $u = m^2 + O(\Lambda) = a^2 + O(\Lambda)$.
\end{NB}

The blow-up formula \eqref{eq:sigma} is {\it not\/} suitable for the
specialization $e_2 = e_3$, as it contains an expression $\partial^2
F_0/\partial a\partial m$, which has $\log(-a+m)/\Lambda$ in the
perturbative part.
We observe that
\begin{equation*}
  \frac{\omega_3}2 - \frac{\omega}{4\pi\sqrt{-1}} 
    \frac{\partial^2 F_0}{\partial a\partial m}
  = - \frac{\omega}{4\pi\sqrt{-1}}\left(
    \frac{\partial^2 F_0}{\partial a^2}
     + \frac{\partial^2 F_0}{\partial a\partial m}
    \right)
\end{equation*}
does not contain the term $\log(-a+m)/\Lambda$ in the perturbative
part. Hence we can evaluate this term at $a=m$.
Therefore we use  $\sigma$, instead of $\sigma_3$ in \eqref{eq:sigma}:
\begin{equation*}
  \begin{split}
  & \lim_{\ve_1,\ve_2\to 0}
  \frac{\bZ_{c_1=0}(\ve_1,\ve_2,\vec{a},\vec{m};t;\Lambda)}
  {Z(\ve_1,\ve_2,\vec{a},\vec{m};\Lambda)}
\\
  =\; &
  {\sqrt{-1}}\Lambda   \exp \Biggl[
  \begin{aligned}[t]
    &
%    - \frac{\pi\sqrt{-1}}2
    - \frac18 \left(\frac{\partial^2 F_0}{\partial m^2}
    + 2\frac{\partial^2 F_0}{\partial a\partial m}
    + \frac{\partial^2 F_0}{\partial a^2}\right)
    - \eta\omega\left\{
    \frac1{4\pi\sqrt{-1}}\left(
    \frac{\partial^2 F_0}{\partial {a}\partial m} + 
    \frac{\partial^2 F_0}{\partial {a}^2}
    \right)
      \right\}^2
\\
   &\quad + \frac{t}{\bgamma}\left\{
      \frac12
      \left(
      \frac{\partial^2 F_0}{\partial\log\Lambda\partial m}
      - 2{m}\right)
      \right\}
    + \pi\sqrt{-1} \frac{t}\omega
    \\
    & \qquad
    + \frac{t\eta}{2\pi\sqrt{-1}}\left(
    \frac{\partial^2 F_0}{\partial {a}\partial m} + 
    \frac{\partial^2 F_0}{\partial {a}^2}
    \right)
    + \frac{t^2 u}6
    \Biggr]
\\
   & \qquad\qquad\times \sigma\left(
     t - 
       \frac\omega{4\pi\sqrt{-1}}\left(
    \frac{\partial^2 F_0}{\partial {a}\partial m} + 
    \frac{\partial^2 F_0}{\partial {a}^2}
    \right)
     \right).
  \end{aligned}
  \end{split}
\end{equation*}
\begin{NB}
Wrong version.
  \begin{equation*}
  =\; 
  -i   \exp
  \begin{aligned}[t]
    &
    \Biggl[
%    - \frac{\pi\sqrt{-1}}2
    - \frac18 \left(\frac{\partial^2 F_0}{\partial m^2}
    + 2\frac{\partial^2 F_0}{\partial a\partial m}
    + \frac{\partial^2 F_0}{\partial a^2}\right)
    - \eta\omega\left\{
    \frac1{4\pi\sqrt{-1}}\left(
    \frac{\partial^2 F_0}{\partial {a}\partial m} + 
    \frac{\partial^2 F_0}{\partial {a}^2}
    \right)
      \right\}^2
\\
   &\quad + \frac{t}{\bgamma}\left\{
      \frac12
      \left(
      \frac{\partial^2 F_0}{\partial\log\Lambda\partial m}
      - 2{m}\right)
      \right\}
    -\pi\sqrt{-1} \frac{t}\omega
    \\
    & \qquad
    - \frac{t\eta}{2\pi\sqrt{-1}}\left(
    \frac{\partial^2 F_0}{\partial {a}\partial m} + 
    \frac{\partial^2 F_0}{\partial {a}^2}
    \right)
    + \frac{t^2 u}6
    \Biggr]
\\
   & \qquad\qquad\times \sigma\left(
     t + 
       \frac\omega{4\pi\sqrt{-1}}\left(
    \frac{\partial^2 F_0}{\partial {a}\partial m} + 
    \frac{\partial^2 F_0}{\partial {a}^2}
    \right)
     \right).
  \end{aligned}
  \end{equation*}
\end{NB}

\begin{NB}
Even earlier attempt:

Let us consider \eqref{eq:d^2Fdadm} at $e_2 = e_3$. Note that
\begin{equation*}
  - \frac{\omega_3}2 + \frac{\omega}{4\pi\sqrt{-1}} 
    \frac{\partial^2 F_0}{\partial a\partial m}
  = \frac{\omega}{4\pi\sqrt{-1}}\left(
    \frac{\partial^2 F_0}{\partial a^2}
     + \frac{\partial^2 F_0}{\partial a\partial m}
    \right)
\end{equation*}
does not contain the term $\log(m-a)/\Lambda$ in the perturbative
part. Therefore we can evaluate this term at $a=m$.
\begin{NB2}
  Note that
\begin{equation*}
  - \frac{\omega_3}2 = \int_{e_2-u/3}^{e_1-u/3} \frac{dx}y.
\end{equation*}
\end{NB2}%
Since $\wp(z)$ degenerates to a trigonometric function at $e_2 = e_3$,
we get
\begin{equation*}
  \frac{u}3 = e_2 - \frac{3e_2}{\sin^2(\sqrt{-3e_2} t)}
\end{equation*}
where $t = 
  \left.- \frac{\omega}{4\pi\sqrt{-1}}\left(
    \frac{\partial^2 F_0}{\partial a^2}
     + \frac{\partial^2 F_0}{\partial a\partial m}
    \right)\right|_{a=m}
$.
\begin{NB2}
It seems
\begin{equation*}
  \left.- \frac{\omega}{4\pi\sqrt{-1}}\left(
    \frac{\partial^2 F_0}{\partial a^2}
     + \frac{\partial^2 F_0}{\partial a\partial m}
    \right)\right|_{a=m}
  = \frac2{\sqrt{-3e_2}} \arctan
  \left(
    \frac{\sqrt{u-6e_2}}{3\sqrt{-e_2}}
  \right).
\end{equation*}
I still do not check that two expressions are compatible.
\end{NB2}

\begin{NB2}
  We have
  \begin{equation*}
    \int \frac{dx}{(x+A)\sqrt{x-2A}}
    = \frac2{\sqrt{3A}} \arctan\left(
      \sqrt{\frac{x-2A}{3A}}
      \right).
  \end{equation*}

  We have
  \begin{equation*}
    \int \frac{dx}{(x-e_2)\sqrt{x+2e_2}}
    = - \frac2{\sqrt{3e_2}} \tanh^{-1}\left(
      \sqrt{\frac{x+2e_2}{3e_2}}
      \right).
  \end{equation*}
\end{NB2}%
\begin{NB2}
  Let us compare the perturbative parts. We have
  \begin{equation*}
    3e_2 = - \left(\frac{\pi}{\omega}\right)^2
    = \frac14 \left(\pd{u}a\right)^2 = a^2 + \cdots,
\qquad
    u = a^2 + \cdots.
  \end{equation*}

From the left hand side, we
  have
  \begin{equation*}
    \frac\omega{2}
    \left(1 + \frac8{2\pi\sqrt{-1}}
      \log\frac{-2\sqrt{-1}a}{\Lambda}
    - \frac2{2\pi\sqrt{-1}}
   \log\frac{2a}{\Lambda}\right).
  \end{equation*}
  It seems that we lost something during we choose the path for the
  integration.
\end{NB2}

Similarly from \eqref{eq:d^2F/Lm} we get 
\begin{equation}
  \begin{split}
    & \frac1{\bgamma}\left(
       m + \frac12
       \frac{\partial^2 F_0}{\partial\log\Lambda\partial m}\right)
     = \frac{\eta}{2\pi\sqrt{-1}}\frac{\partial^2 F_0}{\partial a\partial m}
     - \frac{d}{dt}(\log\sigma_3)(\frac{\omega}{4\pi\sqrt{-1}} 
    \frac{\partial^2 F_0}{\partial a\partial m})
\\
  =\; &
  \frac{\eta}{2\pi\sqrt{-1}}\frac{\partial^2 F_0}{\partial a\partial m}
  - \eta \tau
  + \zeta(\frac{\omega_3}2 - \frac{\omega}{4\pi\sqrt{-1}} 
    \frac{\partial^2 F_0}{\partial a\partial m})
    + \frac{\pi\sqrt{-1}}{\omega},
  \end{split}
\end{equation}
where we have used Legendre's relation
\(
   \eta \tau - \eta_3 = \pi\sqrt{-1}/\omega.
\)

Similarly (\ref{eq:d^2F/dm^2},\ref{eq:d^2F/Lm}) can be evaluated
at $e_2 = e_3$. We have
\begin{equation*}
    \frac1{\bgamma}\left(
       m + \frac12
       \frac{\partial^2 F_0}{\partial\log\Lambda\partial m}\right)
     - \frac{\omega}{2\pi\sqrt{-1}}\frac{\partial^2 F_0}{\partial a\partial m}
     = - \zeta(
     - \frac{\omega}{4\pi\sqrt{-1}}\left(
    \frac{\partial^2 F_0}{\partial a^2}
     + \frac{\partial^2 F_0}{\partial a\partial m}
    \right)) + \eta_3 ?????
\end{equation*}
\end{NB}

\begin{NB}
In \cite[\S13.15]{Bateman} we find
\begin{equation*}
  \begin{split}
  & \wp(t) = e_2 - \frac{3e_2}{\sin^2(\sqrt{-3e_2}t)},
\\
  & \zeta(t) = -e_2 t + \sqrt{-3e_2}\cot(\sqrt{-3e_2}t),
\\
  & \sigma(t) = \frac1{\sqrt{-3e_2}} \sin(\sqrt{-3e_2}t) 
  \exp(-\frac12 e_2 t^2)
  \end{split}
\end{equation*}
at $e_2 = e_3$.
\end{NB}

Now we can specialize $e_2 = e_3$: the $\sigma$-function becomes
\begin{equation*}
  \sigma(t) = \frac\omega{\pi} \sin(\frac{\pi}{\omega}t) 
  \exp\left[\frac16 \left(\frac{\pi}{\omega}\right)^2 t^2\right].
\end{equation*}
We also note
\begin{equation*}
  \eta\omega =
  \begin{NB}
    \frac{\pi^2}6 E_2(\tau) =
  \end{NB}
  \frac{\pi^2}6
\end{equation*}
at $e_2 = e_3$.
Therefore
{\allowdisplaybreaks
\begin{equation*}
  \begin{split}
  & \lim_{\ve_1,\ve_2\to 0}
  \frac{\bZ_{c_1=0}(\ve_1,\ve_2,\vec{a},\vec{m};t;\Lambda)}
  {Z(\ve_1,\ve_2,\vec{a},\vec{m};\Lambda)}
\\
  =\; &
  \frac{\sqrt{-1}\omega\Lambda}{\pi}   \exp \Biggl[
  \begin{aligned}[t]
    &
%    - \frac{\pi\sqrt{-1}}2
    - \frac18 \left(\frac{\partial^2 F_0}{\partial m^2}
    + 2\frac{\partial^2 F_0}{\partial a\partial m}
    + \frac{\partial^2 F_0}{\partial a^2}\right)
%     - \frac16\left\{
%     \frac1{4\sqrt{-1}}\left(
%     \frac{\partial^2 F_0}{\partial {a}\partial m} + 
%     \frac{\partial^2 F_0}{\partial {a}^2}
%     \right)
%       \right\}^2
\\
   &\quad + \frac{t}{\bgamma}\left\{
      \frac12
      \left(
      \frac{\partial^2 F_0}{\partial\log\Lambda\partial m}
      - 2{m}\right)
      \right\}
    + \pi\sqrt{-1} \frac{t}\omega
%     \\
%     & \qquad
%     - \frac{t\eta}{2\pi\sqrt{-1}}\left(
%     \frac{\partial^2 F_0}{\partial {a}\partial m} + 
%     \frac{\partial^2 F_0}{\partial {a}^2}
%     \right)
    + \frac{t^2}6\left(u + \left(\frac{\pi}\omega\right)^2\right)
    \Biggr]
\\
   & \qquad\qquad\times \sin\left(
     \frac{\pi}\omega t -
       \frac1{4\sqrt{-1}}\left(
    \frac{\partial^2 F_0}{\partial {a}\partial m} + 
    \frac{\partial^2 F_0}{\partial {a}^2}
    \right)
     \right).
  \end{aligned}
  \end{split}
\end{equation*}
}
\begin{NB}
{\allowdisplaybreaks
  \begin{equation*}
  \begin{split}
  & \lim_{\ve_1,\ve_2\to 0}
  \frac{\bZ_{c_1=0}(\ve_1,\ve_2,\vec{a},\vec{m};t;\Lambda)}
  {Z(\ve_1,\ve_2,\vec{a},\vec{m};\Lambda)}
\\
  =\;&
  {\sqrt{-1}}\Lambda   \exp \Biggl[
  \begin{aligned}[t]
    &
%    - \frac{\pi\sqrt{-1}}2
    - \frac18 \left(\frac{\partial^2 F_0}{\partial m^2}
    + 2\frac{\partial^2 F_0}{\partial a\partial m}
    + \frac{\partial^2 F_0}{\partial a^2}\right)
    - \eta\omega\left\{
    \frac1{4\pi\sqrt{-1}}\left(
    \frac{\partial^2 F_0}{\partial {a}\partial m} + 
    \frac{\partial^2 F_0}{\partial {a}^2}
    \right)
      \right\}^2
\\
   &\quad + \frac{t}{\bgamma}\left\{
      \frac12
      \left(
      \frac{\partial^2 F_0}{\partial\log\Lambda\partial m}
      - 2{m}\right)
      \right\}
    + \pi\sqrt{-1} \frac{t}\omega
    \\
    & \qquad
    + \frac{t\eta}{2\pi\sqrt{-1}}\left(
    \frac{\partial^2 F_0}{\partial {a}\partial m} + 
    \frac{\partial^2 F_0}{\partial {a}^2}
    \right)
    + \frac{t^2 u}6
    \Biggr]
\\
   & \qquad\qquad\times \sigma\left(
     t - 
       \frac\omega{4\pi\sqrt{-1}}\left(
    \frac{\partial^2 F_0}{\partial {a}\partial m} + 
    \frac{\partial^2 F_0}{\partial {a}^2}
    \right)
     \right).
  \end{aligned}
\\
  =\; &
    \frac{\sqrt{-1}\omega\Lambda}{\pi} \exp \Biggl[
  \begin{aligned}[t]
    &
%    - \frac{\pi\sqrt{-1}}2
    - \frac18 \left(\frac{\partial^2 F_0}{\partial m^2}
    + 2\frac{\partial^2 F_0}{\partial a\partial m}
    + \frac{\partial^2 F_0}{\partial a^2}\right)
    - \frac16\left\{
    \frac1{4\sqrt{-1}}\left(
    \frac{\partial^2 F_0}{\partial {a}\partial m} + 
    \frac{\partial^2 F_0}{\partial {a}^2}
    \right)
      \right\}^2
\\
   &\quad + \frac{t}{\bgamma}\left\{
      \frac12
      \left(
      \frac{\partial^2 F_0}{\partial\log\Lambda\partial m}
      - 2{m}\right)
      \right\}
    + \pi\sqrt{-1} \frac{t}\omega
    \\
    & \qquad
    + \frac{t}{2\pi\sqrt{-1}}\frac{\pi^2}{6\omega}\left(
    \frac{\partial^2 F_0}{\partial {a}\partial m} + 
    \frac{\partial^2 F_0}{\partial {a}^2}
    \right)
    + \frac{t^2 u}6
    \Biggr]
\\
   & \qquad\qquad\times \sin\left(
     \frac{\pi}\omega t -
       \frac1{4\sqrt{-1}}\left(
    \frac{\partial^2 F_0}{\partial {a}\partial m} + 
    \frac{\partial^2 F_0}{\partial {a}^2}
    \right)
     \right)
\\
   & \qquad\qquad\times \exp\left\{    
     \frac16
     \left(\frac{\pi}\omega t -
              \frac1{4\sqrt{-1}}\left(
    \frac{\partial^2 F_0}{\partial {a}\partial m} + 
    \frac{\partial^2 F_0}{\partial {a}^2}
    \right)\right)^2
     \right\}.
  \end{aligned}
\\
  =\; &
  \frac{\sqrt{-1}\omega\Lambda}{\pi} \exp    \Biggl[
  \begin{aligned}[t]
    &
%    - \frac{\pi\sqrt{-1}}2
    - \frac18 \left(\frac{\partial^2 F_0}{\partial m^2}
    + 2\frac{\partial^2 F_0}{\partial a\partial m}
    + \frac{\partial^2 F_0}{\partial a^2}\right)
%     - \frac16\left\{
%     \frac1{4\sqrt{-1}}\left(
%     \frac{\partial^2 F_0}{\partial {a}\partial m} + 
%     \frac{\partial^2 F_0}{\partial {a}^2}
%     \right)
%       \right\}^2
\\
   &\quad + \frac{t}{\bgamma}\left\{
      \frac12
      \left(
      \frac{\partial^2 F_0}{\partial\log\Lambda\partial m}
      - 2{m}\right)
      \right\}
    + \pi\sqrt{-1} \frac{t}\omega
%     \\
%     & \qquad
%     - \frac{t\eta}{2\pi\sqrt{-1}}\left(
%     \frac{\partial^2 F_0}{\partial {a}\partial m} + 
%     \frac{\partial^2 F_0}{\partial {a}^2}
%     \right)
    + \frac{t^2}6\left(u + \left(\frac{\pi}\omega\right)^2\right)
    \Biggr]
\\
   & \qquad\qquad\times \sin\left(
     \frac{\pi}\omega t -
       \frac1{4\sqrt{-1}}\left(
    \frac{\partial^2 F_0}{\partial {a}\partial m} + 
    \frac{\partial^2 F_0}{\partial {a}^2}
    \right)
     \right)
  \end{aligned}
  \end{split}
\end{equation*}
}
\end{NB}

As before, we take the coefficients of $t^0$, $t^1$, $t^2$, compare with
\eqref{eq:vanish} and get
\begin{align}
  &
    1 =   \frac{\sqrt{-1}\omega\Lambda}{\pi} 
  \begin{aligned}[t]
    &
    \exp    \Biggl[
%    - \frac{\pi\sqrt{-1}}2
    - \frac18 \left(\frac{\partial^2 F_0}{\partial m^2}
    + 2\frac{\partial^2 F_0}{\partial a\partial m}
    + \frac{\partial^2 F_0}{\partial a^2}\right) \Biggr]
\\
   & \qquad\qquad\times \sin\left(-
       \frac1{4\sqrt{-1}}\left(
    \frac{\partial^2 F_0}{\partial {a}\partial m} + 
    \frac{\partial^2 F_0}{\partial {a}^2}
    \right)
     \right),
  \end{aligned} \label{eq:1}
\\
  &   
  0 = 
  \begin{aligned}[t]
  & \frac{1}{\bgamma}\left\{
      \frac12
      \left(
      \frac{\partial^2 F_0}{\partial\log\Lambda\partial m}
      - 2{m}\right)
      \right\}
   +  \frac{\pi\sqrt{-1}}\omega
\\
   & \qquad\qquad
   + \frac{\pi}\omega\cot\left(-
       \frac1{4\sqrt{-1}}\left(
    \frac{\partial^2 F_0}{\partial {a}\partial m} + 
    \frac{\partial^2 F_0}{\partial {a}^2}
    \right)
     \right),
  \end{aligned} \label{eq:2}
\\
  &
  0 = 
  \frac{1}3\left(u + \left(\frac{\pi}\omega\right)^2\right)
  - \left(\frac{\pi}\omega\right)^2\sin^{-2} \left(
    - \frac1{4\sqrt{-1}}\left(
    \frac{\partial^2 F_0}{\partial {a}\partial m} + 
    \frac{\partial^2 F_0}{\partial {a}^2}
    \right)
     \right). \label{eq:3}
\end{align}

\begin{NB}
\begin{equation*}
  \begin{split}
  & \frac{d}{dt}\log \sin(
     \frac{\pi}\omega t - \delta)
     = \frac{\pi}\omega \frac{\cos(\frac{\pi}\omega t - \delta)}
     {\sin(\frac{\pi}\omega t - \delta)}
     = \frac{\pi}\omega {\cot(\frac{\pi}\omega t - \delta)},
\\
  & \frac{d^2}{dt^2}\log \sin(
     \frac{\pi}\omega t - \delta)
     = - \left(\frac{\pi}\omega\right)^2
     \frac1{\sin^2(\frac{\pi}\omega t - \delta)}.
  \end{split}
\end{equation*}
\end{NB}

\subsection{Miscellaneous identities}
We assume $a=m$ hereafter, and solve equations (\ref{eq:1},
\ref{eq:2}, \ref{eq:3}) to write down various derivatives of
$F_0$ explicitly.

Since $e_1 - u/3 = 2/3 \left(\nicefrac{\pi}\omega\right)^2 - u/3$, $e_2 -
u/3 = e_3 - u/3 = - 1/3\left(\nicefrac{\pi}\omega\right)^2 - u/3$ is a
solution of $y^2 = 4x^2(x+u) + 4a\Lambda^3 x + \Lambda^6$, we have
\begin{equation*}
  4 \left(x + \frac{u}3 + \frac13\left(\frac{\pi}\omega\right)^2\right)^2
  \left(x + \frac{u}3 - \frac23\left(\frac{\pi}\omega\right)^2\right)
  = 4x^2(x+u) + 4a\Lambda^3 x + \Lambda^6.
\end{equation*}
Thus
\begin{align}
  & \left({u} + \left(\frac{\pi}\omega\right)^2\right)^2
  \left({u} - 2 \left(\frac{\pi}\omega\right)^2\right) 
  = \frac{27}4 \Lambda^6, \label{eq:4}
\\
  & \left({u} + \left(\frac{\pi}\omega\right)^2\right)
  \left({u} - \left(\frac{\pi}\omega\right)^2\right) 
  = 3a\Lambda^3. \label{eq:5}
\end{align}
This suggests the possibility to replace $a$ by $u -
\left(\nicefrac{\pi}\omega\right)^2$ or $u +
\left(\nicefrac{\pi}\omega\right)^2$. Therefore we write various
functions in terms of $u$ and $\nicefrac\pi\omega$ instead of $a$.
\begin{NB}
  It should be possible to deduce these directly from \eqref{eq:g at a=m}.
\end{NB}%
In fact, we will find that it is even more natural to introduce a
function $\T$ given by
\begin{equation*}
  \T \defeq \frac13\left(u + \left(\frac\pi\omega\right)^2\right)
  = \frac13 \left(u - \frac14 \left(\pd{u}a \right)^2\right).
\end{equation*}
Up to constant multiple, this is the {\it contact term\/} for surfaces
in the physics literature, say in \cite{MW,LNS}. 
It will give the contribution of the intersection number $(\alpha^2)$
in Donaldson invariants in view of our formula in
\thmref{thm:partition}, thanks to \eqref{eq:9} proved just below.

The perturbative parts of $u$ and $\nicefrac14 \left(\nicefrac{\partial
    u}{\partial a} \right)^2$ cancel out, so the perturbative part of
$\T$ is $0$. An explicit computation shows
\begin{equation}\label{eq:Texpand}
  \T = \frac{1}{2a}\Lambda^3 + O(\Lambda^6).
\end{equation}

\begin{NB}
Kota introduced a new function $\lambda$ by
\begin{equation*}
  \lambda\defeq \frac{\frac{27}{4}(\Lambda/a)^3-(9u/a^2-8)}{4-3u/a^2}.
\end{equation*}
From the formulas (5.55) and (5.56), we have
\begin{equation*}
  \frac{u}{a^2} + \left(\frac1{2a}\pd{u}a\right)^2
  = -\frac{2}{3}(\lambda+1)(\lambda-2) + 
  \frac{2}{3}\lambda(\lambda-2)
  = \frac23 (-\lambda + 2).
\end{equation*}
From \eqref{eq:5} we have
\begin{equation*}
   \frac{u}{a^2} + \left(\frac1{2a}\pd{u}a\right)^2
   = \frac{3\Lambda^3}a
   \left(u + \left(\frac\pi\omega\right)^2\right)^{-1}
   = \frac{\Lambda^3}{a\T}.
\end{equation*}
Therefore $\lambda$ is essentially $aT$:
\begin{equation*}
   \lambda = 2 - \frac32\frac{\Lambda^3}{a\T}.
\end{equation*}
I do not know what is really `good' variable to work with, yet......
\end{NB}

By \eqref{eq:d^2F} together with $E_2(\tau) = 1$ when $a=m$, we have
\begin{equation}\label{eq:9}
  \frac{\partial^2 F_0}{\partial(\log\Lambda)^2}
  = -3\left( u + \left(\frac{\pi}\omega\right)^2\right)
  = -9\T.
\end{equation}
\begin{NB}
  \begin{equation*}
    = -\bgamma\pd{u}{\log\Lambda}
  \end{equation*}
\end{NB}

Since $\Delta$ vanishes at $a=m$, we have
\(
  \pd{\Delta}{a} + \pd{\Delta}{m} = 0.
\)
Therefore we get
\begin{equation*}
  0 = \left(3u^2 -2a^2 u - \frac92 a\Lambda^3\right)
  \left(\pd{u}a + \pd{u}m\right)
  - 2u^2 a - \frac92 u\Lambda^3 + 12a^2\Lambda^3
\end{equation*}
from \eqref{eq:disc}.
Using (\ref{eq:4}, \ref{eq:5}), we find
\begin{equation}\label{eq:10}
  \begin{split}
  - 2u^2 a - \frac92 u\Lambda^3 + 12a^2\Lambda^3  
  &= \frac4{\Lambda^3}
 \left(\frac\pi\omega\right)^6 \T,
% \frac13 \left(u + \left(\frac\pi\omega\right)^2\right),
\\
  3u^2 -2a^2 u - \frac92 a\Lambda^3
  &= -\frac4{\Lambda^6}
  \left(\frac\pi\omega\right)^6 \T^2.
%  \frac19\left(u + \left(\frac\pi\omega\right)^2\right)^2.
\end{split}
\end{equation}
\begin{NB}
We have
\begin{equation*}
  \begin{split}
  & - 2u^2 a - \frac92 u\Lambda^3 + 12a^2\Lambda^3
  \overset{\eqref{eq:5}}{=} -\frac{2}{3\Lambda^3}u^2 \left(
    u^2 - \left(\frac\pi\omega\right)^4
  \right)
  - \frac92 u\Lambda^3
  + 12 \Lambda^3 
  \frac{1}{9\Lambda^6} \left(
    u^2 - \left(\frac\pi\omega\right)^4
  \right)^2
\\
  =\; &  \frac1{3\Lambda^3}
  \left[
    -2 u^2\left(u^2 - \left(\frac\pi\omega\right)^4\right)
    + 4 \left(u^2 - \left(\frac\pi\omega\right)^4\right)^2
    - \frac{27}2u\Lambda^6
    \right]
\\
  \overset{\eqref{eq:4}}{=}\; &  \frac2{3\Lambda^3}
  \left[
    - u^2\left(u^2 - \left(\frac\pi\omega\right)^4\right)
    + 2 \left(u^2 - \left(\frac\pi\omega\right)^4\right)^2
    - u\left(u + \left(\frac\pi\omega\right)^2\right)^2
    \left(u - 2\left(\frac\pi\omega\right)^2\right)
    \right]
\\
  =\; &  \frac4{3\Lambda^3}
 \left(\frac\pi\omega\right)^6
 \left(u + \left(\frac\pi\omega\right)^2\right).
  \end{split}
\end{equation*}

We have
\begin{equation*}
  \begin{split}
  & 3u^2 -2a^2 u - \frac92 a\Lambda^3
  \overset{\eqref{eq:5}}{=} 
  3u^2 - 2u\frac{\left(u^2 - \left(\nicefrac\pi\omega\right)^4\right)^2}
  {9\Lambda^6}
  -\frac32 \left(u^2 - \left(\nicefrac\pi\omega\right)^4\right)
\\
  =\; &
  \frac1{9\Lambda^6}\left[
  \frac{27}2 \Lambda^6 \left(u^2 + \left(\frac\pi\omega\right)^4\right)
  - 2u^5 + 4u^3 \left(\frac\pi\omega\right)^4
  - 2u \left(\frac\pi\omega\right)^8
  \right]
\\
  \overset{\eqref{eq:4}}{=} \; &
  \frac2{9\Lambda^6}\left[
   \left(u + \left(\frac\pi\omega\right)^2\right)^2
   \left(u - 2\left(\frac\pi\omega\right)^2\right)
   \left(u^2 + \left(\frac\pi\omega\right)^4\right)
  - u^5 + 2u^3 \left(\frac\pi\omega\right)^4
  - u \left(\frac\pi\omega\right)^8
  \right]
\\
  {=} \; &
  \frac2{9\Lambda^6}\left[
    - 2 u^2\left(\frac\pi\omega\right)^6
    - 4 u \left(\frac\pi\omega\right)^8
    - 2 \left(\frac\pi\omega\right)^{10}
  \right]
\\
  {=} \; &
  -\frac4{9\Lambda^6}
  \left(\frac\pi\omega\right)^6
  \left(u + \left(\frac\pi\omega\right)^2\right)^2.
  \end{split}
\end{equation*}
\end{NB}%
Therefore
\begin{equation}\label{eq:7}
    -\frac{1}{\bgamma}%\left\{
%      \frac12
      \left(
      \frac{\partial^2 F_0}{\partial\log\Lambda\partial m}
      - 2{m}\right)
%      \right\}
   -  \frac{2\pi\sqrt{-1}}\omega
   =
    \pd{u}a + \pd{u}m
    = \Lambda^3 \T^{-1}
%    3\left({u} + \left(\frac{\pi}\omega\right)^2\right)^{-1}.
\end{equation}
\begin{NB}
  Or, by using \eqref{eq:5}
  \begin{equation*}
    \pd{u}a + \pd{u}m
    = \frac1a
    \left({u} - \left(\frac{\pi}\omega\right)^2\right).
  \end{equation*}
This is nothing but Kota's equation (4.36)
originally deduced from the contact term equation and the homogeneity.

Or this is equal to
\begin{equation*}
  \frac{\sqrt{3}}2 \sqrt{u - 2\left(\frac\pi\omega\right)^2}.
\end{equation*}
by \eqref{eq:4}.
\end{NB}

\begin{NB}
  Another way to prove \eqref{eq:7}:

By \eqref{eq:2} and \eqref{eq:3} we have
\begin{equation*}
  \left(\frac{1}{\bgamma}\left\{
      \frac12
      \left(
      \frac{\partial^2 F_0}{\partial\log\Lambda\partial m}
      - 2{m}\right)
      \right\}
   +  \frac{\pi\sqrt{-1}}\omega \right)^2
   = \frac13\left(u - 2\left(\frac\pi\omega\right)^2\right).
\end{equation*}
The left hand side is
\(
  \frac14 \left(
    \pd{u}a + \pd{u}m
    \right)^2.
\)
The right hand side is 
\(
  \nicefrac94 \Lambda^6 \left( 
    {u} + \left(\nicefrac{\pi}\omega\right)^2
  \right)^{-2}
\)
by \eqref{eq:4}. Further using \eqref{eq:5} we find
\begin{equation*}
  \pm \left(
    \pd{u}a + \pd{u}m
    \right)
    = \frac1a \left(
      u - \left(\frac\pi\omega\right)^2
    \right).
\end{equation*}
By comparing perturbative parts, we find that we must take $+$.
\end{NB}%

Plugging \eqref{eq:7} to \eqref{eq:2}, we obtain
\begin{equation*}
  \frac{\pi}\omega
  \cot\left(-
       \frac1{4\sqrt{-1}}\left(
    \frac{\partial^2 F_0}{\partial {a}\partial m} + 
    \frac{\partial^2 F_0}{\partial {a}^2}
    \right)
     \right)
     = \frac12\Lambda^3 \T^{-1}.
%    3\left({u} + \left(\frac{\pi}\omega\right)^2\right)^{-1}.
\end{equation*}
The left hand side is
\begin{equation*}
  \frac{\pi\sqrt{-1}}\omega
  \frac{
        \exp\left[-\frac12\left(
    \frac{\partial^2 F_0}{\partial {a}\partial m} + 
    \frac{\partial^2 F_0}{\partial {a}^2}
    \right)\right]
  + 1
    }{
        \exp\left[-\frac12\left(
    \frac{\partial^2 F_0}{\partial {a}\partial m} + 
    \frac{\partial^2 F_0}{\partial {a}^2}
    \right)\right]
  - 1
    }.
\end{equation*}
\begin{NB}
\begin{equation*}
  \begin{split}
  & \frac{\pi}\omega
  \frac{
    \left(\exp\left[-\frac14\left(
    \frac{\partial^2 F_0}{\partial {a}\partial m} + 
    \frac{\partial^2 F_0}{\partial {a}^2}
    \right)\right]
  +
    \exp\left[\frac14\left(
    \frac{\partial^2 F_0}{\partial {a}\partial m} + 
    \frac{\partial^2 F_0}{\partial {a}^2}
    \right)\right]\right)/2
  }{
    \left(\exp\left[-\frac14\left(
    \frac{\partial^2 F_0}{\partial {a}\partial m} + 
    \frac{\partial^2 F_0}{\partial {a}^2}
    \right)\right]
  -
    \exp\left[\frac14\left(
    \frac{\partial^2 F_0}{\partial {a}\partial m} + 
    \frac{\partial^2 F_0}{\partial {a}^2}
    \right)\right]\right)/2\sqrt{-1}
  }
\\
  =\; & \frac{\pi\sqrt{-1}}\omega
  \frac{
        \exp\left[-\frac12\left(
    \frac{\partial^2 F_0}{\partial {a}\partial m} + 
    \frac{\partial^2 F_0}{\partial {a}^2}
    \right)\right]
  + 1
    }{
        \exp\left[-\frac12\left(
    \frac{\partial^2 F_0}{\partial {a}\partial m} + 
    \frac{\partial^2 F_0}{\partial {a}^2}
    \right)\right]
  - 1
    }.
  \end{split}
\end{equation*}
\end{NB}%
Hence
%\begin{equation}
\begin{multline}
\label{eq:8}
    \exp\left[-\frac12\left(
    \frac{\partial^2 F_0}{\partial {a}\partial m} + 
    \frac{\partial^2 F_0}{\partial {a}^2}
    \right)\right]
\\
  = - \left(\frac{2\pi\sqrt{-1}}\omega + \Lambda^3
    \T^{-1}
%    3\left({u} + \left(\frac{\pi}\omega\right)^2\right)^{-1}
  \right)
  \Biggm/
  \left(\frac{2\pi\sqrt{-1}}\omega - \Lambda^3\T^{-1}
%    3\left({u} + \left(\frac{\pi}\omega\right)^2\right)^{-1}
  \right)
\\
  = \frac14 T^{-1}
  \left(\frac{2\pi\sqrt{-1}}\omega + \Lambda^3
    \T^{-1}
%    3\left({u} + \left(\frac{\pi}\omega\right)^2\right)^{-1}
  \right)^2,
%\end{equation}
\end{multline}
where we have used \eqref{eq:4} in the last equality.
\begin{NB}
  We have
\begin{equation*}
    \left(\frac{2\pi\sqrt{-1}}\omega\right)^2 - \Lambda^6 \T^{-2}
    = - 4\left(\frac\pi\omega\right)^2
    -\frac43 \left(u - 2\left(\frac\pi\omega\right)^2\right)
    = -\frac43 \left(u + \left(\frac\pi\omega\right)^2\right)
    = - 4 \T
\end{equation*}
by \eqref{eq:4}.
\end{NB}%

\begin{NB}
  From \eqref{eq:5}, the middle line is equal to
  \begin{equation*}
    - \left(\frac{\pi\sqrt{-1}}\omega + \frac1{2a}
    \left({u} - \left(\frac{\pi}\omega\right)^2\right)\right)
  \Biggm/
  \left(\frac{\pi\sqrt{-1}}\omega - \frac1{2a}
    \left({u} - \left(\frac{\pi}\omega\right)^2\right)\right).
  \end{equation*}
Or
  \begin{equation*}
    -\left(-\frac12 \pd{u}a + \frac12\left(\pd{u}a + \pd{u}m\right)\right)
      \Biggm/
    \left(-\frac12 \pd{u}a - \frac12\left(\pd{u}a - \pd{u}m\right)\right)
    = 
    \frac12 \pd{u}m
      \Biggm/
    \left(\pd{u}a + \frac12 \pd{u}m\right).
  \end{equation*}
This is consistent with Kota's (4.45).
\end{NB}

\begin{NB}
The following seems unnecessary.

From \eqref{eq:7} we have
\begin{equation*}
   \pd{u}m = 3\Lambda^3
   \left({u} + \left(\frac{\pi}\omega\right)^2\right)^{-1}
   + \frac{\sqrt{-1}\pi}{\omega}.
\end{equation*}
\end{NB}

By \eqref{eq:1} and \eqref{eq:3} we have
\begin{equation}\label{eq:6}
      \exp \Biggl[
%    - \frac{\pi\sqrt{-1}}2
    - \frac14 \left(\frac{\partial^2 F_0}{\partial m^2}
    + 2\frac{\partial^2 F_0}{\partial a\partial m}
    + \frac{\partial^2 F_0}{\partial a^2}\right) \Biggr]
  = - \frac1{\Lambda^2} \T.
%   \frac13
%   \left(
%     u + \left(\frac\pi\omega\right)^2
%   \right).
\end{equation}

By \eqref{eq:8} and \eqref{eq:6} we obtain
\begin{equation}
%\begin{multline}
\label{eq:11}
      \exp \Biggl[
%    - \frac{\pi\sqrt{-1}}2
    - \frac12 \left(\frac{\partial^2 F_0}{\partial m^2}
    + \frac{\partial^2 F_0}{\partial a\partial m}
     \right) \Biggr]
\begin{NB}
  = - \frac{\T^2}{\Lambda^4}\,
%   \left(
%     u + \left(\frac\pi\omega\right)^2
%   \right)^2
  \frac{%\left(
      \frac{2\pi\sqrt{-1}}\omega - \Lambda^3 \T^{-1}
%    3\left({u} + \left(\frac{\pi}\omega\right)^2\right)^{-1}
    %\right)
   }
   {%\left(
       \frac{2\pi\sqrt{-1}}\omega + \Lambda^3 \T^{-1}
%    3\left({u} + \left(\frac{\pi}\omega\right)^2\right)^{-1}
   %\right)
   }
%\\
\end{NB}%
  = \frac{4\T^3}{\Lambda^4}\,
  \left(
       \frac{2\pi\sqrt{-1}}\omega + \Lambda^3 \T^{-1}
  \right)^{-2}.
%\end{multline}
\end{equation}

\subsection{Computation of instanton parts}\label{subsec:inst}

Since we will express Mochizuki's formula in terms of {\it instanton
  parts\/} of derivatives of $F_0$ and $A$, $B$, we need to compute
them. Since their perturbative parts are explicit functions, we
just subtract them from the full partition functions. We denote
instanton parts by putting `inst' as sub/superscripts.

\begin{NB}
  The coefficient of $x$:
\end{NB}
We have
\begin{equation}
  \frac1\bgamma\pd{\Fin_0}{\log\Lambda}
  = \frac1\bgamma\left(\pd{F_0}{\log\Lambda} + 2a^2\right)
  = - u + a^2
\end{equation}
from the perturbative part of $\nicefrac{\partial F_0}{\log\Lambda}$
and the definition of $u$ in \eqref{eq:u}.
\begin{NB}
  This $a^2$ cancels with $-s^2$ in (4.12).
\end{NB}%

\begin{NB}
  The coefficient of $((\xi_2-\xi_1)^2)$:
\end{NB}
Since
\(
  \exp \left[
%    - \frac{\pi\sqrt{-1}}2
    -\nicefrac14 \left(\nicefrac{\partial^2 F_0}{\partial m^2}
    + 2\nicefrac{\partial^2 F_0}{\partial a\partial m}
    + \nicefrac{\partial^2 F_0}{\partial a^2}\right) \right]
\)
has
\(
  -\left(
    \nicefrac{2a}\Lambda
  \right)^{-1}
\)
as the perturbative part, we get
\begin{equation}\label{eq:xi^2}
  \exp \Biggl[
%    - \frac{\pi\sqrt{-1}}2
    -\frac14 \left(\frac{\partial^2 \Fin_0}{\partial m^2}
    + 2\frac{\partial^2 \Fin_0}{\partial a\partial m}
    + \frac{\partial^2 \Fin_0}{\partial a^2}\right) \Biggr]
  = 
    \frac{2a}{\Lambda^{3}}
    \T
%     \left(
%       u + \left(\frac\pi\omega\right)^2
%       \right)
\end{equation}
from \eqref{eq:6}.
\begin{NB}
  We have
  \begin{equation*}
%    \begin{split}
%    & 
    \exp\left[
    -\frac{1}{4}\left(\frac{\partial^2 F_0^{\mathrm{pert}}}{\partial a_2^2}
   + 2\frac{\partial^2 F_0^{\mathrm{pert}}}{\partial a \partial m}
   + \frac{\partial^2 F_0^{\mathrm{pert}}}{\partial m^2}\right)\right]
  = \exp\left[
  - 2 \log\frac{-2\sqrt{-1}a}{\Lambda}
  + \log\left(
    \frac{a+m}\Lambda
  \right)\right]
%\\
   =
%\; &
  \left(\frac{-2\sqrt{-1}a}{\Lambda}\right)^{-2}
  (\frac{a+m}\Lambda)
%    \end{split}
  \end{equation*}
If we substitute $m=a$, we get
\begin{equation*}
  -\left(\frac{2a}\Lambda\right)^{-1}.
\end{equation*}
\end{NB}%
Note that the left hand side starts with $1$ as a formal power series
in $\Lambda$. This is compatible with the expansion of the right hand
side in \eqref{eq:Texpand}.
\begin{NB}
  Since $\T = \nicefrac1{2a}\Lambda^3 + \cdots$, we have
$\nicefrac{2a}{\Lambda^3} T = 1 + \cdots$.
\end{NB}%

In the same way, we get
\begin{equation}\label{eq:12}
%\begin{multline}
  \exp\left[
    -\frac12\left(\frac{\partial^2 \Fin_0}{\partial a\partial m}
  + \frac{\partial^2 \Fin_0}{\partial a^2}
   \right)
  \right]
%\\
   \begin{NB}
  =
  - \left(\frac{2a}\Lambda\right)^{3}
%   \frac{\left(\frac{\pi\sqrt{-1}}\omega - \frac32\Lambda^3
%     \left({u} + \left(\frac{\pi}\omega\right)^2\right)^{-1}\right)}
%    {\left(\frac{\pi\sqrt{-1}}\omega + \frac32\Lambda^3
%     \left({u} + \left(\frac{\pi}\omega\right)^2\right)^{-1}\right)},
  \frac{%\left(
      \frac{2\pi\sqrt{-1}}\omega + \Lambda^3\T^{-1}
%    3\left({u} + \left(\frac{\pi}\omega\right)^2\right)^{-1}
    %\right)
   }
   {%\left(
       \frac{2\pi\sqrt{-1}}\omega - \Lambda^3 \T^{-1}
%    3\left({u} + \left(\frac{\pi}\omega\right)^2\right)^{-1}
   %\right)
   }
   \end{NB}%
  = \frac14 \left(\frac{2a}\Lambda\right)^{3} T^{-1}
  \left(\frac{2\pi\sqrt{-1}}\omega + \Lambda^3
    \T^{-1}
%    3\left({u} + \left(\frac{\pi}\omega\right)^2\right)^{-1}
  \right)^2
%\end{multline}
\end{equation}
\begin{NB}
The perturbative part is given by
  \begin{equation*}
    \exp\left[
      -\frac12\left(\frac{\partial^2 F_0^{\mathrm{pert}}}{\partial a\partial m}
    + \frac{\partial^2 F_0^{\mathrm{pert}}}{\partial a^2}
     \right)
    \right]
    =  \left(\frac{-2\sqrt{-1}a}\Lambda\right)^{-4}
    \left(\frac{2a}\Lambda\right)
    = \left(\frac{2a}\Lambda\right)^{-3}.
  \end{equation*}
\end{NB}%
\begin{NB}
  Since the left hand side starts with $1$, we must have
  \begin{equation*}
    \left(\frac{2\pi\sqrt{-1}}\omega + \Lambda^3
    \T^{-1}
%    3\left({u} + \left(\frac{\pi}\omega\right)^2\right)^{-1}
  \right)
  = \frac{2\Lambda^3}{(2a)^2} + \cdots.
  \end{equation*}
\end{NB}%
from \eqref{eq:8}, and
\begin{equation}
%\begin{multline}
\label{eq:13}
  \exp\left[
    -\frac12\left(\frac{\partial^2 \Fin_0}{\partial m^2}
  + \frac{\partial^2 \Fin_0}{\partial a\partial m}
   \right)
  \right]
%\\
\begin{NB}
  =
  - \frac1{2\Lambda^3 a}
  \T^2\,
% \left(
%     u + \left(\frac\pi\omega\right)^2
%   \right)^2
%   \frac{\left(\frac{\pi\sqrt{-1}}\omega - \frac32\Lambda^3
%     \left({u} + \left(\frac{\pi}\omega\right)^2\right)^{-1}\right)}
%    {\left(\frac{\pi\sqrt{-1}}\omega + \frac32\Lambda^3
%     \left({u} + \left(\frac{\pi}\omega\right)^2\right)^{-1}\right)},
  \frac{%\left(
      \frac{2\pi\sqrt{-1}}\omega - \Lambda^3 \T^{-1}
%   3 \left({u} + \left(\frac{\pi}\omega\right)^2\right)^{-1}
    %\right)
   }
   {%\left(
       \frac{2\pi\sqrt{-1}}\omega + \Lambda^3 \T^{-1}
%   3 \left({u} + \left(\frac{\pi}\omega\right)^2\right)^{-1}
   %\right)
   }
\end{NB}%
  = \frac{2\T^3}{\Lambda^3 a}
    \left(
       \frac{2\pi\sqrt{-1}}\omega + \Lambda^3 \T^{-1}
  \right)^{-2}
\end{equation}
%\end{multline}
\begin{NB}
The perturbative part is given by
  \begin{equation*}
    \exp\left[
      -\frac12\left(\frac{\partial^2 F_0^{\mathrm{pert}}}{\partial m^2}
    + \frac{\partial^2 F_0^{\mathrm{pert}}}{\partial a\partial m}
     \right)
    \right]
    =  \left(\frac{2a}\Lambda\right).
  \end{equation*}
\end{NB}%
\begin{NB}
\begin{equation*}
    \frac{2\T^3}{\Lambda^3 a}
  \left(
     \frac{2\pi\sqrt{-1}}\omega + \Lambda^3 \T^{-1}
\right)^{-2}
  = \frac{2}{\Lambda^3 a} \left( \frac{\Lambda^3}{2a} + \cdots \right)^2
  \left( \frac{(2a)^2}{2\Lambda^3} + \cdots \right)^2
  = 1 + \cdots.
\end{equation*}
\end{NB}%
from \eqref{eq:11}.

We need a trick to compute the instanton part of
$q = \exp\left(-\nicefrac12\nicefrac{\partial^2 F_0}{\partial a^2}\right)$, 
since it vanishes at $a=m$.
For the moment we no longer set $a = m$ and consider
\begin{equation*}
  \begin{split}
  q_{\mathrm{inst}}^2 &=
  \exp\left(- \frac{\partial^2 \Fin_0}{\partial a^2}\right)
   = q^2
%   \exp\left(-\frac12\frac{\partial^2 F_0}{\partial a^2}\right)
  \left(\frac{-2\sqrt{-1}a}{\Lambda}\right)^{8}
  \left(\frac{(m+a)(m-a)}{\Lambda^2}\right)^{-1}
\\
  &= \frac{q^2}{m-a}
  \left(\frac{-2\sqrt{-1}a}{\Lambda}\right)^{8}
  \left(\frac{m+a}{\Lambda^2}\right)^{-1}.
  \end{split}
\end{equation*}
\begin{NB}
The perturbative part of $\nicefrac{\partial^2
  F_0}{\partial a^2}$ is
\(
8 \log\nicefrac{-2\sqrt{-1}a}{\Lambda}
  - \log\nicefrac{(a+m)(-a + m)}{\Lambda^2}.
\)
\end{NB}
Since $q$ vanishes at $a=m$, we get
\begin{equation*}
  \left.q_{\mathrm{inst}}^2 \right|_{m=a} =
  - \Lambda
  \left.\pd{(q^2)}a \right|_{m=a}
%    \left(\frac{-2\sqrt{-1}a}{\Lambda}\right)^{-8}
  \left(\frac{2a}{\Lambda}\right)^{7}.
\end{equation*}
\begin{NB}
  $\deg\Lambda = 1$, $\deg\pd{(q^2)}a = -1$.
\end{NB}%
The discriminant $\Delta$ has an expansion 
\(
   \omega^{12} \Delta = (2\pi)^{12}(q^2 - 24 q^4 + \cdots),
\)
\begin{NB}
  The degree of $\Delta$ is $12$. The degree of $\omega$ is $-1$.
\end{NB}%
so
\begin{equation*}
  \omega^{12} \left.\pd{\Delta}{a} \right|_{m=a}
  = \left.\pd{}{a}(\omega^{12}\Delta) \right|_{m=a}
  =  (2\pi)^{12} \left.\pd{(q^2)}a \right|_{m=a}.
\end{equation*}
We differentiate \eqref{eq:disc} by $a$ to get
\begin{equation}\label{eq:derivDelta}
  \left.\pd{\Delta}a\right|_{m=a}
  = -16\Lambda^6
  \left( 3 u^2 - 2 a^2 u - \frac92 a\Lambda^3\right)\pd{u}a
  = -{\sqrt{-1}} \left(\frac{2\pi}\omega\right)^7
  \T^2,
%  \left(u + \left(\frac\pi\omega\right)^2\right)^2,
\end{equation}
where we have used \eqref{eq:10}. Therefore
\begin{equation}\label{eq:qinst}
  q_{\mathrm{inst}}^2 =
  \exp\left(- \frac{\partial^2 \Fin_0}{\partial a^2}\right)
  = \Lambda
  {\sqrt{-1}} 
  \left(\frac{2\pi}\omega\right)^{-5}
  \left(\frac{2a}{\Lambda}\right)^{7}
  \T^2.
\end{equation}
\begin{NB}
  $\deg T = 2$, $\deg \pi/\omega = 1$, $\deg\Lambda = 1$. $\deg
  q_{\mathrm{inst}} = 0$. Therefore the degree matches.

  Let us compute the first two terms:
  \begin{equation*}
      q_{\mathrm{inst}}^2 =
      \Lambda \sqrt{-1} (\sqrt{-1}\cdot 2a + \cdots)^{-5}
        \left(\frac{2a}{\Lambda}\right)^{7}
        \left(\frac{\Lambda^3}{2a}+\cdots \right)^2
      = -1 + \cdots.
  \end{equation*}
There was a sign mistake in \eqref{eq:10}. Corrected on Dec.\ 9.
\end{NB}

It is to be understood that all functions are evaluated at $a=m$
unless an equation contains an expression `$m-a$'.

Substituting \eqref{eq:qinst} into \eqref{eq:12} we get
\begin{equation*}
  \exp\left(- \frac{\partial^2 \Fin_0}{\partial a\partial m}\right)
  =
  \frac{1}{\sqrt{-1}\Lambda}\T^{-4}
  \left(\frac{2a}\Lambda\right)^{-1}
   \left(\frac{2\pi}\omega\right)^{5}
   \left(\frac{\pi\sqrt{-1}}\omega + \frac{\Lambda^3}{2\T}
%    3\left({u} + \left(\frac{\pi}\omega\right)^2\right)^{-1}
  \right)^{4}.
\end{equation*}
\begin{NB}
Original version (it was wrong):
  \begin{equation*}
  = 
  \frac1{\sqrt{-1}\Lambda} 
  \left(\frac{2a}\Lambda\right)^{-1}
%   \frac{\left(\frac{\pi\sqrt{-1}}\omega - \frac32\Lambda^3
%     \left({u} + \left(\frac{\pi}\omega\right)^2\right)^{-1}\right)}
%    {\left(\frac{\pi\sqrt{-1}}\omega + \frac32\Lambda^3
%     \left({u} + \left(\frac{\pi}\omega\right)^2\right)^{-1}\right)},
  \left(
  \frac{%\left(
      \frac{2\pi\sqrt{-1}}\omega - \Lambda^3 \T^{-1}
%    3\left({u} + \left(\frac{\pi}\omega\right)^2\right)^{-1}
    %\right)
   }
   {%\left(
       \frac{2\pi\sqrt{-1}}\omega + \Lambda^3 \T^{-1}
%    3\left({u} + \left(\frac{\pi}\omega\right)^2\right)^{-1}
   %\right)
   }
   \right)^2
   \left(\frac{2\pi}\omega\right)^{5}
  \T^{-2}
  \end{equation*}
\end{NB}%
\begin{NB}
  \begin{equation*}
    \begin{split}
  &  \frac{1}{16\sqrt{-1}\Lambda}\T^{-4}
  \left(\frac{2a}\Lambda\right)^{-1}
   \left(\frac{2\pi}\omega\right)^{5}
   \left(\frac{2\pi\sqrt{-1}}\omega + \Lambda^3
    \T^{-1}
%    3\left({u} + \left(\frac{\pi}\omega\right)^2\right)^{-1}
  \right)^{4}
\\
  =\;&
  \frac{1}{16\sqrt{-1}\Lambda}\left(\frac{2a}{\Lambda^3} + \cdots \right)^{4}
  \left(\frac{2a}\Lambda\right)^{-1}
   \left({2\sqrt{-1}a} + \cdots \right)^{5}
   \left(\frac{2\Lambda^3}{(2a)^2} + \cdots \right)^{4}
   = 1 + \cdots.
    \end{split}
  \end{equation*}
\end{NB}%
Then we substitute this into \eqref{eq:13} to get
\begin{equation}\label{eq:d^2Fdm^2}
    \exp\left(- \frac{\partial^2 \Fin_0}{\partial m^2}\right)
  = 
  \frac{\sqrt{-1}}{\Lambda^7}
  \T^{10}
  \left(\frac{2a}\Lambda\right)^{-1}
   \left(\frac{2\pi}\omega\right)^{-5}
      \left(\frac{\pi\sqrt{-1}}\omega + \frac{\Lambda^3}{2 \T}
%    3\left({u} + \left(\frac{\pi}\omega\right)^2\right)^{-1}
  \right)^{-8}
  .
\end{equation}
\begin{NB}
Original version (it was wrong):
\begin{equation*}
    \frac{\sqrt{-1}}{\Lambda^7}
  \left(\frac{2a}\Lambda\right)^{-1}
   \left(\frac{2\pi}\omega\right)^{-5}
  \T^{6}.
\end{equation*}
\begin{NB2}
  \begin{equation*}
  \left[
  - \frac1{2\Lambda^3 a}
  \T^2\,
% \left(
%     u + \left(\frac\pi\omega\right)^2
%   \right)^2
%   \frac{\left(\frac{\pi\sqrt{-1}}\omega - \frac32\Lambda^3
%     \left({u} + \left(\frac{\pi}\omega\right)^2\right)^{-1}\right)}
%    {\left(\frac{\pi\sqrt{-1}}\omega + \frac32\Lambda^3
%     \left({u} + \left(\frac{\pi}\omega\right)^2\right)^{-1}\right)},
  \frac{%\left(
      \frac{2\pi\sqrt{-1}}\omega - \Lambda^3 \T^{-1}
%    3\left({u} + \left(\frac{\pi}\omega\right)^2\right)^{-1}
    %\right)
   }
   {%\left(
       \frac{2\pi\sqrt{-1}}\omega + \Lambda^3 \T^{-1}
%   3 \left({u} + \left(\frac{\pi}\omega\right)^2\right)^{-1}
   %\right)
   }
   \right]^2
   \times
  {\sqrt{-1}}\Lambda
  \left(\frac{2a}\Lambda\right)
  \left(
  \frac{%\left(
      \frac{2\pi\sqrt{-1}}\omega - \Lambda^3 \T^{-1}
%    3\left({u} + \left(\frac{\pi}\omega\right)^2\right)^{-1}
    %\right)
   }
   {%\left(
       \frac{2\pi\sqrt{-1}}\omega + \Lambda^3 \T^{-1}
%    \left({u} + \left(\frac{\pi}\omega\right)^2\right)^{-1}
   %\right)
   }
   \right)^{-2}
   \left(\frac{2\pi}\omega\right)^{-5}
  \T^{2}.
  \end{equation*}
\end{NB2}
\end{NB}%

Let us consider instanton parts of other derivatives: we have
\begin{NB}
  The coefficient of $(\xi_2-\xi_1,\alpha)$:
\end{NB}%
\begin{equation}\label{eq:d^2FdadL}
  \frac{\partial^2 \Fin_0}{\partial a\partial\log\Lambda}
  + \frac{\partial^2 \Fin_0}{\partial m\partial\log\Lambda}
  = 6a - 3\Lambda^3\T^{-1}
\end{equation}
\begin{NB}
  \begin{equation*}
    = 3\left(\pd{}{a} + \pd{}{m}\right)\left(a^2 - u\right)
    = 6a - 3\left( \pd{u}a + \pd{u}m
    \right)
  \end{equation*}
\end{NB}%
from the definition \eqref{eq:u} of $u$ and \eqref{eq:7}.
\begin{NB}
  This $6a$, after dividing out by $6$, will be canceled with
  $-(\xi_2-\xi_1,\alpha)sz$ in (4.12). And the term
  $-9\Lambda^3 T^{-1}$ will vanish at $a=\infty$ after the substitution
  $\Lambda = \Lambda^{4/3}a^{-1/3}$.
\end{NB}
We also have
\begin{NB}
  the coefficient of $(\tilde\xi,\alpha)z$:
\end{NB}%
\begin{equation}\label{eq:d^2FdmdL}
  \frac{\partial^2 \Fin_0}{\partial m \partial\log\Lambda}
  = -3 \left(\Lambda^3 \T^{-1}
%  3\left(u + \left(\frac{\pi}\omega\right)^2\right)^{-1}
  + \frac{2\pi\sqrt{-1}}\omega\right),
\end{equation}
from the definition of $\omega$ \eqref{eq:duda} and its perturbative
part.
\begin{NB}
  \begin{equation*}
    \mathrm{LHS}
    = 6a - 3\Lambda^3\T^{-1}
%    3\left(u + \left(\frac{\pi}\omega\right)^2\right)^{-1}
  - \frac{\partial^2 \Fin_0}{\partial a\partial\log\Lambda}
  = - 3\Lambda^3 \T^{-1}
%  3\left(u + \left(\frac{\pi}\omega\right)^2\right)^{-1}
  - \frac{\partial^2 F_0}{\partial a\partial\log\Lambda}
  \end{equation*}
\end{NB}

\begin{NB}
  The coefficient of $(\alpha^2)z^2$ is $\nicefrac{\partial^2
    \Fin_0}{\partial(\log\Lambda)^2}$, but this does not have the
  perturbative part, and hence \eqref{eq:9} is enough.
\end{NB}

\begin{NB}
  Part of the coefficient of $\chi(X)$:
\end{NB}
From \eqref{eq:genus1} we have
\begin{equation}\label{eq:Ain}
  \exp\Ain = \left( \frac1{2a} \pd{u}a \right)^{1/2}
  = \left(\frac1{\sqrt{-1}a}\frac{\pi}\omega\right)^{1/2}.
\end{equation}
\begin{NB}
Since the perturbative part
\(
  A^{\mathrm{pert}}
\)
of $A$ is
\(
   \nicefrac12 \log\left( \nicefrac{-2\sqrt{-1}a}\Lambda\right).
\)
\end{NB}%
\begin{NB}
  Since $\nicefrac\pi\omega = \sqrt{-1}a+\cdots$, $\exp\Ain = 1 + \cdots$.
\end{NB}%

In order to compute the instanton part of $B$, we use the same
technique as for $q$, since $B$ also vanishes at $a=m$. The
perturbative part
\begin{NB}
\(
  B^{\mathrm{pert}}
\)
\end{NB}
of $B$ is
\(
   \nicefrac12 \log\left( \nicefrac{-2\sqrt{-1}a}\Lambda\right)
   + \nicefrac18 \log\left(\nicefrac{(m-a)(m+a)}{\Lambda^2}\right).
\)
Thus we have
\begin{equation*}
  \exp \Bin
  = \sqrt{-\pi}\Lambda^{-3/2}
    \left(\frac{\Delta}{m-a}\right)^{1/8}
    \left(\frac{-2\sqrt{-1}a}\Lambda\right)^{-1/2}
    \left(\frac{m+a}{\Lambda^2}\right)^{-1/8}
\end{equation*}
from \eqref{eq:genus1}.
Therefore at $m=a$, we have
\begin{equation*}
  \exp8\Bin
  = \Lambda^{-11}  \left( - \pd{\Delta}{a}\right)
    \left(\frac{2a}\Lambda\right)^{-5}.
\end{equation*}
Using \eqref{eq:derivDelta}, we get
\begin{equation}\label{eq:Bin}
  \exp 8\Bin = 
  {\sqrt{-1}}
  \Lambda^{-11} 
  \left(\frac{2\pi}\omega\right)^7
  \left(\frac{2a}\Lambda\right)^{-5}
  \T^2.
\end{equation}
\begin{NB}
  We have
  \begin{equation*}
  {\sqrt{-1}}
  \Lambda^{-11} 
  \left(\frac{2\pi}\omega\right)^7
  \left(\frac{2a}\Lambda\right)^{-5}
  \T^2
  =  {\sqrt{-1}}
  \Lambda^{-11} 
    \left(2\sqrt{-1} a + \cdots\right)^7
  \left(\frac{2a}\Lambda\right)^{-5}
  \left(\frac{\Lambda^3}{2a}+\cdots\right)^2
  = 1 + \cdots.
  \end{equation*}
\end{NB}%

\begin{NB}
\begin{equation*}
  B^{\mathrm{pert}} + \frac18 \frac{\partial^2 F_0^{\mathrm{pert}}}{\partial a^2}
  = \frac32 \log\frac{-2\sqrt{-1}a}\Lambda.
\end{equation*}
Therefore
\begin{equation*}
  \exp(B)\exp\left(\frac18 \frac{\partial^2 F_0}{\partial a^2}\right)
  = 
  \exp(\Bin)\exp\left(\frac18 \frac{\partial^2 \Fin_0}{\partial a^2}\right)
  \left(\frac{-2\sqrt{-1}a}\Lambda\right)^{3/2}.
\end{equation*}
\end{NB}%

\begin{NB}
  We have
  \begin{equation*}
    \begin{split}
   & \exp\left(8\Bin + \frac{\partial^2 \Fin_0}{\partial a^2}\right)
    = {\sqrt{-1}}
  \Lambda^{-11} 
  \left(\frac{2\pi}\omega\right)^7
  \left(\frac{2a}\Lambda\right)^{-5}
  \T^2 \times
  \Lambda^{-1}
  \frac1{\sqrt{-1}} 
  \left(\frac{2\pi}\omega\right)^{5}
  \left(\frac{2a}{\Lambda}\right)^{-7}
  \T^{-2}
\\  
  =\; &
  \left(\frac{\pi}{a\omega}\right)^{12}.
    \end{split}
  \end{equation*}
This is the same as Kota's (5.37).
\end{NB}%

\begin{NB}
Combining (\ref{eq:d^2Fdm^2}, \ref{eq:Ain}, \ref{eq:Bin}) all
together, we get
\begin{equation*}
  \begin{split}
  \exp \left[8(\Bin - \Ain) + \frac{\partial^2\Fin_0}{\partial m^2}\right]
  & = \T^{-8} \left(\frac{2\pi}\omega\right)^{8}
  \left(
    \frac{\pi\sqrt{-1}}\omega + \frac{\Lambda^3}{2\T}
  \right)^8
\\
  & = \left(\frac{2\pi}\omega\right)^{8}
  \left(
    \frac{\pi\sqrt{-1}}\omega - \frac{\Lambda^3}{2 \T}
  \right)^{-8},
  \end{split}
\end{equation*}
where we have used \eqref{eq:4} in the second equality. Computing the
leading term of the right hand side, we deduce
\begin{equation}\label{eq:BA}
  \begin{split}
  \exp \left[\Bin - \Ain + \frac18\frac{\partial^2\Fin_0}{\partial m^2}\right]
  & = \sqrt{-1} \T^{-1} \left(\frac{2\pi}\omega\right)
  \left(
    \frac{\pi\sqrt{-1}}\omega + \frac{\Lambda^3}{2\T}
  \right)
\\
  & = \sqrt{-1} \left(\frac{2\pi}\omega\right)
  \left(
    \frac{\pi\sqrt{-1}}\omega - \frac{\Lambda^3}{2 \T}
  \right)^{-1},
  \end{split}
\end{equation}
as the left hand side starts with $1$ as a formal power series in $\Lambda$.
\begin{NB2}
  We have
  \begin{equation*}
   \frac{2\pi}\omega = 2\sqrt{-1}a + \cdots, \quad
   \frac{\pi\sqrt{-1}}\omega = - a + \cdots, \quad
   \frac{\Lambda^3}{2\T} = a + \cdots.
  \end{equation*}
\end{NB2}%
\begin{NB2}
  The sign mistake is corrected. Dec.~21.
\end{NB2}

\begin{NB2}
Alternative way:  

From \eqref{eq:AB} we have
\begin{equation*}
  \exp (8(\Bin - \Ain))
  = \left(\frac{\theta'_{11}(0)}{\omega\Lambda}\right)^8 
    \left(\frac{(m-a)(m+a)}{\Lambda^2}\right)^{-1}.
\end{equation*}
We have
\begin{equation*}
  \left.\frac{(\theta'_{11}(0))^8}{m-a}\right|_{m=a}
  = (2\pi)^8 \left.\pd{(q^2)}{a}\right|_{m=a}
\end{equation*}
as $\theta'_{11}(0) = -2\pi q^{1/4} + \dots$.
Using \eqref{eq:derivDelta} and the equation one line above, we have
\begin{equation*}
  \left.\pd{(q^2)}{a}\right|_{m=a}
  = \sqrt{-1} \left(\frac{2\pi}\omega\right)^{-5} T^2.
\end{equation*}
Hence
\begin{equation*}
  \exp (8(\Bin - \Ain))
  = \sqrt{-1}\Lambda^{-7}\left(\frac{2\pi}\omega\right)^{3} 
  \left(\frac{2a}\Lambda\right)^{-1} T^2.
\end{equation*}
Combining with \eqref{eq:d^2Fdm^2} we get the same formula
\end{NB2}

We also have
\begin{equation}\label{eq:chiO}
  \exp\left[12\Ain - 8\Bin\right]
  = \sqrt{-1} \Lambda^5 \T^{-2} \left(\frac{2a}\Lambda\right)^{-1}
  \left(\frac{2\pi}\omega\right)^{-1}
\end{equation}
from (\ref{eq:Ain}, \ref{eq:Bin}).
\begin{NB2}
  We have
  \begin{equation*}
    \sqrt{-1}\Lambda^5 \left( \frac{\Lambda^3}{2a} + \cdots \right)^{-2}
    \left(\frac{2a}\Lambda\right)^{-1}
      \left(2\sqrt{-1}a + \cdots\right)^{-1}
      = 1 + \cdots.
  \end{equation*}
\end{NB2}
\end{NB}

\begin{NB}
\subsection{An intermediate check}

We have
\begin{equation*}
  \begin{split}
  & \exp\left(
    4\Ain + \frac{\partial^2 \Fin_0}{\partial m^2}
    \right)
   = \left(\frac1{\sqrt{-1}a}\frac{\pi}\omega\right)^{2}
     \frac{\Lambda^7}{\sqrt{-1}}
  \T^{-10}
  \left(\frac{2a}\Lambda\right)
   \left(\frac{2\pi}\omega\right)^{5}
      \left(\frac{\pi\sqrt{-1}}\omega + \frac{\Lambda^3}{2\T}
%    3\left({u} + \left(\frac{\pi}\omega\right)^2\right)^{-1}
  \right)^{8}
\\
  =\; &
  \Lambda^5 \T^{-10}
    \left(\frac{2a}\Lambda\right)^{-1}
    \left(\frac{2\pi}{\sqrt{-1}\omega}\right)^{7}
      \left(\frac{\pi\sqrt{-1}}\omega + \frac{\Lambda^3}{2\T}
%    3\left({u} + \left(\frac{\pi}\omega\right)^2\right)^{-1}
  \right)^{8}.
  \end{split}
\end{equation*}
\end{NB}

%\subsection{Replacement of $\Lambda$ by $\Lambda^{4/3}a^{-1/3}$}
\subsection{The variable $\phi$}

In the partition function, we need to substitute
$\Lambda^{4/3}a^{-1/3}$ into $\Lambda$.  We denote the substitution by
$\left.\bullet\right|_{\Lambda=\Lambda^{4/3}a^{-1/3}}$.

Let $\fT \defeq\left.\T\right|_{\Lambda=\Lambda^{4/3}a^{-1/3}}$. By
\eqref{eq:Texpand} it has the expansion $\nicefrac{\Lambda^4}{2a^2} +
\cdots$. So we can choose the branch of its square root so that it
starts as 
\( 
   \sqrt{\fT} = \nicefrac{\Lambda^2}{\sqrt{2}a} + \cdots.
\)
We set
\begin{equation}
  \label{eq:phi}
  \phi \defeq \frac{\sqrt{\fT}}{\Lambda}.
\end{equation}
\begin{NB}

  Kota's $\phi$ is given by $1 + \lambda = 3\phi^4$. By (5.63), we have
  \begin{equation*}
    \frac13(1 + \lambda)
    = \frac1{9\Lambda^4} \left( \pd{u}{\log\Lambda} \right)^2
    \overset{\eqref{eq:9}}{=} 
    \frac1{9\Lambda^4}
    \left(u + \left(\frac{\pi}\omega\right)^2\right)^2
    = \frac{\fT^2}{\Lambda^4}.
  \end{equation*}
Therefore
\begin{equation*}
  \phi = \frac{\sqrt{\fT}}{\Lambda}.
\end{equation*}
\end{NB}

From (\ref{eq:4},\ref{eq:5}) we have
\begin{align*}
   & \left(\left.{u}\right|_{\Lambda = \Lambda^{4/3}a^{-1/3}}
     + \left.\left(\frac{\pi}\omega\right)^2\right|_{\Lambda = \Lambda^{4/3}a^{-1/3}}
     \right)^2
  \left(\left.{u}\right|_{\Lambda = \Lambda^{4/3}a^{-1/3}}
    - 2 \left.\left(\frac{\pi}\omega\right)^2\right|_{\Lambda = \Lambda^{4/3}a^{-1/3}}
  \right)
  = \frac{27}4 \Lambda^8 a^{-2},
\\
   & \left(\left.{u}\right|_{\Lambda = \Lambda^{4/3}a^{-1/3}}
     + \left.\left(\frac{\pi}\omega\right)^2\right|_{\Lambda = \Lambda^{4/3}a^{-1/3}}
     \right)
  \left(\left.{u}\right|_{\Lambda = \Lambda^{4/3}a^{-1/3}}
    - \left.\left(\frac{\pi}\omega\right)^2\right|_{\Lambda = \Lambda^{4/3}a^{-1/3}}
    \right) 
  = 3\Lambda^4.
\end{align*}
\begin{NB}
This is compatible with the transformation law, as
\begin{equation*}
   \left({u}(a^{4/3},\Lambda^{4/3})
     + \left(\frac{\pi}\omega\right)^2(a^{4/3},\Lambda^{4/3})
     \right)^2
  \left({u}(a^{4/3},\Lambda^{4/3})
    - 2 \left(\frac{\pi}\omega\right)^2(a^{4/3},\Lambda^{4/3})
    \right) 
  = \frac{27}4 \Lambda^8.
\end{equation*}
\end{NB}%
From the second equation and the definition of $\phi$, we get
\begin{equation}\label{eq:uom}
     \left.\left(\frac1{\Lambda^2} u\right)\right|_{\Lambda = \Lambda^{4/3}a^{-1/3}}
  = \frac12\left(3\phi^2 + \phi^{-2}\right),
\quad
  \left.\left(\frac1{\Lambda^2}
  \left(\frac{\pi}\omega\right)^2\right)\right|_{\Lambda = \Lambda^{4/3}a^{-1/3}}
  = \frac12 \left(3\phi^2 - \phi^{-2}\right).
\end{equation}
Substituting this to the first equation, we obtain
\begin{equation}\label{eq:aT}
    \frac{1}4\Lambda^2 a^{-2}
  = \phi^4 \left(-\frac12\phi^2 + \frac{1}{2\phi^2}\right)
  = \frac12 \phi^2\left(-\phi^4 + 1\right).
\end{equation}
Therefore
\begin{equation}\label{eq:da}
  \frac{da}a
  = - \frac{d\phi}{\phi(1 - \phi^4)}
  (1 - 3\phi^4).
\end{equation}
\begin{NB}
  \begin{equation*}
  = -\left(\frac{d\phi}\phi - \frac{2\phi^3 d\phi}{1 - \phi^4}\right)
  = - \frac{d\phi}{\phi(1 - \phi^4)}
  \left((1 - \phi^4) - 2\phi^4\right)
  \end{equation*}
\end{NB}%
By the above formulas, all the terms computed in
\subsecref{subsec:inst} can be expressed merely by $\phi$. Hence we
will treat $\phi$ as a variable instead of $a$.
\begin{NB}
We shall take the residue at $a=m=0$. But this is the same as the
residue at $a=m=\infty$ (up to sign), as the function is in
$\C[\nicefrac{a}\Lambda, \nicefrac{\Lambda}a]]$: we just take the
coefficient of $(\nicefrac{a}\Lambda)^{0}$.
\end{NB}%
We will write the differential $\mathcal B(\xi_1,\xi;a)da$, of which we
take the residue in Mochizuki's formula in terms of $\phi$. The
explicit formula will be given in the next section, but it is already
clear that it will involve several square roots and rational
expressions in $\phi$, when we expand it as a  series in $x$ and $z$. We
will see that square roots, in fact, do not appear, so we get a
rational differential in $\phi$ defined over $\proj^1$.

We will use the residue theorem to re-write Mochizuki's formula as sum
of residues at other poles in the next section. But it is instructive
to see the meaning of poles at this stage.

Since $\phi = \nicefrac{\Lambda}{\sqrt{2}a} + \cdots$, we have $\phi =
0$ at $a=\infty$.
By \eqref{eq:aT} there are other point $\phi^4 = 1$ giving
$a=\infty$. By \eqref{eq:da} they are indeed poles of the differential.
From \eqref{eq:uom} we have
\begin{equation*}
  \left.u^2 \right|_{\Lambda = \Lambda^{4/3}a^{-1/3}}= 4\Lambda^4.
\end{equation*}
As $u$ is coupled with the variable $x$ for the $\mu$-class of the
point in the formula in \thmref{thm:partition}, these correspond to
the KM-simple type condition in Definition~\ref{def:KMsimple}. In
\cite{Witten} it was noted that the Seiberg-Witten curve (for the pure
theory) has singularities at those points, and they give the
Seiberg-Witten invariant contribution to Donaldson
invariants. Therefore even before the actual calculation, it is
natural to expect that the residues at $\phi^4 = 1$ give what is
expected in Witten's conjecture \eqref{eq:Witten}.

There are other poles, which is already seen in \eqref{eq:uom}, at
$\phi^4 = 1/3$. At those points $\pi/\omega = \nicefrac{\sqrt{-1}}2
\nicefrac{\partial u}{\partial a}$ vanishes. It means that the
Seiberg-Witten curve completely degenerates as we have $e_1 = e_2 =
e_3$. It is called a {\it superconformal point\/} in the physics
literature, and is the origin of the superconformal simple type
condition \cite{MMP}. Therefore it is natural to expect that the
residue at $\phi^4 = 1/3$ is related to the superconformal simple type
condition. We will see that this is indeed so.

\begin{NB}
An earlier formulation:

Suppose $G$ is homogeneous of degree $d$. Then
\begin{equation*}
  G(a,\Lambda) = \sum_n G_n \left(\frac\Lambda{a}\right)^{3n} a^d.
\end{equation*}
Therefore
\begin{equation*}
  \begin{split}
  & G(a,\Lambda^{4/3}a^{-1/3})
  = \sum_n G_n\left(\frac\Lambda{a}\right)^{4n} a^d
  = \sum_n G_n\left(\frac\Lambda{a}\right)^{4n} (a^{4/3})^d a^{-d/3}
\\
  = \; & G(a^{4/3},\Lambda^{4/3}) a^{-d/3}
  = G(a^{4/3},\Lambda^{4/3}) (a^{4/3})^{-d/4}.
  \end{split}
\end{equation*}
If $G$ contains a perturbative part like $\log (a/\Lambda)$, we need a
separate treatment.

Since $u$ is of degree $2$, we have
\begin{equation*}
   u^{\mathrm{inst}}(a,\Lambda^{4/3}a^{-1/3}) =
   u^{\mathrm{inst}}(a^{4/3},\Lambda^{4/3}) a^{-2/3}.
\end{equation*}
Since $u^{\mathrm{pert}} = a^2$, we get
\begin{equation*}
   u(a,\Lambda^{4/3}a^{-1/3})
     = a^{-2/3} u(a^{4/3},\Lambda^{4/3}).
\end{equation*}
\begin{NB2}
  \begin{equation*}
   \mathrm{LHS} = a^2 + u^{\mathrm{inst}}(a,\Lambda^{4/3}a^{-1/3})
     = a^2 + u^{\mathrm{inst}}(a^{4/3},\Lambda^{4/3}) a^{-2/3}
     = a^2 + \left(u(a^{4/3},\Lambda^{4/3}) - a^{8/3}\right) a^{-2/3}.
  \end{equation*}
But this is obvious as we can apply the above transformation formula
also for $u = u^{\mathrm{pert}} + u^{\mathrm{inst}}$.
\end{NB2}
Similarly we have
\begin{equation*}
   \left(\frac{\pi}\omega\right)(a,\Lambda^{4/3}a^{-1/3})
     = a^{-2/3} \left(\frac{\pi}\omega\right)(a^{4/3},\Lambda^{4/3}).
\end{equation*}
\end{NB}

\begin{NB}
  We consider the equation above obtained from \eqref{eq:5} as a
  quadratic curve in $\proj^2$:
\begin{equation*}
  \left( u + \left(\frac\pi\omega\right)^2\right)
  \left( u - \left(\frac\pi\omega\right)^2\right)
  = \left(\sqrt{3}\Lambda^2\right)^2
\end{equation*}
Therefore
\begin{equation*}
  u = \zeta_0^2 + \zeta_1^2,\quad
  \left(\frac\pi\omega\right)^2 = \zeta_0^2 - \zeta_1^2,\quad
  \sqrt{3}\Lambda^2 = 2\zeta_0 \zeta_1,
\end{equation*}
and hence
\begin{equation*}
   \frac{u}{\Lambda^2}
   = \frac{\sqrt{3}}2\left(\zeta + \zeta^{-1}\right),\quad
   \frac1{\Lambda^2} \left(\frac\pi\omega\right)^2
   = \frac{\sqrt{3}}2\left(\zeta - \zeta^{-1}\right),\quad
   \frac1{\Lambda^2} \left({u} + \left(\frac\pi\omega\right)^2\right)
   = {\sqrt{3}}\zeta,
\end{equation*}
with $\zeta = \zeta_0/\zeta_1$. Therefore
$\zeta = \sqrt{3}\fT/\Lambda^2$.
\end{NB}

\section{Computation}

\begin{NB}
\subsection{Rearrangements}

We replace $\xi_1$, $\xi_2$ by the usual Seiberg-Witten class
$\tilde\xi_1 = 2\xi_1 - K_X$ together with an auxiliary term $\xi_2 -
\xi_1$. Once this will be done, the final answer should be independent
of $\xi_2-\xi_1$.

We should also note $(\xi_1,\xi_1 - K_X) = 0$, i.e.,
\begin{equation*}
  (\tilde\xi_1^2) = (K_X^2) = 2\chi(X) + 3\sigma(X).
\end{equation*}
This is called the {\it SW-simple type\/} condition. Then
\begin{equation*}
  \begin{split}
    & (\xi_2 - \xi_1, \xi - K_X)
      = (\xi_2 - \xi_1,\tilde\xi_1 - (\xi_1 - \xi_2))
    = ((\xi_2 - \xi_1)^2) + (\xi_2 - \xi_1, \tilde\xi_1),
\\
    & ((\xi - K_X)^2)
      = ((\tilde\xi_1 - (\xi_1 - \xi_2))^2)
    = (\tilde\xi_1^2) + 2(\tilde\xi_1, \xi_2 - \xi_1)
    + ((\xi_2 - \xi_1)^2).
  \end{split}
\end{equation*}

Therefore
\begin{equation*}
  \begin{split}
 & 
 \begin{aligned}[t]
 & \frac{1}{8}\frac{\partial^2 \Fin_0}{\partial a^2}
((\xi_2-\xi_1)^2)
+\frac{1}{4}\frac{\partial^2 \Fin_0}{\partial a \partial m}
(\xi_2-\xi_1,\xi-K_X)
\\
&\quad
+\frac{1}{8}\frac{\partial^2 \Fin_0}{\partial m^2}
((\xi-K_X)^2) + \chi(X)\Ain+\sigma(X)\Bin
 \end{aligned}
\\
 =\; &
 \begin{aligned}[t]
  & \frac{1}{8}\left(\frac{\partial^2 \Fin_0}{\partial a^2}
   + 2\frac{\partial^2 \Fin_0}{\partial a \partial m}
   + \frac{\partial^2 \Fin_0}{\partial m^2}\right)
 ((\xi_2 - \xi_1)^2)
 + \frac14
 \left(
   \frac{\partial^2 \Fin_0}{\partial a \partial m}
   + \frac{\partial^2 \Fin_0}{\partial m^2}
   \right)
 (\xi_2 - \xi_1, \tilde\xi_1)
\\
  & \quad
  + \left(\Ain + \frac14 \frac{\partial^2 \Fin_0}{\partial m^2}\right)
  \chi(X) 
  + \left(\Bin + \frac38 \frac{\partial^2 \Fin_0}{\partial m^2}\right)
  \sigma(X).
 \end{aligned}
  \end{split}
\end{equation*}

We also note that
\begin{equation*}
  \chi(\shfO_X)
  \begin{NB2}
    = \int_X \Todd_2(X)
    = \frac{\ve_1^2 + \ve_2^2 + 3\ve_1\ve_2}{12\ve_1\ve_2}
    = \frac14\left(\frac{\ve_1\ve_2}{\ve_1\ve_2}
    + \frac{\ve_1^2 + \ve_2^2}{3\ve_1\ve_2}\right)
  \end{NB2}%
  = \frac14 \left(\chi(X) + \sigma(X)\right).
\end{equation*}
Conversely
\begin{equation*}
    \chi(X) = 12\chi(\shfO_X) - (K_X)^2,\quad
    \sigma(X) = (K_X^2) - 8\chi(\shfO_X).
\end{equation*}
\begin{NB2}
Therefore
  \begin{equation*}
    \begin{split}
   &   \left(\Ain + \frac14 \frac{\partial^2 \Fin_0}{\partial m^2}\right)
  \chi(X) 
  + \left(\Bin + \frac38 \frac{\partial^2 \Fin_0}{\partial m^2}\right)
  \sigma(X)
\\
  = \; &
  \left(- \Ain + \Bin
    + \frac18 \frac{\partial^2 \Fin_0}{\partial m^2}\right)(K_X^2)
  + \left(12\Ain - 8\Bin\right)\chi(\shfO_X)
    \end{split}
  \end{equation*}
\end{NB2}

We also note
\begin{equation*}
  \begin{split}
  & \frac16 \frac{\partial^2 \Fin_0}{\partial a\partial\log\Lambda}
  (\xi_2 - \xi_1,\alpha) z
  + 
  \frac16 \frac{\partial^2 \Fin_0}{\partial m\partial\log\Lambda}
  (\xi - K_X,\alpha) z
\\
  =\; &
  \frac16 \left(
  \frac{\partial^2 \Fin_0}{\partial a\partial\log\Lambda}
  + \frac{\partial^2 \Fin_0}{\partial m\partial\log\Lambda}\right)
  (\xi_2 - \xi_1,\alpha) z
  + 
  \frac16 \frac{\partial^2 \Fin_0}{\partial m\partial\log\Lambda}
  (\tilde \xi,\alpha) z
  \end{split}
\end{equation*}

\subsection{Various terms}

In this subsection, we compute various terms in Mochizuki's formula.
This part should be moved to {\color{red}{\bf NB}} in the final version.

*****************************************************************

From (\ref{eq:uom}, \ref{eq:aT}) we have
\begin{equation*}
  \begin{split}
  & \frac{2\pi\sqrt{-1}}{\omega\Lambda}
  = \sqrt{-2}\left( \frac{3\fT}{\Lambda^2}
    - \frac{\Lambda^2}{\fT}\right)^{1/2}
  = \sqrt{2}\left( - 3\phi^2 + \phi^{-2}\right)^{1/2}
  = - \sqrt{2} \phi^{-1} \sqrt{1 - 3\phi^4},
\\
  & \Lambda^3 (a\fT)^{-1}
  = \sqrt{2}\left(-\phi^2 + \phi^{-2}\right)^{1/2}
  = \sqrt{2} \phi^{-1} \sqrt{1 - \phi^4}.
  \end{split}
\end{equation*}
We will show that the final answer is independent of the choice of the
branch of the square root later.

*****************************************************************

\begin{NB2}
  $x$
\end{NB2}

\begin{equation*}
  -s^2
  + \left.\frac13 
    \pd{\Fin_0}{\log\Lambda}\right|_{\Lambda = \Lambda^{4/3} a^{-1/3}} = -u
  = - \frac{\Lambda^2}2\left(3\phi^2 + \phi^{-2}\right)
\end{equation*}
from \eqref{eq:uom}.

*****************************************************************

\begin{NB2}
  $(\alpha^2)z^2$
\end{NB2}

\begin{equation*}
  \left.\frac1{18} \frac{\partial^2\Fin_0}{\partial(\log\Lambda)^2}
    \right|_{\Lambda = \Lambda^{4/3} a^{-1/3}}
      = -\frac12 \fT
      = -\frac12 \phi^2 \Lambda^2
\end{equation*}
by \eqref{eq:9}.
\begin{NB2}
  There is a sign inconsistency with Kota's answer. But this is
  probably harmless, as we will take the residue at both $\phi^2 = 1$
  and $-1$.
  At $\phi^2 = \pm 1$, we have
\(
  \mp \frac12 \Lambda^2 (\alpha^2)z^2
  \mp 2\Lambda^2 x,
\)
which {\it is\/} the expected answer.
\end{NB2}%

***************************************************************

\begin{NB2}
  $((\xi_2 - \xi_1)^2)$
\end{NB2}

We have
\begin{equation*}
  \left.\exp \Biggl[
%    - \frac{\pi\sqrt{-1}}2
    \frac14 \left(\frac{\partial^2 \Fin_0}{\partial m^2}
    + 2\frac{\partial^2 \Fin_0}{\partial a\partial m}
    + \frac{\partial^2 \Fin_0}{\partial a^2}\right) \Biggr]
  \right|_{\Lambda = \Lambda^{4/3} a^{-1/3}}
  = \frac{\Lambda^4}{2a^2} \fT^{-1}
\end{equation*}
from \eqref{eq:xi^2}.
We will take the $(\xi_2 - \xi_1)^2$ power of this function. There are
also factors containing $(\xi_2 - \xi_1)^2$. In total, we get
%\begin{multline}
\begin{equation}
  2^{-1/2}
  \left(\frac{2a}\Lambda\right)
  \left.\exp \Biggl[
%    - \frac{\pi\sqrt{-1}}2
    \frac18 \left(\frac{\partial^2 \Fin_0}{\partial m^2}
    + 2\frac{\partial^2 \Fin_0}{\partial a\partial m}
    + \frac{\partial^2 \Fin_0}{\partial a^2}\right)\Biggr]
  \right|_{\Lambda = \Lambda^{4/3} a^{-1/3}}
  = \frac{\Lambda}{\sqrt{\fT}} = \phi^{-1}.
\end{equation}
%\end{multline}
More precisely, this holds up to sign. But we compare the leading term
of both sides, and find that the above holds from our choice of $\sqrt{\fT}$.

*********************************

\begin{NB2}
  $(\xi_2 - \xi_1, \tilde\xi_1)$
\end{NB2}

\begin{equation*}
  \begin{split}
  & 2^{-1/2}
  \left.
  \exp\left[
    \frac14\left(\frac{\partial^2 \Fin_0}{\partial a\partial m}
  + \frac{\partial^2 \Fin_0}{\partial a^2}
   \right)
    \right]  \right|_{\Lambda = \Lambda^{4/3} a^{-1/3}}
  = \frac{\Lambda^2}{2\fT^{3/2}}
    \left(
       \frac{2\pi\sqrt{-1}}\omega + \frac{\Lambda^4}{a\fT}
  \right)
\\
  =\; &
  \frac1{\sqrt{2}} \phi^{-4}\left(
    - \sqrt{1 - 3\phi^4} + \sqrt{1 - \phi^4}
    \right)
  \end{split}
\end{equation*}
from \eqref{eq:13}.
\begin{NB2}
  This is equal to
  \begin{equation*}
    \left\{
    \frac1{\sqrt{2}\phi^{2}} \left(
    \sqrt{1 - \phi^4} - \sqrt{1 - 3\phi^4}
    \right)
    \right\}^{-1}.
  \end{equation*}
\end{NB2}%
At $\phi^2 = \pm 1$, this becomes $\mp \sqrt{-1}$.

*********************************

\begin{NB2}
  $(\xi_2 - \xi_1, \alpha)$
\end{NB2}

\begin{equation*}
  - s + \frac16 \left.\left(
      \frac{\partial^2 \Fin_0}{\partial a\partial\log\Lambda}
      + 
      \frac{\partial^2 \Fin_0}{\partial m\partial\log\Lambda}
    \right)\right|_{\Lambda = \Lambda^{4/3} a^{-1/3}}
  = -\frac12 \frac{\Lambda^4}{a\fT}
  = -\frac\Lambda{\sqrt{2}}\, \frac{\sqrt{1-\phi^4}}{\phi}
\end{equation*}
by \eqref{eq:d^2FdadL}.
This vanishes at $\phi^4 = 1$.

*********************************

\begin{NB2}
  $(\tilde\xi_1, \alpha)$
\end{NB2}

From \eqref{eq:d^2FdmdL} we have
\begin{equation*}
  \frac16\left.
      \frac{\partial^2 \Fin_0}{\partial m\partial\log\Lambda}
    \right|_{\Lambda = \Lambda^{4/3} a^{-1/3}}
    = -\frac12\left(
      \frac{\Lambda^4}{a\fT} + \frac{2\pi\sqrt{-1}}\omega
      \right)
    = -\frac{\Lambda}{\sqrt{2}\phi}
    \left(
      \sqrt{1 - \phi^4} - \sqrt{1 - 3\phi^4}
      \right).
\end{equation*}

*********************************

\begin{NB2}
  $(K_X^2)$
\end{NB2}

From \eqref{eq:BA} we have
\begin{equation*}
  \begin{split}
  & \left.
    \exp \left[\Bin - \Ain + \frac18\frac{\partial^2\Fin_0}{\partial m^2}\right]
  \right|_{\Lambda = \Lambda^{4/3} a^{-1/3}}
  = \sqrt{-1} \left(\frac{2\pi}\omega\right)
  \left(
    \frac{\pi\sqrt{-1}}\omega - \frac{\Lambda^4}{2a \fT}
  \right)^{-1}
\\
  =\; &
  - \sqrt{2} \Lambda \phi^{-1}\sqrt{1-3\phi^4}
  \times
  \left\{
  \frac\Lambda{\sqrt{2}\phi}
  \left(
    - \sqrt{1 - 3\phi^4} - \sqrt{1 - \phi^4}
    \right)\right\}^{-1}
\\
  =\; &
    2 \sqrt{1 - 3\phi^4} \left(
    \sqrt{1 - \phi^4} + \sqrt{1 - 3\phi^4}
    \right)^{-1}.
  \end{split}
\end{equation*}
\begin{NB2}
  The sign mistake is corrected. Dec.~21.
\end{NB2}

At $\phi^4 = 1$, this is equal to
\(
  2\sqrt{-2}/(\sqrt{-2}) = 2.
\)

*********************************

\begin{NB2}
$\chi(\shfO_X)$
\end{NB2}

From \eqref{eq:chiO} we have
\begin{equation*}
  \begin{split}
  & \left(
    \frac{2s}\Lambda
    \right)^3 2^{-2}
  \left.\exp\left[12\Ain - 8\Bin\right]\right|_{\Lambda = \Lambda^{4/3} a^{-1/3}}
\\
  = \; &
  \left(
    \frac{2s}\Lambda
    \right)^3 2^{-2}
    \times \sqrt{-1} \Lambda^{20/3} a^{-5/3} \fT^{-2} 2^{-1}
  \left(\frac{a}\Lambda\right)^{-4/3}
  \left(\frac{2\pi}\omega\right)^{-1}
\\
  =\; & \sqrt{-1} \Lambda^5
  \fT^{-2}  \left(\frac{2\pi}\omega\right)^{-1}
  = - \sqrt{-1} \phi^{-4}
  \times
  \frac{\sqrt{-1}}{\sqrt{2}}\frac{\phi}{\sqrt{1 - 3\phi^4}}
  = \frac1{\sqrt{2}\phi^{3}\sqrt{1 - 3\phi^4}}
  .
  \end{split}
\end{equation*}

\begin{NB2}
  Here are old checks:

We have
\begin{multline*}
%\begin{equation*}
  \left.\exp\left(\Ain + \frac14 \frac{\partial^2 \Fin_0}{\partial m^2}\right)
    \right|_{\Lambda = \Lambda^{4/3} a^{-1/3}}
\\
  =
  \left.\left(\frac1{\sqrt{-1}a}\frac{\pi}\omega\right)^{1/2}
  \times
  \left[
   \frac{\sqrt{-1}}{\Lambda^7}
  \left(\frac{2a}\Lambda\right)^{-1}
   \left(\frac{2\pi}\omega\right)^{-5}
  \T^{6}\right]^{-1/4}
    \right|_{\Lambda = \Lambda^{4/3} a^{-1/3}}
\\
  = \left(\frac1{\sqrt{-1}a}\frac{\pi}\omega\right)^{1/2}
  \times
  \left[
   \frac{\sqrt{-1}a^{7/3}}{\Lambda^{28/3}}
  \left(\frac{2a}\Lambda\right)^{-4/3}
   \left(\frac{2\pi}\omega\right)^{-5}
  \fT^{6}\right]^{-1/4}
\\
  = (\sqrt{-1})^{-3/4}   2^{13/12}
  \Lambda^{5/4}
  \left(\frac{\pi}\omega\right)^{7/4}
  \fT^{-3/2}
  \left(\frac{2a}\Lambda\right)^{- 3/4}
\end{multline*}
%\end{equation*}
from \eqref{eq:Ain} and \eqref{eq:d^2Fdm^2}.

Let us study the power of $a$ after the substitution $\Lambda =
\Lambda^{4/3}a^{-1/3}$:
\begin{equation*}
  -\frac12
  - \frac74 \times \frac13
  + \frac14 \times \frac43
  = -\frac34.
\end{equation*}
This term appears with $\chi(X)$ in Mochizuki's formula. It will
cancel with $3\chi(\shfO_X) = \nicefrac34(\chi(X)+\sigma(X))$~!

We have
\begin{multline*}
  \left.
  \exp\left(\Bin + \frac38 \frac{\partial^2 \Fin_0}{\partial m^2}\right)
  \right|_{\Lambda = \Lambda^{4/3} a^{-1/3}}
\\
  =
  \left(- {\sqrt{-1}} 
  \Lambda^{-44/3}a^{11/3}
  \left(\frac{2\pi}\omega\right)^7
  \left(\frac{2a}\Lambda\right)^{-20/3}
  \fT^2\right)^{1/8}
  \times
    \left[
   \frac{\sqrt{-1}a^{7/3}}{\Lambda^{28/3}}
  \left(\frac{2a}\Lambda\right)^{-4/3}
   \left(\frac{2\pi}\omega\right)^{-5}
  \fT^{6}\right]^{-3/8}
\\
  = (\sqrt{-1})^{-1/2} %2^{11/4}
  2^{5/12}
  \fT^{-2}\left(\frac{2\pi}\omega\right)^{11/4}
  \left(\frac{2a}\Lambda\right)^{-3/4}
  \Lambda^{5/4}.
\end{multline*}
from \eqref{eq:Bin} and \eqref{eq:d^2Fdm^2}.
\begin{NB2}
The contribution of the factor $\exp\left(\nicefrac38
  \nicefrac{\partial^2 \Fin_0}{\partial m^2}\right)$ to the power of
$a$ is $-3/8$ as above. From $\Bin$, we get
\begin{equation*}
  \frac{11}8\times \frac13 - \frac58\times\frac43
  = -\frac38.
\end{equation*}
So the sum is again
\begin{equation*}
  -\frac38 - \frac38 = -\frac34,
\end{equation*}
and cancel with $3\chi(\shfO_X) = \nicefrac34(\chi(X)+\sigma(X))$~!
\end{NB2}

Consider the term \eqref{eq:13}. After the substitution $\Lambda =
\Lambda^{4/3}a^{-1/3}$, we get
\begin{equation*}
  \left.  \exp\left[
    -\frac12\left(\frac{\partial^2 \Fin_0}{\partial m^2}
  + \frac{\partial^2 \Fin_0}{\partial a\partial m}
   \right)
  \right]\right|_{\Lambda = \Lambda^{4/3}a^{-1/3}}
%\\
  =
  \frac{2\fT^3}{\Lambda^4}
  \left(
      \frac{2\pi\sqrt{-1}}\omega + \Lambda^4 (a\fT)^{-1}
  \right)^{-2}.
\end{equation*}
\begin{NB2}
Original version:
  \begin{equation*}
      - \frac1{2\Lambda^4}
  \fT^2\,
% \left(
%     u + \left(\frac\pi\omega\right)^2
%   \right)^2
%   \frac{\left(\frac{\pi\sqrt{-1}}\omega - \frac32\Lambda^3
%     \left({u} + \left(\frac{\pi}\omega\right)^2\right)^{-1}\right)}
%    {\left(\frac{\pi\sqrt{-1}}\omega + \frac32\Lambda^3
%     \left({u} + \left(\frac{\pi}\omega\right)^2\right)^{-1}\right)},
  \frac{%\left(
      \frac{2\pi\sqrt{-1}}\omega - \Lambda^4 (a\fT)^{-1}
%   3 \left({u} + \left(\frac{\pi}\omega\right)^2\right)^{-1}
    %\right)
   }
   {%\left(
       \frac{2\pi\sqrt{-1}}\omega + \Lambda^4 (a\fT)^{-1}
%   3 \left({u} + \left(\frac{\pi}\omega\right)^2\right)^{-1}
   %\right)
   }.
  \end{equation*}
\end{NB2}
From \eqref{eq:cubic} we have
\begin{equation*}
  \begin{split}
  \Lambda^2 (a\fT)^{-1}
  &= \frac2{\sqrt{3}\Lambda^2} \left(u(a,\Lambda^{4/3}a^{-1/3})
      - 2 \left(\frac\pi\omega\right)^2(a,\Lambda^{4/3}a^{-1/3})\right)^{1/2}
\\
  &= \sqrt{2}\left(-\frac{\fT}{\Lambda^2} + \frac{\Lambda^2}\fT\right)^{1/2}.
  \end{split}
\end{equation*}
We have
\begin{equation*}
  \frac{2\pi\sqrt{-1}}{\omega\Lambda^2}
  = \sqrt{2}\left( \frac{3\fT}{\Lambda^2}
    - \frac{\Lambda^2}{\fT}\right)^{1/2}.
\end{equation*}

$\Lambda^4 (a\fT)^{-1}$ becomes $0$ at $\fT = \pm
\Lambda^2$. Therefore the above becomes
\begin{equation*}
  - \frac12.
\end{equation*}
Since this term appears as 
\(
  \exp\left[
    \frac14\left(\nicefrac{\partial^2 \Fin_0}{\partial m^2}
  + \nicefrac{\partial^2 \Fin_0}{\partial a\partial m}
   \right)
  \right](\xi_2 - \xi_1,\tilde\xi_1),
\)
we get
\begin{equation*}
  \sqrt{2}^{(\xi_2 - \xi_1,\tilde\xi_1)}.
\end{equation*}
This cancels with $2^{-(\xi_2 - \xi_1,\tilde\xi_1)/2}$ in Mochizuki's
formula !
\end{NB2}

\subsection{}

Substituting all terms into Mochizuki's formula to get (recall $y =
(2,\xi,n)$)
\begin{NB2}
  Here $\widetilde{\cal A}(\xi_1,y)$ means the expression before
  taking the residue at $s=0$. I need a correction later. I believe
  that it is better to change $\widetilde{\mathcal A}$, to the
  expression before taking the residue, and express Mochizuki's
  formula as $\Res_{s=0}\widetilde{\mathcal A}$.
\end{NB2}%

\begin{equation*}
  \begin{split}
    \mathcal B(\tilde\xi_1,\xi) \defeq
  & \sum_n \Lambda^{4n - (\xi^2) - 3\chi(\shfO_X)} \widetilde{\cal A}(\xi_1,y)
\\
  = \; &
  \begin{aligned}[t]
  & -(-1)^{\frac{((\xi_2-\xi_1)^2)+(K_X,\xi_2-\xi_1))}{2}+\chi({\cal O}_X)}
  \frac{1 - 3\phi^4}{1 - \phi^4} \frac{d\phi}\phi
\\  
  & \times
  \exp\Biggl[- \frac{\Lambda^2}2\left(3\phi^2 + \phi^{-2}\right) x
  - \frac12\phi^2\Lambda^2 (\alpha^2)z^2 \Biggr]
  \phi^{-((\xi_2 - \xi_1)^2)}
\\
  & \times
    \left(
      \frac1{\sqrt{2}} \phi^{-4}\left(
      -\sqrt{1 - 3\phi^4} + \sqrt{1 - \phi^4}
      \right)
    \right)^{(\xi_2 - \xi_1, \tilde\xi_1)}
\\
  & \times
    \exp\left(
            -\frac\Lambda{\sqrt{2}}\, \frac{\sqrt{1-\phi^4}}{\phi}    
            (\xi_2 - \xi_1, \alpha)z
      \right)
\\
  & \times
   \exp\left(-\frac{\Lambda}{\sqrt{2}\phi}
    \left(
      \sqrt{1 - \phi^4} - \sqrt{1 - 3\phi^4}
      \right)
      (\tilde\xi_1, \alpha)
      \right)
\\
  & \times
   \left(
     2 \sqrt{1 - 3\phi^4} \left(
    \sqrt{1 - \phi^4} + \sqrt{1 - 3\phi^4}
    \right)^{-1}
    \right)^{(K_X^2)}
   \left(
     \frac1{\sqrt{2}\phi^{3}\sqrt{1 - 3\phi^4}}
   \right)^{\chi(\shfO_X)}.
  \end{aligned}
  \end{split}
\end{equation*}
Using
\(
  \xi_2 - \xi_1 = \xi - K_X - \tilde\xi_1
\)
and
\(
   (\tilde\xi_1^2) = (K_X^2),
\)
we rewrite this as follows:
\end{NB}%
\subsection{Explicit expression of the differential}

Substituting all terms computed in the previous section into the
formula in \thmref{thm:partition}, we obtain
\begin{equation}\label{eq:N1}
  \mathcal B(\tilde\xi_1,\xi;a)da =
  \begin{aligned}[t]
  & -(-1)^{\frac{(\xi,\xi+K_X)-(K_X^2) - (K_X,\tilde\xi_1)}{2} + \chi_h(X)}
  \frac{1 - 3\phi^4}{1 - \phi^4} \frac{d\phi}\phi
\\  
  & \times
  \exp\Biggl[- \frac{\Lambda^2}2\left(3\phi^2 + \phi^{-2}\right) x
  - \frac12\phi^2\Lambda^2 (\alpha^2)z^2 \Biggr]
  \phi^{-((\xi - K_X)^2)-(K_X)^2 - 3\chi_h(X)}
\\
  & \times
    \left(
      \frac1{\sqrt{2}} \phi^{-2}\left(
      \sqrt{1 - \phi^4} - \sqrt{1 - 3\phi^4}
      \right)
    \right)^{(\xi - K_X, \tilde\xi_1)}
\\
  & \times
   \exp\left(\frac{\Lambda}{\sqrt{2}}
     \phi^{-1}
     \left(
     \sqrt{1 - 3\phi^4}
      (\tilde\xi_1, \alpha)z
    -  
    \sqrt{1-\phi^4}
            (\xi - K_X, \alpha)z
     \right)\right)
\\
  & \times
   \left(
     \sqrt{2} \sqrt{1 - 3\phi^4}
    \right)^{(K_X^2) - \chi_h(X)}.
  \end{aligned}
\end{equation}
\begin{NB}
Here we have used
\begin{equation*}
  ((\xi_2 - \xi_1)^2) + (K_X,\xi_2 - \xi_1)
  \equiv  (\xi, \xi + K_X) - (K_X^2) - (K_X,\tilde\xi_1)   \pmod 4,
\end{equation*}
which follows from $\tilde\xi_1 \equiv K_X \pmod 2$.
\begin{NB2}
  \begin{equation*}
    \begin{split}
    & ((\xi_2 - \xi_1)^2) + (K_X,\xi_2 - \xi_1)
    = ((\xi - K_X - \tilde\xi_1)^2) + (K_X, \xi - K_X - \tilde\xi_1)
\\
    =\; & ((\xi - K_X)^2) - 2(\xi-K_X,\tilde\xi_1) + (\tilde\xi_1^2)
    + (K_X, \xi - K_X) - (K_X,\tilde\xi_1)
\\
    =\; &
    (\xi, \xi - K_X) - 2(\xi-K_X,\tilde\xi_1) + (K_X^2) - (K_X,\tilde\xi_1)
\\
   \equiv\; &
   (\xi, \xi - K_X) - 2(\xi-K_X, K_X) + (K_X^2) - (K_X,\tilde\xi_1)
   \pmod 4
\\
   \equiv \; &
   (\xi, \xi + K_X) - (K_X^2) - (K_X,\tilde\xi_1)   \pmod 4
    \end{split}
  \end{equation*}
  This is Kota's (5.82).
  When we replace $\tilde\xi_1$ by
  $-\tilde\xi_1$ (keeping $\xi$), we have the difference
  $2(K_X,\tilde\xi_1) \equiv 2(K_X^2) \bmod 4$.

Note that
\(
   (\xi,\xi+K_X)
\)
is even by Wu's (or the adjunction) formula and
\(
  (-1)^{\nicefrac{(\xi, \xi + K_X)}2}
\)
is the usual term coming from the orientation convention. We also have
that
\(
   (K_X^2) - (K_X,\tilde\xi_1) = (K_X, K_X - \tilde\xi_1)
   = 2(K_X,K_X - \xi_1)
\)
is even.
\end{NB2}%
\end{NB}%

\begin{NB}
  There is inconsistency with Kota's formula (5.81). The sign
$(-1)^{(K_X^2)}$ is missing in (5.81). 

  My mistake is corrected. Dec.~21.
\end{NB}%

This is a simple substitution except that we need to take square
roots or $8^{\mathrm{th}}$ roots for some expressions. For example,
the term with $((\xi-K_X)^2)$ is
\begin{equation*}
  2^{-1/2}
  \left(\frac{2a}\Lambda\right)
  \left.\exp \Biggl[
%    - \frac{\pi\sqrt{-1}}2
    \frac18 \left(\frac{\partial^2 \Fin_0}{\partial m^2}
    + 2\frac{\partial^2 \Fin_0}{\partial a\partial m}
    + \frac{\partial^2 \Fin_0}{\partial a^2}\right)\Biggr]
  \right|_{\Lambda = \Lambda^{4/3} a^{-1/3}}.
%  = \frac{\Lambda}{\sqrt{\fT}} = \phi^{-1}.
\end{equation*}
From \eqref{eq:xi^2} the square of this is equal to
$\Lambda^2/\fT$. Since the leading term of the above is
$\sqrt{2}a/\Lambda$, we find that it is equal to
${\Lambda}/{\sqrt{\fT}} = \phi^{-1}$ from our choice of
$\sqrt{\fT}$. We use the same argument for other expressions involving
square roots.

When we expand $\widetilde{\mathcal A}(\xi_1,y;a)$ into a formal
power series in $z$, $x$ as $\sum_{k,l} {\mathcal A}_{k,l} z^k x^l$,
we will be interested in the case $k + 2l = 4n - (\xi^2) -
3\chi_h(X) = \dim M_H(y)$, otherwise the residue at $\phi=0$ vanishes by
the cohomology degree reason.
Note also that $\dim M_H(y)\equiv -(\xi^2) - 3\chi_h(X)\bmod 4$ is
independent of $n$.
Therefore we decompose $\mathcal B(\tilde\xi_1,\xi;a)$ as
\begin{equation*}
  \mathcal B(\tilde\xi_1,\xi;a)
  = \mathcal B^{(0)}(\tilde\xi_1,\xi;a)
  + \mathcal B^{(1)}(\tilde\xi_1,\xi;a)
  + \mathcal B^{(2)}(\tilde\xi_1,\xi;a)
  +  \mathcal B^{(3)}(\tilde\xi_1,\xi;a)
\end{equation*}
according to $(k+2l)\bmod 4$. If we write variables $(x,z)$, those are
given explicitly as
\begin{equation*}
  \mathcal B^{(p)}(\tilde\xi_1,\xi;a)(x,z)
  = \frac14
  \sum_{q=0}^3
   (\sqrt{-1})^{-qp}
     \mathcal B(\tilde\xi_1,\xi;a)((-1)^q x,(\sqrt{-1})^q z).
\end{equation*}
\begin{NB}
  For example,
  \begin{equation*}
    \mathcal B^{(3)}(\tilde\xi_1,\xi;a)
    =\frac14
      \left(
      \mathcal B(\tilde\xi_1,\xi;a)(x,z)
     + \sqrt{-1}     \mathcal B(\tilde\xi_1,\xi;a)(-x,\sqrt{-1}z)
     + (\sqrt{-1})^2  \mathcal B(\tilde\xi_1,\xi;a)(x,-z)
     + (\sqrt{-1})^3   \mathcal B(\tilde\xi_1,\xi;a)(-x,-\sqrt{-1}z)
    \right).
  \end{equation*}
\end{NB}%
We will be concerned with $\mathcal B^{(\dim
  M_H(y))}(\tilde\xi_1,\xi;a)$, where we understand
$\dim M_H(y)$ modulo $4$ as explained above.

We will be interested in the sum over all Seiberg-Witten classes
$\tilde\xi_1$. Therefore we can combine the contribution for
$\tilde\xi_1$ and $-\tilde\xi_1$, using
\(
   \SW(-\tilde\xi_1) = (-1)^{\chi_h(X)} \SW(\tilde\xi_1).
\)
\begin{NB}
and
\(
   (-1)^{(K_X,K_X-\tilde\xi_1)/2}
   = (-1)^{(K_X,K_X+\tilde\xi_1)/2} (-1)^{(K_X^2)}.
\)
\end{NB}%
Hence we will be interested in
\begin{equation}\label{eq:interested}
  \mathcal B^{(\dim M_H(y))}(\tilde\xi_1,\xi;a)
  + (-1)^{\chi_h(X)}\mathcal B^{(\dim M_H(y))}(-\tilde\xi_1,\xi;a).
\end{equation}
\begin{NB}
{\allowdisplaybreaks
\begin{equation}\label{eq:N2}
  \begin{split}
    & \mathcal B(\tilde\xi_1,\xi) +
    (-1)^{\chi(\shfO_X)}\mathcal B(-\tilde\xi_1,\xi)
\\
  =\;   & \sum_n \Lambda^{4n - (\xi^2) - 3\chi(\shfO_X)} 
  \left(\widetilde{\cal A}(\xi_1,y)
  +  (-1)^{\chi(\shfO_X) - (K_X^2)}\widetilde{\cal A}(K_X - \xi_1,y)\right)
\\
  =\; &
  \begin{aligned}[t]
  & -(-1)^{\frac{(\xi,\xi+K_X)-(K_X^2) - (K_X,\tilde\xi_1)}{2} + \chi(\shfO_X)}
  \frac{1 - 3\phi^4}{1 - \phi^4} \frac{d\phi}\phi
\\  
  & \times
  \exp\Biggl[- \frac{\Lambda^2}2\left(3\phi^2 + \phi^{-2}\right) x
  - \frac12\phi^2\Lambda^2 (\alpha^2)z^2 \Biggr]
  \phi^{-((\xi - K_X)^2)-(K_X)^2 - 3\chi(\shfO_X)}
\\
  & \times
   \left(
     \sqrt{2} \sqrt{1 - 3\phi^4}
    \right)^{(K_X^2) - \chi(\shfO_X)}
\\
  & \times\Biggl[
    \left(
      \frac1{\sqrt{2}} \phi^{-2}\left(
      \sqrt{1 - \phi^4} - \sqrt{1 - 3\phi^4}
      \right)
    \right)^{(\xi - K_X, \tilde\xi_1)}
\\
  & \qquad \times
   \exp\left(\frac{\Lambda}{\sqrt{2}}
     \phi^{-1}
     \left(
     \sqrt{1 - 3\phi^4}
      (\tilde\xi_1, \alpha)z
    -  
    \sqrt{1-\phi^4}
            (\xi - K_X, \alpha)z
     \right)\right)
\\
  & \quad + (-1)^{(K_X^2)-\chi(\shfO_X)}
    \left(
      \frac1{\sqrt{2}} \phi^{-2}\left(
      \sqrt{1 - \phi^4} {\color{red}+} \sqrt{1 - 3\phi^4}
      \right)
    \right)^{(\xi - K_X, \tilde\xi_1)}
\\
  & \qquad \times
   \exp\left(\frac{\Lambda}{\sqrt{2}}
     \phi^{-1}
     \left(
    {\color{red}-} \sqrt{1 - 3\phi^4}
      (\tilde\xi_1, \alpha)z
    -  
    \sqrt{1-\phi^4}
            (\xi - K_X, \alpha)z
     \right)\right)\Biggr].
  \end{aligned}
  \end{split}
\end{equation}
}
\end{NB}

\begin{NB}
An attempt to re-write \eqref{eq:N1}:

Let
\[
  \lambda \defeq
  \frac1{\sqrt{2}\phi^2}(\sqrt{1-\phi^4}-\sqrt{1-3\phi^4}).
\]
Then
\begin{gather*}
  \lambda^{-1}
  = \frac1{\sqrt{2}\phi^{2}}(\sqrt{1-\phi^4}+\sqrt{1-3\phi^4}),
\quad
  \frac{\lambda + \lambda^{-1}}2
  = \frac1{\sqrt{2}\phi^{2}}\sqrt{1-\phi^4},
\quad
  \frac{-\lambda + \lambda^{-1}}2
  = \frac1{\sqrt{2}\phi^{2}}\sqrt{1-3\phi^4},
\\
  \left(\frac{1 - \lambda^2}{1 + \lambda^2}\right)^2
  = \left(\frac{-\lambda+\lambda^{-1}}{\lambda+\lambda^{-1}}\right)^2
  = \frac{1-3\phi^4}{1-\phi^4},
\quad
  \frac2{2 + (\lambda+\lambda^{-1})^2} = \phi^4.
\end{gather*}

We expand \eqref{eq:N1} as
{\allowdisplaybreaks
\begin{equation}\label{eq:N3}
  \begin{split}
    & \mathcal B(\pm \tilde\xi_1,\xi)
\\
= \; &   
  \begin{aligned}[t]
  & - (-1)^{\frac{(\xi,\xi+K_X)-(K_X^2) \mp (K_X,\tilde\xi_1)}{2} + \chi(\shfO_X)}
  \frac{1 - 3\phi^4}{1 - \phi^4}\frac{d\phi}\phi 
  \phi^{-((\xi - K_X)^2)-(K_X)^2 - 3\chi(\shfO_X)}
\\
  & \times
    \left(
      \frac1{\sqrt{2}} \phi^{-2}\left(
      \sqrt{1 - \phi^4} \mp \sqrt{1 - 3\phi^4}
      \right)
    \right)^{(\xi - K_X, \tilde\xi_1)}
    \left(
     \sqrt{2} \sqrt{1 - 3\phi^4}
    \right)^{(K_X^2) - \chi(\shfO_X)}
\\  
  & \times
  \sum_{2l+k\equiv\dim M_H(y)\bmod 4}
  \left(- \frac{\Lambda^2}2\left(3\phi^2 + \phi^{-2}\right) x
        - \frac12\phi^2\Lambda^2 (\alpha^2)z^2\right)^{l}\frac1{l!}
\\
  & \qquad\times
   \left(\frac{\Lambda}{\sqrt{2}}
     \phi^{-1}
     \left(\pm
     \sqrt{1 - 3\phi^4}
      (\tilde\xi_1, \alpha)z
    -  
    \sqrt{1-\phi^4}
            (\xi - K_X, \alpha)z
     \right)\right)^{k}\frac1{k!}.
  \end{aligned}
\\
  =\; &
  \begin{aligned}[t]
  & - (-1)^{\frac{(\xi,\xi+K_X)-(K_X^2) - (K_X,\tilde\xi_1)}{2} + \chi(\shfO_X)}
  2^{(K_X^2)-\chi_h(X)}
  \frac{1 - 3\phi^4}{1 - \phi^4}\frac{d\phi}\phi 
  \phi^{-((\xi - K_X)^2) + (K_X^2) - 5\chi(\shfO_X) + k + 2l} 
\\
  & \times
  \sum_i \binom{(\xi - K_X, \tilde\xi_1)}{i}
    \left(
      \frac1{\sqrt{2}} \phi^{-2}
      \sqrt{1 - \phi^4}
    \right)^{(\xi - K_X, \tilde\xi_1)-i}
    \left(
      \frac{\mp 1}{\sqrt{2}} \phi^{-2}
      \sqrt{1 - 3\phi^4}
      \right)^{i + (K_X^2) - \chi(\shfO_X)}
\\
  & \times 
  \left.
   \begin{cases}
     1 & \text{for $+$}
\\
    (-1)^{\chi_h(X)} & \text{for $-$}     
   \end{cases}\right\}
\\
  & \times
  \sum_{2l+k\equiv\dim M_H(y)\bmod 4}
  \left(- \frac{\Lambda^2}2\left(3 + \phi^{-4}\right) x
        - \frac12\Lambda^2 (\alpha^2)z^2\right)^{l}\frac1{l!}
\\
  & \qquad\times
  \sum_j \binom{k}{j}
     (\tilde\xi_1, \alpha)^j
     (\xi - K_X, \alpha)^{k-j}
     \left(\pm\frac{\Lambda}{\sqrt{2}}
     \phi^{-2}
     \sqrt{1 - 3\phi^4}
   \right)^j
     \left(-
       \frac{\Lambda}{\sqrt{2}}\phi^{-2}
       \sqrt{1-\phi^4}\right)^{k-j}
     \frac{z^k}{k!}.
  \end{aligned}
  \end{split}
\end{equation}
We} will take the part $2l+k\equiv\dim M_H(y)\bmod 4$, as otherwise the
residue at $\phi=0$ vanishes by the degree reason.
If we consider $\mathcal
B(\tilde\xi_1,\xi)+(-1)^{\chi_h(X)}\mathcal B(-\tilde\xi_1,\xi)$,
we just take the terms with $i+(K_X^2)-\chi_h(X) + j \equiv 0
\bmod 2$. Therefore $\sqrt{1-3\phi^4}$ appears only in even powers. We
also have $(\xi-K_X,\tilde\xi_1)-i+k-j \equiv k + (K_X^2) -\chi_h(X)
+ (\xi-K_X,\tilde\xi_1) \equiv k + \dim M_H(y) \bmod 2$ by
\eqref{eq:sign}. Therefore $\sqrt{1-\phi^4}$ appears only in even
powers in $\mathcal B^{(p)}(\tilde\xi_1,\xi) +
(-1)^{\chi(\shfO_X)}\mathcal B^{(p)}(-\tilde\xi_1,\xi)$
if $p\equiv\dim M_H(y)\bmod 2$.
Moreover we have
\begin{equation*}
  \begin{split}
  & -((\xi-K_X)^2) + (K_X^2) - 5\chi(\shfO_X) + \dim M_H(y)
  \equiv -(\xi^2) + 2(\xi,K_X) - (\xi^2)
\\
 \equiv \; &
 2(\xi,\xi+K_X) \equiv 0 \pmod 4
  \end{split}
\end{equation*}
by Wu's formula (or the adjunction). Therefore $\phi$ appears only in
the form $\phi^4$.
\end{NB}

\begin{Proposition}\label{prop:parity}
  \textup{(1)} The combination
  $\mathcal B^{(p)}(\tilde\xi_1,\xi;a)da
  + (-1)^{\chi_h(X)}\mathcal B^{(p)}(-\tilde\xi_1,\xi;a)da$
  is unchanged under the the sign change of $\sqrt{1-3\phi^4}$.

  \textup{(2)} Suppose that $p \equiv \dim M_H(y)\bmod 2$. Then
  $\mathcal B^{(p)}(\tilde\xi_1,\xi;a)da$ is unchanged under the
  simultaneous sign change of $\sqrt{1-\phi^4}$ and $\sqrt{1-3\phi^4}$.

  In particular, if $p \equiv \dim M_H(y)\bmod 2$, $\mathcal
  B^{(p)}(\tilde\xi_1,\xi;a)da + (-1)^{\chi_h(X)}\mathcal
  B^{(p)}(-\tilde\xi_1,\xi;a)da$ contains even powers of $\sqrt{1 -
    3\phi^4}$ and $\sqrt{1-\phi^4}$, and hence is a rational $1$-form
  in $\phi$.

  \textup{(3)} The expression
$\mathcal B^{(\dim M_H(y))}(\tilde\xi_1,\xi;a)da
  + (-1)^{\chi_h(X)}\mathcal B^{(\dim M_H(y))}(-\tilde\xi_1,\xi;a)da$
  is a rational $1$-form in $\phi^4$.
\end{Proposition}

\begin{proof}
(1) Looking at \eqref{eq:N1}, we see that the replacement of 
$\sqrt{1-3\phi^4}$ by $-\sqrt{1-3\phi^4}$ has the same effect as the
replacement of $\tilde\xi_1$ by $-\tilde\xi_1$ together with the
multiplication by $(-1)^{\chi_h(X)}$, as
\begin{equation*}
  \frac1{\sqrt{2}\phi^{2}}(\sqrt{1-\phi^4}+\sqrt{1-3\phi^4})
  = \left\{
  \frac1{\sqrt{2}\phi^{2}}(\sqrt{1-\phi^4}-\sqrt{1-3\phi^4})
  \right\}^{-1}
\end{equation*}
and
\[
   (-1)^{(K_X,K_X-\tilde\xi_1)/2}
   = (-1)^{(K_X,K_X+\tilde\xi_1)/2} (-1)^{(K_X^2)}.
\]
Therefore the combination $\mathcal B^{(p)}(\tilde\xi_1,\xi;a)da +
(-1)^{\chi_h(X)}\mathcal B^{(p)}(-\tilde\xi_1,\xi;a)da$ is
unchanged.

(2) Looking at \eqref{eq:N1}, we find that the replacement
$\sqrt{1-\phi^4}$, $\sqrt{1-3\phi^4}$ by $-\sqrt{1-\phi^4}$,
$-\sqrt{1-3\phi^4}$ has the same effect as the replacement  of $(x,z)$
by $(x,-z)$ together with the multiplication by
$(-1)^{(\xi-K_X,\tilde\xi_1)+(K_X^2)-\chi_h(X)}$. From the
definition, the first replacement gives the multiplication by
$(-1)^p$. Now the assertion follows from the following:
\begin{equation}\label{eq:sign}
  \begin{split}
  & (\xi-K_X,\tilde\xi_1) + \dim M_H(y)
  \equiv (\xi-K_X,\tilde\xi_1) + (\xi^2) + \chi_h(X)
\\
  \equiv\; & (\xi-K_X,K_X) + (\xi,K_X) + \chi_h(X)
  \equiv (K_X^2) - \chi_h(X) \pmod 2.
  \end{split}
\end{equation}
\begin{NB}
  The first one is the dimension formula. The second one is
  $\tilde\xi_1\equiv K_X \bmod 2$ and Wu's formula `$(\alpha,\alpha)
  \equiv (w_2(X),\alpha)\bmod 2$ for any $\alpha$' (the adjunction for
  complex surfaces). The third one is obvious. This is also true for a
  $C^\infty$ 4-manifold, if we replace $K_X$ by $w_2(X)$.
\end{NB}%

For a later purpose we need a refinement:
\begin{equation}\label{eq:refined}
  \begin{split}
  & (\xi-K_X,\tilde\xi_1) + (K_X^2) + 3\chi_h(X)
  \equiv (\xi-K_X,\tilde\xi_1) + (K_X^2) - (\xi^2) - \dim M_H(y)
  \pmod 4
\\
  =\; & (\xi-K_X,\tilde\xi_1-K_X) + (\xi, K_X - \xi) - \dim M_H(y).
  \end{split}
\end{equation}

(3) Looking at \eqref{eq:N1} again, we find that the replacement of 
$\phi$ by $\sqrt{-1}\phi$ has the same effect as the replacement
$(x,z)$ by $(-x,-\sqrt{-1}z)$ together with the multiplication by
\begin{equation*}
  (-1)^{(\xi-K_X,\tilde\xi_1)}
  (\sqrt{-1})^{-((\xi - K_X)^2)-(K_X)^2 - 3\chi_h(X)}.
\end{equation*}
The first replacement gives the multiplication by $(\sqrt{-1})^{-\dim
  M_H(y)}$. Therefore the assertion follows from
\begin{equation*}
\begin{split}
& -\dim M_H(y) - 2(\xi-K_X,\widetilde{\xi}_1) -
\left\{((\xi-K_X)^2)+(K_X^2)+3\chi_h(X)\right\}\\
\equiv\; &
(\xi^2)+3\chi_h(X)-2(\xi-K_X,K_X)
-\left\{((\xi-K_X)^2)+(K_X^2)+3\chi_h(X)\right\}
\equiv 0 \pmod 4.
\end{split}
\end{equation*}
\end{proof}

\begin{NB}
  By the last argument, if $p\equiv \dim M_H(y) + 2\bmod 4$, then the
  replacement $\phi$ by $\sqrt{-1}\phi$ gives the multiplication by
  $-1$. (Therefore the expression is a rational $1$-form in $\phi^2$.)
  But it means that
  \begin{equation*}
    \Res_{\phi=a} \left[\mathcal B^{(p)}(\tilde\xi_1,\xi)
  + (-1)^{\chi_h(X)}\mathcal B^{(p)}(-\tilde\xi_1,\xi)
    \right]
    =
   - \Res_{\phi=\sqrt{-1}a} \left[\mathcal B^{(p)}(\tilde\xi_1,\xi)
  + (-1)^{\chi_h(X)}\mathcal B^{(p)}(-\tilde\xi_1,\xi)
    \right]
  \end{equation*}
  In particular, the residues at $\phi = \pm 1$, $\pm \sqrt{-1}$ or
  $\phi = \pm 3^{-1/4}$, $\pm \sqrt{-1}3^{-1/4}$ cancel out, and we
  cannot derive any useful information in this case.
\end{NB}

\begin{NB}
Here is an old version:
\begin{Claim}
  Suppose that $\tilde\xi_1$ is a Seiberg-Witten class. 
  Then $\sqrt{1 - 3\phi^4}$ and $\sqrt{1-\phi^4}$ appear as even
  powers in $\mathcal B(\xi_1,\xi)$.
  Thus $\mathcal B(\xi_1,\xi)$ is a rational $1$-form in $\phi$.
\end{Claim}

\begin{proof}

Let us consider the coefficient of $z^k$ in \eqref{eq:N2}.
\begin{NB2}
Take terms containing square roots:
  \begin{equation*}
    \begin{split}
  & \left(
     \sqrt{1 - 3\phi^4}
    \right)^{(K_X^2) - \chi(\shfO_X)}
\\
  & \times\Bigl[
    \left(
      \sqrt{1 - \phi^4} - \sqrt{1 - 3\phi^4}
    \right)^{(\xi - K_X, \tilde\xi_1)}
\\
  & \qquad \times
     \left(
     \sqrt{1 - 3\phi^4}
      (\tilde\xi_1, \alpha)
    -  
    \sqrt{1-\phi^4}
            (\xi - K_X, \alpha)
     \right)^k
\\
  & \quad + (-1)^{(K_X^2)-\chi(\shfO_X)}
    \left(
      \sqrt{1 - \phi^4} {\color{red}+} \sqrt{1 - 3\phi^4}
    \right)^{(\xi - K_X, \tilde\xi_1)}
\\
  & \qquad \times
     \left(
    {\color{red}-} \sqrt{1 - 3\phi^4}
      (\tilde\xi_1, \alpha)
    -  
    \sqrt{1-\phi^4}
            (\xi - K_X, \alpha)
     \right)^k\Bigr]
\\
=\; &
    \left(
     \sqrt{1 - 3\phi^4}
    \right)^{(K_X^2) - \chi(\shfO_X)} \sum (\tilde\xi_1,\alpha)^{k-j}
    (\xi-K_X,\alpha)^j
\\
   &\times\Biggl[
    \binom{(\xi-K_X,\tilde\xi_1)}{i}\binom{k}{j}
         (-1)^{i+j}\sqrt{1 - 3\phi^4}^{i+k-j}
         \sqrt{1 - \phi^4}^{(\xi-K_X,\tilde\xi_1)-i+j}
\\
   &\qquad +
   (-1)^{(K_X^2)-\chi(\shfO_X)}
    \binom{(\xi-K_X,\tilde\xi_1)}{i}\binom{k}{j}
         (-1)^{k}\sqrt{1 - 3\phi^4}^{i+k-j}
         \sqrt{1 - \phi^4}^{(\xi-K_X,\tilde\xi_1)-i+j}\Biggr]
\\
=\; &
  \sum (\tilde\xi_1,\alpha)^{k-j}
    (\xi-K_X,\alpha)^j
    \Biggl[
    \binom{(\xi-K_X,\tilde\xi_1)}{i}\binom{k}{j}
    \sqrt{1 - 3\phi^4}^{i+k-j+(K_X^2)-\chi(\shfO_X)}
         \sqrt{1 - \phi^4}^{(\xi-K_X,\tilde\xi_1)-i+j}
\\
   &\qquad\qquad\times
         \left(  (-1)^{i+j} - (-1)^{(K_X^2)-\chi(\shfO_X)+k}\right)\Biggr].
    \end{split}
  \end{equation*}
\end{NB2}%
If we change the sign of $\sqrt{1-3\phi^4}$, the coefficient is
multiplied by
\(
  (-1)^{2\{(K_X^2)-\chi(\shfO_X)\}} = 1.
\)
If we change the sign of $\sqrt{1-\phi^4}$, the coefficient is
multiplied by
\(
  (-1)^{k + (\xi-K_X, \tilde\xi_1) + (K_X^2) - \chi(\shfO_X)} = 1
\)
by \eqref{eq:sign}.
\begin{NB2}
  Alternative argument:

  The term in the above sum is nonzero only if
  $i+j\equiv k + (K_X^2) -\chi(\shfO_X)\bmod 2$. Therefore
  $\sqrt{1-3\phi^4}$ appears only in even powers. We also have
  $(\xi-K_X,\tilde\xi_1)-i+j \equiv k + (K_X^2) -\chi(\shfO_X)
  + (\xi-K_X,\tilde\xi_1) \equiv 0 \bmod 2$ by \eqref{eq:sign}. Therefore
  $\sqrt{1-\phi^4}$ appears only in even powers.
\end{NB2}%
Therefore the above expression contains only even powers
of $\sqrt{1-3\phi^4}$ and $\sqrt{1-\phi^4}$. Thus it is
a rational $1$-form in $\phi$.
\end{proof}
\end{NB}

\begin{NB}
The following is not necessary, but easier to start.
\begin{Claim}
  Suppose that $\tilde\xi_1 = 0$ is a Seiberg-Witten class. Then
  $\sqrt{1 - 3\phi^4}$ and $\sqrt{1-\phi^4}$ appear as even powers in
  $\mathcal B(\xi_1,\xi)$.
  Thus $\mathcal B(\xi_1,\xi)$ is a rational $1$-form in $\phi$.
\end{Claim}

\begin{proof}
Then
$(K_X^2) = (\tilde\xi_1^2) = 0$. Moreover, we have
$\chi(\shfO_X)$ is even, as
\(
   \SW(\tilde\xi_1) = (-1)^{\chi(\shfO_X)} \SW(-\tilde\xi_1).
\)
Therefore $(K_X^2) - \chi(\shfO_X)$ is even, and
$\sqrt{1-3\phi^4}$ appear in an even power.

Let us expand Mochizuki's formula (5.11) into a formal power series in
$z$, $x$, $\Lambda$ as $\sum_{k,l,n} M_{k,l,m} z^k x^l \Lambda^{\dim
  M_H(y)}$. Then we must have $k + 2l = \dim M_H(y)$ by the cohomology
degree reason. We have
\(
  (\xi^2) \equiv 0 \bmod 2
\)
from \eqref{eq:sign}, and hence $\dim M_H(y) \equiv -(\xi^2) -
3\chi(\shfO_X) \equiv (\xi^2)\equiv 0 \bmod 2$. Therefore $k$ must be
even. Looking at \eqref{eq:N1} we find that $\sqrt{1-\phi^4}$ appears
in an even power.
\end{proof}
\end{NB}

From the form of $\mathcal B(\tilde\xi_1,\xi;a)da$ in \eqref{eq:N1},
we find that the differential \eqref{eq:interested} has poles possibly
only at $\phi^4=0$, $\infty$, ${1}$ and ${\nicefrac13}$.
Mochizuki's formula is given by the residue at $\phi^4=0$.
The power of $\phi$, containing $-(\xi-K_X)^2$ is {\it very\/}
negative since $\xi$ is sufficiently ample when we apply Mochizuki's
formula to compute Donaldson invariants. Therefore it is not so easy
to compute the residue at $\phi^4=0$ directly.
Therefore we use the residue theorem
\begin{equation*}
  \left(
  \Res_{\phi^4=0} + \Res_{\phi^4=\infty} + \Res_{\phi^4=1} + \Res_{\phi^4=1/3}
  \right)
  \left[
    \text{the differential \eqref{eq:interested}}
%   \mathcal B^{(\dim M_H(y))}(\tilde\xi_1,\xi;a)da+(-1)^{\chi_h(X)}\mathcal
%   B^{(\dim M_H(y))}(-\tilde\xi_1,\xi;a)da
  \right]= 0,
\end{equation*}
to compute residues at $\infty$, $1$, $1/3$ instead.

\subsection{Residue at $\phi=\infty$}

We first treat the simplest (possible) pole $\phi=\infty$.
Recall that we expand $\mathcal B(\tilde\xi_1,\xi;a)da$ as formal
power series in $x$, $z$ and take coefficients of $x^k z^l$ with $k+2l
= 4n - (\xi^2) - 3\chi_h(X) = \dim M_H(y)$.
Let us denote this part as $\mathcal B^{[\dim
  M_H(y)]}(\tilde\xi_1,\xi;a)da$.
\begin{NB}
  I am not sure this is a good notation. Before we considered $\dim
  M_H(y)$ modulo $4$, but not this time.
\end{NB}%
The residue at $\phi^4=0$ is the same as that of $\widetilde{\mathcal
  A}(\xi_1,y;a)$ by the cohomological degree reason, but it is not
equal to $\widetilde{\mathcal A}(\xi_1,y;a)$ itself as we still take
the sum over all $n$.
Recall that when we use Mochizuki's formula in \thmref{thm:Mochizuki},
we expand $\mathcal B^{(\dim M_H(y))}(\tilde\xi_1,\xi;a)da$ in $x$,
$z$, compute the residue at $\phi^4=0$, and then take the sum over
$y$. Thus we actually need to compute the residue of $\mathcal
B^{[\dim M_H(y)]}(\tilde\xi_1,\xi;a)da$.

\begin{Proposition}
  $\mathcal B^{[\dim
    M_H(y)]}(\tilde\xi_1,\xi;a)da+(-1)^{\chi_h(X)}\mathcal B^{[\dim
    M_H(y)]}(-\tilde\xi_1,\xi;a)da$ is regular at $\phi^4=\infty$, if
  $\chi(y) > 0$.
\end{Proposition}

\begin{proof}
  Terms appearing in \eqref{eq:N1} have the following order of
  vanishing at $\phi = \infty$:
\begin{gather*}
  \operatornamewithlimits{Order}_{\phi=\infty}({\phi})=-1,\quad
  \operatornamewithlimits{Order}_{\phi=\infty}(\frac{d\phi}{\phi})=-1,\quad\
  \operatornamewithlimits{Order}_{\phi=\infty}(3\phi^2+\phi^{-2})^k = -2k,\\
  \operatornamewithlimits{Order}_{\phi=\infty}(\phi^{-1}\sqrt{1-\phi^4})^l = -l,\quad
  \operatornamewithlimits{Order}_{\phi=\infty}(\phi^{-1}\sqrt{1-3\phi^4})^l
    = -l,
\\
  \operatornamewithlimits{Order}_{\phi=\infty}(\sqrt{1-3\phi^4})
    = -2
\end{gather*}
Therefore  $\mathcal B^{[\dim
    M_H(y)]}(\tilde\xi_1,\xi)+(-1)^{\chi_h(X)}\mathcal B^{[\dim
    M_H(y)]}(-\tilde\xi_1,\xi)$ has zero of order at least
\begin{multline*}
-1+[((\xi-K_X)^2)+(K_X^2)+3\chi({\cal O}_X)]
-\dim M_H(y)-2(K_X^2)+2\chi({\cal O}_X)\\
=(\xi,\xi-2K_X)+5\chi({\cal O}_X)-\dim M_H(y)-1.
\end{multline*}
This is equal to $4\chi(y) - 1$. The assertion follows.
\end{proof}

\subsection{Residue at $\phi^4=1$}

Next we study the residue at $\phi^4=1$. We will show that it is
identified with Witten's formula.

\begin{NB}
  The following is far from complete.

The residue at $\phi = (\sqrt{-1})^k$ ($k=0,1,2,3$) is
\begin{equation*}
  \begin{split}
    &
  \begin{aligned}[t]
  & -(-1)^{\frac{(\xi,\xi+K_X)-(K_X^2) - (K_X,\tilde\xi_1)}{2}}
  \frac{2}{\prod_{l\neq k}((\sqrt{-1})^k-(\sqrt{-1})^l)} \frac1{(\sqrt{-1})^k}
\\  
  & \times
  \exp\Biggl[- 2{\Lambda^2}(-1)^k x
  - \frac12 (-1)^k\Lambda^2 (\alpha^2)z^2 \Biggr]
  (\sqrt{-1})^{-k\{((\xi - K_X)^2)+(K_X^2)+3\chi(\shfO_X)\}}
\\
  & \times\left(-\sqrt{2}\sqrt{-2}\right)^{(K_X^2)-\chi(\shfO_X)}
\\
  & \times\Biggl[
    \left(
      \frac1{\sqrt{2}} (-1)^{k}\left(
      - \sqrt{-2}
      \right)
    \right)^{(\xi - K_X, \tilde\xi_1)}
\\
  & \qquad \times
   \exp\left(\frac{\Lambda}{\sqrt{2}}
     (\sqrt{-1})^{-k}
     \sqrt{-2}
      (\tilde\xi_1, \alpha)z
    \right)
\\
  & \quad + (-1)^{(K_X^2)-\chi(\shfO_X)}
    \left(
      \frac1{\sqrt{2}} (-1)^{k} (\sqrt{- 2})
    \right)^{(\xi - K_X, \tilde\xi_1)}
\\
  & \qquad \times
   \exp\left(\frac{\Lambda}{\sqrt{2}}
     (\sqrt{-1})^{-k}
     \left(
    -  \sqrt{- 2}
      (\tilde\xi_1, \alpha)z
     \right)\right)\Biggr]
  \end{aligned}
\\
  =\; &
  \begin{aligned}[t]
  & - \frac12 (-1)^{\frac{(\xi,\xi+K_X)-(K_X^2) - (K_X,\tilde\xi_1)}{2}}
  \,2^{(K_X^2)-\chi(\shfO_X)}
\\  
  & \times
  \exp\Biggl[- 2{\Lambda^2}(-1)^k x
  - \frac12 (-1)^k\Lambda^2 (\alpha^2)z^2 \Biggr]
  (\sqrt{-1})^{-k\{((\xi - K_X)^2)+(K_X)^2+3\chi(\shfO_X)\}}
\\
  & \times
    (-1)^{k(\xi - K_X, \tilde\xi_1)}
    \left(\sqrt{- 1}\right)^{(\xi - K_X, \tilde\xi_1)+(K_X^2)-\chi(\shfO_X)}
\\
  & \times\Biggl[
      \left(
      - 1
      \right)^{(\xi - K_X, \tilde\xi_1)+(K_X^2)-\chi(\shfO_X)}
   \exp\left({\Lambda}
     (\sqrt{-1})^{-k+1}
      (\tilde\xi_1, \alpha)z
    \right)
\\
  & \qquad + 
   \exp\left({\Lambda}
     (\sqrt{-1})^{-k+3}
      (\tilde\xi_1, \alpha)z\right)\Biggr].
  \end{aligned}
  \end{split}
\end{equation*}
\begin{NB2}
We have
\begin{equation*}
  \prod_{l\neq k}((\sqrt{-1})^k-(\sqrt{-1})^l)
  = (\sqrt{-1})^{3k} (1 - i)(1 + 1)(1 + i)
  = 4(\sqrt{-1})^{3k}.
\end{equation*}
The mistake was correct on Dec.\ 15.
\end{NB2}%
By \eqref{eq:sign}
\(
   (-1)^{(\xi - K_X, \tilde\xi_1)+(K_X^2)-\chi(\shfO_X)}
\)
can be replace by
\(
   (-1)^{\dim M_H(y)}.
\)
Note also
\begin{equation*}
  \begin{split}
  & ((\xi - K_X)^2)+(K_X^2)+3\chi(\shfO_X)
  = (\xi^2)+3\chi(\shfO_X) - 2(\xi,K_X) + 2(K_X^2)
\\
  \equiv \; &
  - \dim M_H(y) - 2(\xi,K_X) + 2(K_X^2)
  \bmod 4
  \end{split}
\end{equation*}
Therefore
\begin{equation*}
  (\sqrt{-1})^{-k\{((\xi - K_X)^2)+(K_X)^2+3\chi(\shfO_X)\}}
  =  (\sqrt{-1})^{k\dim M_H(y)}(-1)^{k\{(\xi,K_X)+(K_X^2)\}}.
\end{equation*}
The second term $(-1)^{k\{(\xi,K_X)+(K_X^2)\}}$ cancels with
$(-1)^{k(\xi - K_X, \tilde\xi_1)}$ since $\tilde\xi_1\equiv K_X\bmod 2$.

We also have
\begin{equation*}
  (\xi-K_X,\tilde\xi_1) + (K_X^2) - \chi(\shfO_X)
  \equiv (\xi-K_X,\tilde\xi_1-K_X) + (\xi,K_X-\xi) - \dim M_H(y)
  \bmod 4
\end{equation*}
by \eqref{eq:refined}.

Therefore we get
\begin{equation*}
    \begin{aligned}[t]
  & - \frac12 (-1)^{\frac{(\xi,\xi+K_X)-(K_X^2) - (K_X,\tilde\xi_1)}{2}}
  \,2^{(K_X^2)-\chi(\shfO_X)}
\\  
  & \times
  \exp\Biggl[- 2{\Lambda^2}(-1)^k x
  - \frac12 (-1)^k\Lambda^2 (\alpha^2)z^2 \Biggr]
  (\sqrt{-1})^{k\dim M_H(y)}
\\
  & \times
%    (-1)^{k(\xi - K_X, \tilde\xi_1)}
    \left(\sqrt{- 1}\right)^{-\dim M_H(y)}
    (-1)^{(\xi-K_X,\frac{\tilde\xi_1-K_X}2) + \frac{(\xi, K_X - \xi)}2}
\\
  & \times\Biggl[
      \left(
      - 1
      \right)^{\dim M_H(y)}
   \exp\left({\Lambda}
     (\sqrt{-1})^{-k+1}
      (\tilde\xi_1, \alpha)z
    \right)
\\
  & \qquad + 
   \exp\left({\Lambda}
     (\sqrt{-1})^{-k+3}
      (\tilde\xi_1, \alpha)z\right)\Biggr].
  \end{aligned}
\end{equation*}
We have
\begin{equation*}
  \begin{split}
  & - \frac{(K_X,K_X + \tilde\xi_1)}2
  + (\xi-K_X, \frac{\tilde\xi_1 - K_X}2)
  + \frac{(\xi,K_X - \xi)}2
  = - \frac{(\xi,\xi-\tilde\xi_1)}2 - (K_X,\tilde\xi_1)
\\
  \equiv\; &
  \frac{(\xi,\xi-\tilde\xi_1)}2 - (K_X^2)\bmod 2.
  \end{split}
\end{equation*}
Therefore
\begin{equation*}
    \begin{aligned}[t]
  & - \frac12 (-1)^{\frac{(\xi,\xi+K_X)+(\xi,\xi-\tilde\xi_1)}{2}}
  \,2^{(K_X^2)-\chi(\shfO_X)} (-1)^{(K_X)^2}
\\  
  & \times
  \exp\Biggl[- 2{\Lambda^2}(-1)^k x
  - \frac12 (-1)^k\Lambda^2 (\alpha^2)z^2 \Biggr]
  (\sqrt{-1})^{k\dim M_H(y)}
\\
  & \times
%    (-1)^{k(\xi - K_X, \tilde\xi_1)}
    \left(\sqrt{- 1}\right)^{-\dim M_H(y)}
%    (-1)^{(\xi-K_X,\frac{\tilde\xi_1-K_X}2) + \frac{(\xi, K_X - \xi)}2}
\\
  & \times\Biggl[
      \left(
      - 1
      \right)^{\dim M_H(y)}
   \exp\left({\Lambda}
     (\sqrt{-1})^{-k+1}
      (\tilde\xi_1, \alpha)z
    \right)
\\
  & \qquad + 
   \exp\left({\Lambda}
     (\sqrt{-1})^{-k+3}
      (\tilde\xi_1, \alpha)z\right)\Biggr].
  \end{aligned}
\end{equation*}
The residue is the same for $k=0$, $2$ and $k=1$, $3$. For $k=0$ and
$1$, the value are related as
\begin{equation*}
  \Res_{\phi=\sqrt{-1}}
    = 
    (\sqrt{-1})^{-\dim M_H(y)}
  \left.\Res_{\phi=1}\right|_{\substack{x\mapsto -x\\z\mapsto \sqrt{-1}z}}.
\end{equation*}
\begin{NB2}
Since the sum is expressed as $\sum M_{k,l,n}z^k x^l \Lambda^{\dim
  M_H(y)}$ with $k+2l= \dim M_H(y)$, we must have
\[
  \left.(\text{ans.})\right|_{\substack{x\mapsto -x\\z\mapsto \sqrt{-1}z}}
  = (\sqrt{-1})^{\dim M_H(y)}(\text{ans}).
\]
Since
\begin{equation*}
  (\text{ans.}) = -2\left(\Res_{\phi=1} + \Res_{\phi=\sqrt{-1}}\right),
\end{equation*}
the above property is compatible.
\end{NB2}

*********************************************************

The residue of $\mathcal B(\tilde\xi_1,\xi) +
(-1)^{\chi(\shfO_X)} \mathcal B(-\tilde\xi_1,\xi)$ at $\phi = 1$
  is given by
\begin{equation*}
  \begin{split}
    &
  \begin{aligned}[t]
  & - \frac12 (-1)^{\frac{(\xi,\xi+K_X)-(K_X^2) - (K_X,\tilde\xi_1)}{2} + \chi_h(X)}
% \\  
%   & \times
  \exp\Biggl[- 2{\Lambda^2} x
  - \frac12 \Lambda^2 (\alpha^2)z^2 \Biggr]
% \\
%   & \times
  \left(\sqrt{2}\sqrt{-2}\right)^{(K_X^2)-\chi(\shfO_X)}
\\
  & \times\Biggl[
    \left(
      \frac1{\sqrt{2}} \left(
      - \sqrt{-2}
      \right)
    \right)^{(\xi - K_X, \tilde\xi_1)}
% \\
%   & \qquad \times
   \exp\left(\frac{\Lambda}{\sqrt{2}}
     \sqrt{-2}
      (\tilde\xi_1, \alpha)z
    \right)
\\
  & \qquad + (-1)^{(K_X^2)-\chi(\shfO_X)}
    \left(
      \frac1{\sqrt{2}} (\sqrt{- 2})
    \right)^{(\xi - K_X, \tilde\xi_1)}
% \\
%   & \qquad \times
   \exp\left(\frac{\Lambda}{\sqrt{2}}
     \left(
    -  \sqrt{- 2}
      (\tilde\xi_1, \alpha)z
     \right)\right)\Biggr]
  \end{aligned}
\\
  =\; &
  \begin{aligned}[t]
  & - \frac12 (-1)^{\frac{(\xi,\xi+K_X)-(K_X^2) - (K_X,\tilde\xi_1)}{2}}
  \,2^{(K_X^2)-\chi(\shfO_X)}
% \\  
%   & \times
  \exp\Biggl[- 2{\Lambda^2} x
  - \frac12 \Lambda^2 (\alpha^2)z^2 \Biggr]
\\
  & \times
    \left(\sqrt{- 1}\right)^{(\xi - K_X, \tilde\xi_1)+(K_X^2)-\chi(\shfO_X)}
    (-1)^{(K_X^2)}
\\
  & \times\Biggl[
      \left(
      - 1
      \right)^{(\xi - K_X, \tilde\xi_1)+(K_X^2)-\chi(\shfO_X)}
   \exp\left({\Lambda}
     \sqrt{-1}
      (\tilde\xi_1, \alpha)z
    \right)
% \\
%   & \qquad 
  + 
   \exp\left(-{\Lambda}\sqrt{-1}
      (\tilde\xi_1, \alpha)z\right)\Biggr].
  \end{aligned}
  \end{split}
\end{equation*}
By \eqref{eq:sign}
\(
   (-1)^{(\xi - K_X, \tilde\xi_1)+(K_X^2)-\chi(\shfO_X)}
\)
can be replace by
\(
   (-1)^{\dim M_H(y)}.
\)
Note also
\begin{equation*}
  (\xi-K_X,\tilde\xi_1) + (K_X^2) - \chi(\shfO_X)
  \equiv (\xi-K_X,\tilde\xi_1-K_X) + (\xi,K_X-\xi) - \dim M_H(y)
  \bmod 4
\end{equation*}
by \eqref{eq:refined}. Then we have
\begin{equation*}
  \begin{split}
  & - \frac{(K_X,K_X + \tilde\xi_1)}2
  + (\xi-K_X, \frac{\tilde\xi_1 - K_X}2)
  + \frac{(\xi,K_X - \xi)}2
  = - \frac{(\xi,\xi-\tilde\xi_1)}2 - (K_X,\tilde\xi_1)
\\
  \equiv\; &
  \frac{(\xi,\xi-\tilde\xi_1)}2 - (K_X^2)\pmod 2.
  \end{split}
\end{equation*}
Hence we have
\begin{equation}
  \begin{split}
  & \Res_{\phi=1}\left[\mathcal B(\tilde\xi_1,\xi) +
(-1)^{\chi(\shfO_X)} \mathcal B(-\tilde\xi_1,\xi)\right]
\\
 =\; &
     \begin{aligned}[t]
  & - \frac12 (-1)^{\frac{(\xi,\xi+K_X)+(\xi,\xi-\tilde\xi_1)}{2}}
  \,2^{(K_X^2)-\chi(\shfO_X)} 
% \\  
%   & \times
  \exp\Biggl[- 2{\Lambda^2} x
  - \frac12 \Lambda^2 (\alpha^2)z^2 \Biggr]
\\
   & \times
%    (-1)^{k(\xi - K_X, \tilde\xi_1)}
    \left(\sqrt{- 1}\right)^{-\dim M_H(y)}
%    (-1)^{(\xi-K_X,\frac{\tilde\xi_1-K_X}2) + \frac{(\xi, K_X - \xi)}2}
% \\
%   & \times
  \Biggl[
      \left(
      - 1
      \right)^{\dim M_H(y)}
   \exp\left({\Lambda}
     \sqrt{-1}
      (\tilde\xi_1, \alpha)z
    \right)
% \\
%   & \qquad 
  + 
   \exp\left(-{\Lambda}
     \sqrt{-1}
      (\tilde\xi_1, \alpha)z\right)\Biggr].
  \end{aligned}
  \end{split}
\end{equation}

*********************************************************
\end{NB}%

By \eqref{eq:N1}, the residue of $\mathcal B(\tilde\xi_1,\xi;a)da$ at
$\phi = 1$ is given by
\begin{equation*}
  \begin{split}
    &
  \begin{aligned}[t]
  & - \frac12 (-1)^{\frac{(\xi,\xi+K_X)-(K_X^2) - (K_X,\tilde\xi_1)}{2} + \chi_h(X)}
% \\  
%   & \times
  e^{- 2{\Lambda^2} x
  - \frac12 \Lambda^2 (\alpha^2)z^2}
% \\
%   & \times
  \left(\sqrt{2}\sqrt{-2}\right)^{(K_X^2)-\chi_h(X)}
\\
  & \times%\Biggl[
    \left(
      \frac1{\sqrt{2}} \left(
      - \sqrt{-2}
      \right)
    \right)^{(\xi - K_X, \tilde\xi_1)}
% \\
%   & \qquad \times
   \exp\left(\frac{\Lambda}{\sqrt{2}}
     \sqrt{-2}
      (\tilde\xi_1, \alpha)z
    \right)
% \\
%   & \qquad + (-1)^{(K_X^2)-\chi(\shfO_X)}
%     \left(
%       \frac1{\sqrt{2}} (\sqrt{- 2})
%     \right)^{(\xi - K_X, \tilde\xi_1)}
% % \\
% %   & \qquad \times
%    \exp\left(\frac{\Lambda}{\sqrt{2}}
%      \left(
%     -  \sqrt{- 2}
%       (\tilde\xi_1, \alpha)z
%      \right)\right)\Biggr]
  \end{aligned}
\\
  =\; &
  \begin{aligned}[t]
  & - (-1)^{\frac{(\xi,\xi+K_X)-(K_X^2) - (K_X,\tilde\xi_1)}{2}+(K_X^2)}
  \,2^{(K_X^2)-\chi_h(X)-1}
% \\  
%   & \times
  \exp\Biggl[- 2{\Lambda^2} x
  - \frac12 \Lambda^2 (\alpha^2)z^2 \Biggr]
\\
  & \times
    \left(\sqrt{- 1}\right)^{-\{(\xi - K_X, \tilde\xi_1)+(K_X^2)-\chi_h(X)\}}
% \\
%   & \times%\Biggl[
%       \left(
%       - 1
%       \right)^{(\xi - K_X, \tilde\xi_1)+(K_X^2)-\chi_h(X)}
   \exp\left({\Lambda}
     \sqrt{-1}
      (\tilde\xi_1, \alpha)z
    \right)
% \\
%   & \qquad 
%   + 
%    \exp\left(-{\Lambda}\sqrt{-1}
%       (\tilde\xi_1, \alpha)z\right)\Biggr]
    .
  \end{aligned}
  \end{split}
\end{equation*}
By \eqref{eq:refined}
\begin{equation*}
  (\xi-K_X,\tilde\xi_1) + (K_X^2) - \chi_h(X)
  \equiv (\xi-K_X,\tilde\xi_1-K_X) + (\xi,K_X-\xi) - \dim M_H(y)
  \bmod 4.
\end{equation*}
We combine the first two terms, which are even, with the factor coming
from $((K_X^2)+(K_X,\tilde\xi_1))/2$:
\begin{equation*}
  \begin{split}
  & - \frac{(K_X,K_X + \tilde\xi_1)}2
  + (\xi-K_X, \frac{\tilde\xi_1 - K_X}2)
  + \frac{(\xi,K_X - \xi)}2
  = - \frac{(\xi,\xi-\tilde\xi_1)}2 - (K_X,\tilde\xi_1)
\\
  \equiv\; &
  \frac{(\xi,\xi-\tilde\xi_1)}2 - (K_X^2)\pmod 2.
  \end{split}
\end{equation*}
\begin{NB}
  \begin{equation*}
    \begin{split}
    \mathrm{LHS} &=
    -\frac{(K_X,K_X+\tilde\xi_1)}2
    -(K_X,\frac{\tilde\xi_1-K_X}2)
    +(\xi,\frac{\tilde\xi_1-K_X}2 + \frac{K_X-\xi}2)
    = -(K_X,\tilde\xi_1) + \frac{(\xi,\tilde\xi_1-\xi)}2.
    \end{split}
  \end{equation*}
\end{NB}%
Hence we get
\begin{equation*}
  \begin{split}
%  & 
  \Res_{\phi=1}\mathcal B(\tilde\xi_1,\xi;a)da
% \\
  =
%  \; &
     \begin{aligned}[t]
  & - (-1)^{\frac{(\xi,\xi+K_X)+(\xi,\xi-\tilde\xi_1)}{2}}
  \,2^{(K_X^2)-\chi_h(X)-1} 
% \\  
%   & \times
  \exp\Biggl[- 2{\Lambda^2} x
  - \frac12 \Lambda^2 (\alpha^2)z^2 \Biggr]
\\
   & \times
%    (-1)^{k(\xi - K_X, \tilde\xi_1)}
    \left(\sqrt{- 1}\right)^{\dim M_H(y)}
%    (-1)^{(\xi-K_X,\frac{\tilde\xi_1-K_X}2) + \frac{(\xi, K_X - \xi)}2}
% \\
%   & \times
%  \Biggl[
%       \left(
%       - 1
%       \right)^{\dim M_H(y)}
   \exp\left({\Lambda}
     \sqrt{-1}
      (\tilde\xi_1, \alpha)z
    \right)
% \\
%   & \qquad 
%   + 
%    \exp\left(-{\Lambda}
%      \sqrt{-1}
%       (\tilde\xi_1, \alpha)z\right)\Biggr]
  \end{aligned}
  \end{split}
\end{equation*}
and
\begin{equation*}
  \begin{split}
   & \frac12
   \Res_{\phi=1}
   \left[\mathcal B^{(\dim M_H(y))}(\tilde\xi_1,\xi;a)da +
     (-1)^{\chi_h(X)} \mathcal B^{(\dim M_H(y))}(-\tilde\xi_1,\xi;a)da\right]
\\
   =\; &
     \begin{aligned}[t]
  & - (-1)^{\frac{(\xi,\xi+K_X)+(\xi,\xi-\tilde\xi_1)}{2}}
  \,2^{(K_X^2)-\chi_h(X)-3} 
\\  
  & \times
  \Biggl[
  e^{- 2{\Lambda^2} x
  - \frac12 \Lambda^2 (\alpha^2)z^2}
   \left\{
    \left(\sqrt{- 1}\right)^{\dim M_H(y)}
    e^{{\Lambda}
     \sqrt{-1}
      (\tilde\xi_1, \alpha)z}
    +
    \left(\sqrt{- 1}\right)^{-\dim M_H(y)}
    e^{-{\Lambda}
     \sqrt{-1}
      (\tilde\xi_1, \alpha)z}
  \right\}
\\  
  & \qquad +
  e^{2{\Lambda^2} x
  + \frac12 \Lambda^2 (\alpha^2)z^2}
   \left\{
    e^{-{\Lambda}
      (\tilde\xi_1, \alpha)z}
    +
    \left({- 1}\right)^{-\dim M_H(y)}
    e^{{\Lambda}
      (\tilde\xi_1, \alpha)z}
  \right\}\Biggr]
  ,
  \end{aligned}
  \end{split}
\end{equation*}
where we have used
\(
   (\xi,\tilde\xi_1) + \chi_h(X) \equiv \dim M_H(y) \pmod 2
\)
(cf.\ \eqref{eq:sign}).
The residues at $\phi=\sqrt{-1}$, $-1$, $-\sqrt{-1}$ are the same as
above by \propref{prop:parity}(3). Thus we multiply the above by $4$
for the contribution from $\phi^4=1$.

This contribution satisfies the KM-simple type condition, i.e., it is
killed by $(\partial/\partial x)^2 - 4\Lambda^4$. If we consider the
contribution to the Donaldson series ${\mathscr D}^\xi$, we get
\begin{equation*}
  - 
  (-1)^{\frac{(\xi,\xi+K_X)}{2}}
  2^{(K_X^2)-\chi_h(X)+1} 
  e^{2{\Lambda^2} x
  + \frac12 \Lambda^2 (\alpha^2)z^2}
  \sum_{\tilde\xi_1} \SW(\tilde\xi_1)
     \begin{aligned}[t]
  & (-1)^{\frac{(\xi,\xi-\tilde\xi_1)}{2}}
%   \left\{
    e^{-{\Lambda}
      (\tilde\xi_1, \alpha)z}
%    \right\}
    .
  \end{aligned}
\end{equation*}
Replacing $\tilde\xi$ by $-\tilde\xi$, removing the sign factor
$(-1)^{(\xi,\xi+K_X)/2}$ as in \subsecref{subsec:complex} and
multiplying with the $2$ from Mochizuki's convention, we get the right
hand side of \eqref{eq:Witten} with the opposite sign. Therefore, if
the residue at $\phi^4 = 1/3$ vanishes, we obtain \eqref{eq:Witten}.

\subsection{Residue at $\phi^4 = 1/3$}

\begin{Proposition}
  Suppose that $X$ is of superconformal simple type. Then
\[
  \sum_{\tilde\xi_1} \SW(\tilde\xi_1)
  \mathcal B^{(\dim M_H(y))}(\tilde\xi_1,\xi;a)da
\]
is regular at $\phi^4 = 1/3$.
\end{Proposition}

\begin{proof}
  Let
\begin{equation*}
   f(\lambda) \defeq
   \begin{aligned}[t]
   & \sum_{\tilde\xi_1}
   (-1)^{\frac{(K_X,K_X+\tilde\xi_1)}2} \SW(\tilde\xi_1)
     \lambda^{(\xi-K_X,\tilde\xi_1)}
     \left\{
      (-\lambda+\lambda^{-1})(\tilde\xi_1,\alpha)
      - (\lambda+\lambda^{-1})(\xi-K_X,\alpha)
     \right\}^k,
   \end{aligned}
\end{equation*}
where we assume $k$ has the same parity as $\dim M_H(y)$. By
\eqref{eq:sign}, we have $f(\lambda) = (-1)^{\chi_h(X)-(K_X^2)}f(-\lambda)$.

By the superconformal simple type condition, we have
\begin{equation*}
  f^{(n)}(1) = 0\quad\text{for $n=0,\dots, \chi_h(X)- (K_X^2) - 3$}.
\end{equation*}
\begin{NB}
For example, we have
\begin{equation*}
  \begin{split}
  f(1) &= \sum_{\tilde\xi_1} (-1)^{\frac{(K_X,K_X+\tilde\xi_1)}2} \SW(\tilde\xi_1)
  (-2(\xi-K_X,\alpha))^k,
\\
  f'(1) &= \sum_{\tilde\xi_1} (-1)^{\frac{(K_X,K_X+\tilde\xi_1)}2} \SW(\tilde\xi_1)
  \left[
  {(\xi-K_X,\tilde\xi_1)}(-2(\xi-K_X,\alpha))^k
  - 2 (-2(\xi-K_X,\alpha))^{k-1}(\tilde\xi_1,\alpha) \right].
  \end{split}
\end{equation*}
\end{NB}%
Therefore $f(\lambda) \in (\lambda -
1)^{\chi_h(X)-(K_X^2)-2}\C[\lambda^\pm]$. Since $f(-\lambda)$ is equal to
$f(\lambda)$ up to sign, we also have
$f(\lambda)\in(\lambda+1)^{\chi_h(X)-(K_X^2)-2}\C[\lambda^\pm]$. Therefore
\begin{equation*}
  f(\lambda)\in (\lambda - \lambda^{-1})^{\chi_h(X)-(K_X^2)-2}\C[\lambda^\pm].
\end{equation*}
From this we have the assertion by substituting 
\(
  \nicefrac1{\sqrt{2}}\phi^{-2}(\sqrt{1-\phi^4}-\sqrt{1-3\phi^4})
\)
to $\lambda$ .
\end{proof}

Next we study the converse direction:
\begin{Proposition}\label{prop:conv}
Suppose that
\[
  \Res_{\phi^4=1/3}\left(
  \sum_{\tilde\xi_1} \SW(\tilde\xi_1)
  \mathcal B^{(\dim M_H(y))}(\tilde\xi_1,\xi;a)da\right)
\]
depends only on $(\xi\bmod 2)$ up to sign. Then $X$ is of
superconformal simple type.
\end{Proposition}

Since the residues at the other poles depend only on $(\xi\bmod 2)$ up to
sign, the assumption is satisfied. Therefore $X$ is of superconformal
simple type. Then the residue at $\phi^4=1/3$ vanishes by the
previous proposition, and the sum of the residues at $\phi^4=0$ and
$\phi^4=1$ is zero. This  proves Witten's conjecture \eqref{eq:Witten}.

Before starting the proof of \propref{prop:conv}, we give some
preparation.

We fix $\xi^\circ$ and consider $\xi = K_X + t(\xi^\circ-K_X)$ with
$t\in 2\Z_{\ge 0} + 1$ as a function in $t$. Replacing $\tilde\xi_1$
by $-\tilde\xi_1$ if necessary, we may assume
$(\xi^\circ-K_X,\tilde\xi_1)\ge 0$.
We expand \eqref{eq:N1} by using the binomial theorem:
\begin{equation*}
  \begin{split}
  & \frac12
  \left(
  \mathcal B^{(\dim M_H(y))}(\tilde\xi_1,\xi;a)da 
  + (-1)^{\chi_h(X)}\mathcal B^{(\dim M_H(y))}(-\tilde\xi_1,\xi;a)da
  \right)
\\
=\; &
  \begin{aligned}[t]
  & - (-1)^{\frac{(\xi,\xi+K_X)-(K_X^2) - (K_X,\tilde\xi_1)}{2}}
  2^{(K_X^2)-\chi_h(X) + 1}
  \frac{\phi^4}{1 - \phi^4}\frac{d\phi}\phi 
  \sum_{i,j,k,l} 
  \phi^{-((\xi - K_X)^2) + (K_X^2) - 5\chi_h(X) + k + 2l} 
\\
  & \times
  (-1)^{i+k-j+(K_X^2)-\chi_h(X)} \Lambda^{k}
  \binom{(\xi - K_X, \tilde\xi_1)}{i}
  \binom{k}{j}
     (\tilde\xi_1, \alpha)^j
     (\xi - K_X, \alpha)^{k-j}
\\
  & \times
    \left(
      \frac{1 - \phi^4}{{2} \phi^{4}}
    \right)^{((\xi - K_X, \tilde\xi_1)-i+k-j)/2}
    \left(
      \frac{1 - 3\phi^4}{2\phi^4}
    \right)^{(i + j + (K_X^2) - \chi_h(X) + 2)/2}
\\
  & \times
  \left(- \frac{\Lambda^2}2\left(3 + \phi^{-4}\right) x
        - \frac12\Lambda^2 (\alpha^2)z^2\right)^{l}\frac1{l!}
     \frac{z^k}{k!},
  \end{aligned}
  \end{split}
\end{equation*}
where the summation runs over
\begin{equation*}
  2l+k\equiv\dim M_H(y)\bmod 4,\qquad
  i + j + (K_X^2) - \chi_h(X) \equiv 0 \bmod 2.
\end{equation*}
Moreover, since we are interested in the residue at $\phi^4 = 1/3$, we
only need to consider terms with
\begin{equation}\label{eq:bound}
  i+j + (K_X^2) - \chi_h(X) + 2 \le -2.
\end{equation}
\begin{NB}
  That is
\(
   i+j \le \chi_h(X) - (K_X^2) - 4.
\)
\end{NB}

We put
\begin{equation*}
  \zeta = \frac{1-3\phi^4}{2\phi^4}.
\end{equation*}
\begin{NB}
  And $\phi^4 = \nicefrac1{2\zeta+3}$.
\end{NB}
Then the above is equal to
\begin{equation}\label{eq:t}
    \begin{aligned}[t]
  & - (-1)^{\frac{(\xi,\xi+K_X)-(K_X^2) - (K_X,\tilde\xi_1)}{2}}
  2^{(K_X^2)-\chi_h(X)-1}
%  \frac{d\zeta}{\zeta+1}
  d\zeta
  \sum_{i,j,k,l} 
  (2\zeta+3)^{\{((\xi - K_X)^2) - (K_X^2) + 5\chi_h(X) - k - 2l\}/4-1}
\\
  & \times
  (-1)^{k} \Lambda^{k}
  \binom{(\xi - K_X, \tilde\xi_1)}{i}
  \binom{k}{j}
     (\tilde\xi_1, \alpha)^j
     (\xi - K_X, \alpha)^{k-j}
\\
  & \times
    \left(
      \zeta + 1
    \right)^{((\xi - K_X, \tilde\xi_1)-i+k-j)/2-1}
    \zeta^{(i+j + (K_X^2) - \chi_h(X) + 2)/2}
% \\
%   & \times
  \left(- {\Lambda^2}\left(\zeta + 3\right) x
        - \frac12\Lambda^2 (\alpha^2)z^2\right)^{l}\frac1{l!}
     \frac{z^k}{k!},
  \end{aligned}
\end{equation}

In order to illustrate the idea of the proof, let us first consider
the simplest nontrivial case $(K_X^2) - \chi_h(X) = -5$. 
(The case $(K_X^2) - \chi_h(X) = -4$ is too simple.)
\begin{NB}
  Let us first consider the simplest nontrivial case $(K_X^2) -
  \chi_h(X) = -4$. Then we only need to consider $i=0$, $j=0$, when we
  take the residue at $\zeta = 0$ by \eqref{eq:bound}. Therefore, up
  to constant, the above is equal to
\begin{equation*}
  (-1)^{(K_X,K_X+\tilde\xi_1)/2}
  \sum_{k,l}
  3^{((\xi - K_X)^2) - (K_X^2) + 5\chi_h(X) - k - 2l\}/4}
  (-\Lambda)^k
  (\xi - K_X, \alpha)^{k}
  \left( - 3{\Lambda^2} x
        - \frac12\Lambda^2 (\alpha^2)z^2\right)^{l}\frac1{l!}
     \frac{z^k}{k!}.
\end{equation*}
Taking the sum over the Seiberg-Witten classes, we get
\begin{equation*}
  \begin{split}
  & \sum_{\tilde\xi_1} \SW(\tilde\xi_1) \res_{\phi^4=1/3}
  \mathcal B^{(\dim M_H(y))} (\tilde\xi_1,\xi)
\\
=\; &
\begin{aligned}[t]
  &
\sum_{\tilde\xi_1} (-1)^{(K_X,K_X+\tilde\xi_1)/2} \SW(\tilde\xi_1) 
\\
  &\times
  \sum_{k,l}
  3^{((\xi - K_X)^2) - (K_X^2) + 5\chi_h(X) - k - 2l\}/4}
  (-\Lambda)^k
  (\xi - K_X, \alpha)^k
  \left(- 3{\Lambda^2} x
        - \frac12\Lambda^2 (\alpha^2)z^2\right)^{l}\frac1{l!}
     \frac{z^k}{k!}.
\end{aligned}
  \end{split}
\end{equation*}
up to a nonzero constant independent of $\xi$, $x$, $z$, $\Lambda$.

If this is independent of $t$ (recall $\xi = K_X + t(\xi^\circ-K_X)$),
we must have
\begin{equation*}
  \sum_{\tilde\xi_1} (-1)^{(K_X,K_X+\tilde\xi_1)/2} \SW(\tilde\xi_1) = 0.
\end{equation*}
This is the superconformal simple type condition when
$(K_X^2) - \chi_h(X) = -4$.
\end{NB}%
We only need to consider terms with $i+j = 1$, i.e., $i=1$, $j=0$ and
$i=0$, $j=1$ by \eqref{eq:bound}. Then, up to a constant, the residue
of \eqref{eq:t} is
\begin{equation*}
  \begin{split}
  & (-1)^{(K_X,K_X+\tilde\xi_1)/2}
  \sum_{k,l}
  3^{((\xi - K_X)^2) - (K_X^2) + 5\chi_h(X) - k - 2l\}/4}
  (-\Lambda)^k
  \left(- 3{\Lambda^2} x
        - \frac12\Lambda^2 (\alpha^2)z^2\right)^{l}\frac1{l!}
     \frac{z^k}{k!}
\\
  &\times\left(
  (\xi - K_X, \tilde\xi_1)
  (\xi - K_X, \alpha)^k
  + k(\tilde\xi_1,\alpha) (\xi - K_X, \alpha)^{k-1}
  \right).
  \end{split}
\end{equation*}
Since $\Lambda$, $x$, $z$ are formal variables, each term for
individual $k$, $l$ must be independent of $t$.
Since $(\xi - K_X, \tilde\xi_1)(\xi-K_X,\alpha)^k$ and
$k(\tilde\xi_1,\alpha) (\xi - K_X, \alpha)^{k-1}$ have different
degree in $t$ (the former has degree $k+1$, the latter has $k-1$),
they cannot cancel out. Therefore we must have
\begin{equation*}
  \sum_{\tilde\xi_1} (-1)^{(K_X,K_X+\tilde\xi_1)/2} (\tilde\xi_1,\alpha)
  \SW(\tilde\xi_1) = 0.
\end{equation*}
This is the superconformal simple type condition when
$(K_X^2) - \chi_h(X) = -5$.

\begin{NB}
More generally, when we compute the residue, the coefficients of
$(\xi-K_X,\alpha)^{k-j}$ vanish for all $k$, $j$. So we consider the
term with $j = \chi_h(X)-(K_X^2) -4$ (and hence $i=0$). Then the
residue is just an evaluation, and we conclude the vanishing
\eqref{eq:scs} for $n = \chi_h(X)-(K_X^2) -4$ as above.
\begin{NB2}
  This is the vanishing of the deepest coefficient.
\end{NB2}%

Next we consider the case $j= \chi_h(X)-(K_X^2) - 5$. Then we have
either $i=0$ or $1$. If $i=1$, then we get an expression
\(
   \binom{(\xi-K_X,\tilde\xi_1)}1 = (\xi-K_X,\tilde\xi_1),
\)
and hence the contribution vanishes by \eqref{eq:scs} with $n =
\chi_h(X)-(K_X^2) -4$. For the term with $i=0$, we need to expand
$(2\zeta+3)^{\heartsuit}(\zeta+1)^{\spadesuit}$ to compute the residue as
\begin{equation*}
  \begin{split}
    (\zeta+1)^{\spadesuit} &= 1 + \spadesuit\zeta + \cdots,\qquad
    \spadesuit = ((\xi-K_X,\tilde\xi_1)+k-j)/2 - 1
\\
    (2\zeta+3)^{\heartsuit} &= 3^\heartsuit + \heartsuit 3^{\heartsuit-1} 2\zeta + \cdots\qquad
    \heartsuit = \{((\xi - K_X)^2) - (K_X^2) + 5\chi_h(X) - k - 2l\}/4-1.
  \end{split}
\end{equation*}
Thus we get $\spadesuit + 2\heartsuit 3^{\heartsuit-1}$ times
something nonvanishing, but the vanishing with $n = \chi_h(X)-(K_X^2)
- 4$ implies that the contribution of $(\xi-K_X,\tilde\xi_1)$ in
$\spadesuit$ disappears. But then since this must be independent of
$\xi$, we must have \eqref{eq:scs} with $n = \chi_h(X)-(K_X^2) -5$.

Next consider the case $j=\chi_h(X)-(K_X^2)-6$. Then we have $i=2$,
$=1$, $=0$. But $i=2$, $=1$ terms vanish by \eqref{eq:scs} with $n =
\chi_h(X)-(K_X^2) -4$ and $n = \chi_h(X)-(K_X^2) -5$. Now I believe
that we can see a pattern, and get \eqref{eq:scs} with $n =
\chi_h(X)-(K_X^2) - 6$. We then conclude \eqref{eq:scs} with
$n=0,\dots,\chi_h(X)-(K_X^2) - 4$.
\end{NB}

\begin{proof}[Proof of \propref{prop:conv}]
  By the same reason as in the special case $(K_X^2)-\chi_h(X) = -5$,
  each term for individual $k$, $l$ must be independent of $t$.
  \begin{NB}
    Since $\Lambda$ appears as $\Lambda^{k+2l}$, we can take terms
    with fixed $k+2l$. Then we consider
    \begin{equation*}
      \left.\left(\pd{}{\zeta}\right)^p
      (-\Lambda^2(\zeta+3)x - \Lambda^2(\alpha^2)z^2/2)^l\right|_{\zeta=0}
      = l(l-1)\cdots (l-p+1) (-\Lambda^2 x)^p 
      (-3\Lambda^2 x - \Lambda^2(\alpha^2)z^2/2)^{l-p}.
    \end{equation*}
So degree for $x$ is determined from $l$.
  \end{NB}

We expand terms in \eqref{eq:t} as
\begin{equation*}
  \begin{split}
& (2\zeta+3)^{\{((\xi - K_X)^2) - (K_X^2) + 5\chi_h(X) - k - 2l\}/4-1}
\left(
      \zeta + 1
    \right)^{((\xi - K_X, \tilde\xi_1)-i+k-j)/2-1}\\
=\; &
3^{\{((\xi - K_X)^2) - (K_X^2) + 5\chi_h(X) - k - 2l\}/4-1}\\
& \times \Biggl[
1+ \zeta 
\begin{aligned}[t]
 & \biggl\{ \frac{2}{3} \left(\frac14
\left(- (K_X^2) + 5\chi_h(X) - 2l\right)-1\right)
+\frac{-i-j}2-1+\frac{k}{3} \\
  & \qquad
+\frac{(\xi - K_X, \tilde\xi_1)}2+\frac{2}{3}((\xi - K_X)^2) \biggr\} +
\cdots
\Biggr].
\end{aligned}
  \end{split}
\end{equation*}
The coefficient of $\zeta^l$ in $[\ ]$ has
the leading term (as a polynomial in $k$)
\begin{equation}\label{eq:nonzero}
  \sum_{l_1+l_2=l} \frac{(-k)^{l_1}}{4^{l_1}}\frac{2^{l_1}}{3^{l_1}}\frac{1}{l_1!}
  \times \frac{k^{l_2}}{2^{l_2}}\frac{1}{l_2!} 
  = \frac{k^l}{3^l l!}
  \ne 0.
\end{equation}
\begin{NB}
  I corrected Kota's calculation. Jan. 10, 2010
\end{NB}
Moreover the coefficient of $\zeta^l$ is a polynomial of $(\xi - K_X,
\tilde\xi_1)$ whose degree is at most $l$.
\begin{NB}
  This will be used or not ?
\end{NB}%

When we multiply the above expression with
$\zeta^{(i+j+(K_X^2)-\chi_h(X)+2)/2}$ in \eqref{eq:t}, it contributes
to the residue at $\zeta=0$ only if
\[
   i+j=-2l+(\chi_h(X)-(K_X^2)-4).
\]
And the residue is a linear combination of
\begin{equation*}
(\xi-K_X,\tilde\xi_1)^{p+q}((K_X-\xi)^2)^r(\alpha,\tilde\xi_1)^j
(\xi-K_X,\alpha)^{k-j}
\end{equation*}
of various $p$, $q$, $r$, $j$. ($k$ is fixed, as we explained at the
beginning.)
Here $q$ and $r$ come from the above expansion, and $p$ appears when
we expand $\binom{(\xi-K_X,\tilde\xi_1)}{i}$. Therefore we have
\begin{equation}
\begin{split}
0 \leq p \leq i & \text{ and $i \ne 0$ implies $p \ne 0$}\\
q+r \leq l, &\;\; q,r \geq 0 
\end{split}
\end{equation}
Up to the factor $3^{-\{((\xi - K_X)^2)}$, each term is a polynomial
in $t$ with degree $m:=k-j+p+q+2r$.
We will consider each coefficient of $t^m$ whose sum over $p$, $q$, $r$,
$j$ and Seiberg-Witten classes $\tilde\xi_1$ must be $0$ by our assumption.

We set $j_{\max}\defeq \chi_h(X)-(K_X^2)-4$. Then $j \leq j_{\max}$
and the equality holds if and only if $i=l=0$. We will check the
superconformal simple type condition \eqref{eq:scs} by  descending
induction on $n$. Starting from $n=j_{\max}$, we check it $n = j_{\max} - 2$,
$j_{\max} - 4$ and so on.

We assume
\begin{equation*}
  k-j_{\max} \leq m \leq k, \qquad
  k-j_{\max}\equiv m \bmod 2.
\end{equation*}
We will be interested in $j+p+q$, which will appear as $n$ in
\eqref{eq:scs}. We first note that
\begin{equation}
  \label{eq:note1}
j+p+q \leq i+j+l=-l+(\chi_h(X)-(K_X^2)-4) \leq j_{\max}.
\end{equation}
The equality holds if and only if $l=q=r=0$ and $p=i=0$.
We next note that
\begin{equation}
  \label{eq:note2}
j+p+q=k-m+2(p+q+r) \geq k-m.
\end{equation}
The equality holds if and only if $p=q=r=0$ and $(i,j)=(0,k-m)$.
Thus each coefficient of $t^m$ is
\begin{equation}
  \label{eq:A}
A (\alpha,\tilde\xi_1)^{k-m}
+(\text{ higher order terms }),
\end{equation}
where $A \ne 0$ by \eqref{eq:nonzero} and
the higher order terms mean sum of monomials with
$j+p+q > k-m$. 

We now start the descending induction on $k-m$. Start with $k - m =
j_{\max}$. Then (\ref{eq:note1}, \ref{eq:note2}) imply $j=j_{\max}$ and
$p=q=r=0$ and $i=n=0$. Thus there are no higher order terms in the
above expression, and we get the superconformal simple type condition
\eqref{eq:scs} for $n=j_{\max}$.

If \eqref{eq:scs} is true for $n > k-m$, then the sum of higher order
terms in \eqref{eq:A} over $\tilde\xi_1$ vanishes. Hence we also get
\eqref{eq:scs} with $n=k-m$. This completes the proof.
\begin{NB}
For $p \equiv \chi_h(X)-(K_X^2) \mod 2$,
\begin{multline}
(\tilde\xi_1,x \alpha+y \beta)^p+
(-1)^{\chi_h(X)-(K_X^2)}(-\tilde\xi_1,x \alpha+y \beta)^p\\
=
2x^p(\tilde\xi_1,\alpha)^p+2\binom{p}{2}x^{p-2}y^2
(\tilde\xi_1,\alpha)^{p-2}(\tilde\xi_1,\beta)^2+...
\end{multline}
Apply this to $\beta=\eta-K_X$.
\end{NB}%
\end{proof}

%\bibliographystyle{myamsplain}
%\bibliography{nakajima,mybib}

\end{document}